\documentclass{article}
\usepackage[utf8]{inputenc}

\usepackage{amssymb}
\usepackage{amsthm}
\usepackage{amsmath,amscd}
\usepackage[mathscr]{euscript}
\usepackage[all]{xy}
\usepackage{lmodern}
\usepackage[T1]{fontenc}
\usepackage[textwidth=14cm,hcentering]{geometry}
\usepackage[colorlinks=true,linkcolor=red,citecolor=blue]{hyperref}
\usepackage{cleveref}
\usepackage{mathtools}
\usepackage{tikz}
\usetikzlibrary{cd}
\usetikzlibrary{arrows,positioning}
\crefname{ex}{Example}{Examples}

\title{Dualizable presentable $\infty$-categories}
\author{Maxime Ramzi}
\date{}

\newtheorem{thm}{Theorem}[section]
\newtheorem{lm}[thm]{Lemma}
\newtheorem{prop}[thm]{Proposition}
\newtheorem{cor}[thm]{Corollary}

\newtheorem{add}[thm]{Addendum}
\newtheorem*{thm*}{Theorem}

\theoremstyle{definition}
\newtheorem{defn}[thm]{Definition}
\newtheorem{cons}[thm]{Construction}
\newtheorem{assu}[thm]{Assumption}
\newtheorem{nota}[thm]{Notation}

\newtheorem{ex}[thm]{Example}
\newtheorem{exs}[thm]{Examples}
\newtheorem{exc}[thm]{Exercise}
\newtheorem{rmk}[thm]{Remark}
\newtheorem{ques}[thm]{Question}
\newtheorem{conj}[thm]{Conjecture}
\newtheorem{warn}[thm]{Warning}
\newtheorem{obs}[thm]{Observation}
\AtBeginEnvironment{defn}{\pushQED{\qed}}
\AtEndEnvironment{defn}{\popQED\enddefn}
\AtBeginEnvironment{cons}{\pushQED{\qed}}
\AtEndEnvironment{cons}{\popQED\endcons}
\AtBeginEnvironment{assu}{\pushQED{\qed}}
\AtEndEnvironment{assu}{\popQED\endassu}
\AtBeginEnvironment{nota}{\pushQED{\qed}}
\AtEndEnvironment{nota}{\popQED\endnota}
\AtBeginEnvironment{conv}{\pushQED{\qed}}
\AtEndEnvironment{conv}{\popQED\enddefn}\AtBeginEnvironment{ex}{\pushQED{\qed}}
\AtEndEnvironment{ex}{\popQED\endex}
\AtBeginEnvironment{exs}{\pushQED{\qed}}
\AtEndEnvironment{exs}{\popQED\endexs}
\AtBeginEnvironment{exc}{\pushQED{\qed}}
\AtEndEnvironment{exc}{\popQED\endexc}
\AtBeginEnvironment{rmk}{\pushQED{\qed}}
\AtEndEnvironment{rmk}{\popQED\endrmk}
\AtBeginEnvironment{ques}{\pushQED{\qed}}
\AtEndEnvironment{ques}{\popQED\endques}
\AtBeginEnvironment{conj}{\pushQED{\qed}}
\AtEndEnvironment{conj}{\popQED\endconj}
\AtBeginEnvironment{warn}{\pushQED{\qed}}
\AtEndEnvironment{warn}{\popQED\endwarn}
\AtBeginEnvironment{obs}{\pushQED{\qed}}
\AtEndEnvironment{obs}{\popQED\endobs}

\crefname{ques}{Question}{Questions}

\newtheorem{thmx}{Theorem}

\newcommand{\op}{^{\mathrm{op}}}
\newcommand{\cat}{\mathbf}
\newcommand{\Cat}{\cat{Cat}}
\newcommand{\on}{\operatorname}
\newcommand{\id}{\mathrm{id}}
\newcommand{\Fun}{\on{Fun}}
\newcommand{\map}{\on{map}}
\newcommand{\Map}{\on{Map}}
\newcommand{\Hom}{\on{Hom}}
\newcommand{\NN}{\mathbb N}

\newcommand{\Sph}{\mathbb S}
\newcommand{\CMon}{\mathrm{CMon}}
\newcommand{\CGrp}{\mathrm{CGrp}}
\newcommand{\Set}{\cat{Set}}
\newcommand{\Ab}{\cat{Ab}}
\newcommand{\Ss}{\cat S}
\newcommand{\coAlg}{\mathrm{coAlg}}
\newcommand{\Sp}{\cat{Sp}}

\newcommand{\Oo}{\mathcal O}

\newcommand{\PrL}{\cat{Pr}^\mathrm{L} }

\newcommand{\Alg}{\mathrm{Alg}}
\newcommand{\CAlg}{\mathrm{CAlg}}
\newcommand{\LMod}{\cat{LMod}}
\newcommand{\RMod}{\cat{RMod}}
\newcommand{\Mod}{\cat{Mod}}
\newcommand{\coMod}{\cat{coMod}}

\newcommand{\THH}{\mathrm{THH}}

\newcommand{\perf}{\mathrm{perf}}

\newcommand{\Ind}{\mathrm{Ind}}

\newcommand{\pt}{\mathrm{pt}}
\newcommand{\colim}{\mathrm{colim}}
\newcommand{\fib}{\mathrm{fib}}
\newcommand{\Psh}{\cat{Psh}}
\newcommand{\Sh}{\cat{Sh}}

\newcommand{\ca}{\mathrm{ca}}

\newcommand{\dbl}{{\mathrm{dbl}}}
\newcommand{\at}{{\mathrm{at}}}
\newcommand{\one}{\mathbf{1}}

\newcommand{\QCoh}{\mathrm{QCoh}}
\newcommand{\V}{\mathcal V}

\newcommand{\CompAss}{\mathrm{CompAss}}
\newcommand{\st}{\mathrm{st}}
\newcommand{\Dbl}[1]{\Mod(#1)^\dbl}
\newcommand{\Prdbl}{(\PrL_{\st})^\dbl}
\newcommand{\At}[1]{\Mod(#1)^\at}
\newcommand{\M}{\mathcal M}
\newcommand{\N}{\mathcal N}
\newcommand{\D}{\mathcal D}
\newcommand{\C}{\mathcal C}
\newcommand{\E}{\mathcal E}
\newcommand{\PP}{\mathcal P}
\newcommand{\Rig}{\mathrm{Rig}}

\newcommand{\Loc}{\mathrm{Loc}}

\newcommand{\Calk}{\mathrm{Calk}}
\newcommand{\Pro}{\mathrm{Pro}}

\DeclareFontFamily{U}{min}{}
\DeclareFontShape{U}{min}{m}{n}{<-> udmj30}{}

\newcommand{\category}{$\infty$-category}
\newcommand{\categories}{$\infty$-categories}

\begin{document}
\maketitle
\begin{abstract}
    We prove that for any presentably symmetric monoidal $\infty$-category $\mathcal{V}$, the $\infty$-category $\mathbf{Mod}_\mathcal{V}(\mathbf{Pr}^{\mathrm{L}})^{\mathrm{dbl}}$ of dualizable presentable $\V$-modules and internal left adjoints between them is itself presentable. 

Along the way, we survey formal properties of these dualizable $\mathcal V$-modules. We pay close attention to the case of the $\infty$-category of spectra, where we survey the foundational properties of ``compact morphisms''.
\end{abstract}
\tableofcontents
\setcounter{secnumdepth}{0}
\section{Introduction}
\subsection{Motivation}
\subsubsection{Continuous $K$-theory}
Algebraic $K$-theory is a fundamental invariant in algebra and geometry. Originally defined for rings by Grothendieck, the type of inputs it allows has been generalized several times over. The fundamental work of Blumberg--Gepner--Tabuada \cite{BGT} shows that in the setting of stable \categories, $K$-theory (in both its connective and nonconnective variants) has a very natural universal property, characterized, say, as a universal localizing invariant on small stable \categories{} (for the nonconnective variant). 

The recent work of Efimov \cite{Hoyois,youtubeDustin1} has shown that there is a natural extension of the scope of localizing invariants to so-called dualizable presentable stable \categories{} (henceforth, dualizable categories), or equivalently compactly assembled stable \categories. 

There are many interesting examples of such \categories{} which are not compactly generated (the latter correspond, under this extension, to the small stable \categories{} from the work of \cite{BGT}), such as categories of sheaves on locally compact Hausdorff spaces, categories of nuclear modules arising in condensed mathematics, categories coming from almost mathematics, or, finally, following the recent disproof of the telescope conjecture \cite{NoTel}, categories of $L_n$-acyclic spectra. Furthermore, this framework allows to phrase certain classical $K$-theoretic phenomena more naturally, such as \cite[Theorem 18]{tamme}, recovered as \cite[Corollary 13]{Hoyois}. 
\subsubsection{Brauer groups}
The strongest dualizability assumption one can make on an object is probably invertibility. Invertible categories are closely related to the theory of Azumaya algebras and Brauer groups - indeed, given a commutative ring spectrum $R$, an algebra $A$ is Azumaya in the sense of \cite{toen,BRS} if and only if the \category{} $\Mod_A(\Sp)$ of $A$-module spectra is invertible as an $R$-linear presentable \category{}. 

However, there could \emph{a priori} be more invertible categories - for example, if they are non-compactly generated, cf. the introduction to \cite[Section 3]{toen}. One of the original goals of this project was to answer the following question, posed to us by Lior Yanovski: 
\begin{ques}\label{ques:Brsmall}
    Given a commutative ring spectrum $R$, is the categorical Brauer group $\mathbf{Br}^{\mathrm{cat}}(R) := \mathbf{Pic}(\PrL_R)$ a small group ? More generally, replacing $\Mod_R$ with $\V$, for some $\V\in\CAlg(\PrL)$ ? 
\end{ques}
If all invertible categories are compactly generated, the answer is clearly yes, but we were not able to prove that they were. To answer this question, we embarked on trying to understand the formal properties of dualizable categories. Our main result answers this question in the positive, see \Cref{rmk:Bryes}.

\subsection{Outline}
As explained earlier, there has been a lot of recent interest in dualizable stable \categories. The goal of this paper is to survey some of the formal, or foundational properties of the \category{} of dualizable categories, as well as related structures and notions.  

Both because one might be interested in non-absolute localizing invariants, and because one might be interested in categorical Brauer groups and variants thereof, we have decided to study dualizable categories in the generality of a base $\V\in\CAlg(\PrL)$\footnote{Note that we do not require $\V$ to be stable.}. In the case where $\V=\Sp$, or where $\V$ is a \emph{rigid} $\Sp$-algebra, many of the results we present here are due to the current emerging folklore surrounding this notion, particularly as pioneered by Clausen-Scholze, Efimov \cite{sashaI}, and Krause-Nikolaus-Pützstück \cite{KNSh}. 

Besides extending much of the ideas and results in this folklore to an arbitrary ``enriching'' base $\V$, our main original contribution is the following answer to \Cref{ques:Brsmall} and to another question by Yanovski:
\begin{thmx}\label{thm:main}
    Let $\V$ be a presentably symmetric monoidal \category. The \category{} $\Dbl{\V}$ of dualizable presentable $\V$-modules and \emph{internal left adjoints} between them is itself presentable\footnote{It is crucial to restrict the morphisms, otherwise the morphism spaces are not even small!}. 
    
Furthermore, if $\V\in\CAlg(\PrL_\kappa)$, i.e. $\V$ is $\kappa$-compactly generated, its unit is $\kappa$-compact, and the tensor product preserves $\kappa$-compacts, for some \emph{uncountable} regular cardinal $\kappa$, then $\Dbl{\V}$ is $\kappa$-compactly generated. 
\end{thmx}
\begin{rmk}\label{rmk:Bryes}
    From this, a positive answer to \Cref{ques:Brsmall} follows immediately, as the Picard group of any presentably symmetric monoidal \category{} is small. In fact, a weaker theorem suffices to answer this positively, namely it suffices to prove that there is a bound on the accessibility rank of dualizable $\V$-modules, cf. \Cref{thm:CardBound}. 
\end{rmk}
In fact, we deduce this theorem from a stronger theorem, namely:
\begin{thmx}
    Let $\V\in\CAlg(\PrL_\kappa)$ for some uncountable regular cardinal $\kappa$. There is a (non-full) inclusion $\Dbl{\V}\subset\Mod_\V(\PrL_\kappa)$, and this inclusion is comonadic, for an accessible comonad on $\Mod_\V(\PrL_\kappa)$. 
\end{thmx}

\begin{rmk}
    The crucial point of this theorem, namely the one after which one starts believing \Cref{thm:main}, is the first part, namely the fact that every dualizable $\V$-module is $\kappa$-presentable. In fact, I learned from Thomas Blom\footnote{Private communication, work in progress.} that the fact that this inclusion is comonadic can actually be obtained as a formal consequence from the theory of ``lax idempotent comonads''.
\end{rmk}

In turn, this theorem is proved by obtaining an understanding of dualizable $\V$-modules and their relationship to the ones we call \emph{atomically generated}. More specifically, we obtain a $\V$-linear analogue of Lurie's characterization of dualizable objects in $\PrL_\st$, see \Cref{thm:Luriethmgeneralbase}. 

We furthermore use this to \emph{deduce} the presentability of the \category{} of atomically generated $\V$-modules. Interpreting these as ``Cauchy-complete $\V$-enriched \categories'' (in $1$-category theory, they are equivalent to enriched $\V$-categories having all absolute $\V$-colimits), we can see this as a proof that these Cauchy complete $\V$-categories form a presentable \category. 

Besides proving these theorems, we spend a lot of time surveying basic properties of these categories, as well as of unstable compactly assembled categories, which are not necessarily needed for proving these results. 

Among other things, we study the notion of \emph{compact morphisms}, which allow for an intrinsic study of compactly assembled categories, quite analogous to that of compactly generated categories. We use them for example to describe limits and internal homs in $\Prdbl$.

There is a companion to this paper, namely \cite{companion}, which is meant to study a multiplicative version of the story of dualizable \categories, that of rigid (or more generally locally rigid) \categories. We also introduce there the notion of \emph{$\V$-atomic morphisms} in $\V$-modules, which are meant to be analogues of compact morphisms, and are also supposed to allow for an intrinsic study of dualizable $\V$-modules, though in a more restricted fashion.  

\begin{rmk}
    Since a draft of this paper first appeared on my webpage, Efimov released the preprint \cite{sashaI}, in which he provides a different proof of \Cref{thm:main} in the case $\V=\Sp$\footnote{This easily implies the case where $\V$ is rigid over $\Sp$, though not necessarily the general stable case, to my knowledge.}. There, he explicitly studies $\kappa$-compactness in $\Prdbl$ and produces an explicit $\omega_1$-compact generator. 
\end{rmk}
 
\subsection*{Perspectives}
In this paper, we work in the generality of a ``base $\V\in\CAlg(\PrL)$'', and we study the \category{} $\Mod_\V(\PrL)$, or perhaps more honestly, the corresponding $(\infty,2)$-category. It is probably worth wondering how much of our work extends to more general ``doubly presentable'' $(\infty,2)$-categories, of which $\Mod_\V(\PrL)$ should be the most basic example. An example which, in general, is not covered by this example, is $\PrL_{\mathcal X}$ where $\mathcal X$ is some $\infty$-topos, in the sense of \cite{MartiniWolf}. 

There is always a fully faithful, symmetric monoidal embedding\footnote{See \cite[Section 8.3]{MartiniWolf}.} $\Mod_\mathcal X(\PrL)\to \PrL_\mathcal X$ which compares a naive theory of ``$\mathcal X$-parametrized presentable \categories'' to a less naive one, and when $\mathcal X$ is $0$-localic (e.g. if it is sheaves on a topological space), this is an equivalence. 

It sounds likely that our methods could extend to $\PrL_\mathcal X$, but at this point it is unclear in which generality exactly they do. We thus raise the following as a meta-conjecture, since one of the terms appearing in it is not well-defined yet:
\begin{conj}
    Let $\mathcal P$ be a ``doubly-presentably'' symmetric monoidal $(\infty,2)$-category\footnote{$\Mod_\V(\PrL)$ and $\PrL_\mathcal X$, as above, are supposed to be examples of such.}. The $(\infty,1)$-category $\mathcal P^\dbl$ of dualizable objects and $\mathcal P$-internal left adjoints therein is presentable.
\end{conj}
One can make a similar conjecture for the category of rigid algebras in $\mathcal P$, cf. \cite[Theorem A]{companion}. Part of the conjecture (and maybe the most important part of the conjecture) is the existence of a reasonable notion of `` doubly-presentably symmetric monoidal $(\infty,2)$-~category''\footnote{Closely related morally, but maybe not technically, is the existence of a reasonable notion of $2$-stable $(\infty,2)$-category, of which $\Cat^\perf$ and $\PrL_\st$ are supposed to be the prime examples.}.

More generally, the results and some of the methods from both the present paper and the companion paper \cite{companion} raise this question to a point where it becomes increasingly awkward not to answer it. 

Surrounding \Cref{ques:Brsmall}, we were only able to answer the smallness question, and not the more pressing questions such as ``Are there any non-compactly generated invertible $R$-linear categories ?''. In \cite{Stefanich}, Stefanich proves that there are none for $R=k$ a (discrete) field, and manages to deduce the same result over slightly more general base, but the generality remains restricted. 

In fact, there are not even known counterexamples to the following conjecture: 
\begin{conj}
    Let $R$ be a commutative ring spectrum, and $C$ a dualizable $R$-linear presentable stable \category. If the coevaluation morphism $\Mod_R\to C\otimes_R C^\vee$ is an internal left adjoint, then $C$ is compactly generated. 
\end{conj}
A weaker conjecture would ask for both the coevaluation and the evaluation morphism to be internal left adjoints, and finally, the weakest of the three conjectures asks for invertibility. We have made no progress on these questions, and thus leave them for future work (of course, they also have analogues over more general bases than $\Mod_R$, using the notion of ``atomically generated'' in place of compactly generated).

In this paper, we study the basics of an emerging higher algebra of dualizable categories, but we only scratch the surface, and as pointed out above, our knowledge regarding basic-sounding questions is lacking: much remains to be done. 

We also do not go very deep in the direction of Efimov's applications to $K$-theory, which involve much more technical results about (among other things) dualizable internal homs and dualizable limits. 
 
\subsection{Sectionwise outline}
We now outline the contents of the paper, section by section. 
\begin{itemize}
    \item In \Cref{section:prelim}, we set the stage by describing our basic objects of study: (dualizable) $\V$-modules, internal left adjoints, and compact assembly. The properties that are proved there will be used throughout. 
    \item In \Cref{section:atomic}, we introduce compact morphisms. We spend quite a bit of time studying compact morphisms and their variants, stably, and unstably, and we conclude the compact map part by describing the foundation of Efimov's work on localizing invariants. 
    \item \Cref{section:comonad} is where we prove our presentability theorem. 
    \item  In the final \Cref{section:limits}, we explain how to compute limits and internal homs in $\Prdbl$ - their existence follows from our presentability result, but it can also be established by hand. We also survey a number of elementary examples. 
\end{itemize}
\begin{rmk}
    As explained in the introduction, \Cref{thm:main} can be viewed as a pretext for me to write about dualizable \categories{} and survey the landscape. In particular, the paper is not specifically geared towards this, and a geodesic approach to this result might look different. Specifically, one needs to understand \Cref{thm:Luriethmgeneralbase} and  \Cref{prop:cardboundunstable} (for which a few, but not very many preliminaries about $\kappa$-compact maps are needed) and then one can jump to \Cref{section:comonad}.
\end{rmk}

We have three appendices. In \Cref{app:stable}, we review basic properties of stable \categories{} which will be needed when working out some examples and properties of the theory for $\V=\Sp$; in \Cref{app:invariants}, we review the basics of splitting or additive invariants, as a prerequisite for our brief discussion of Efimov's continuous $K$-theory; and finally in \Cref{app:milnor}, we review an unstable version of the Milnor exact sequence, which we use when comparing various versions of compact maps in the unstable setting.

\subsection{Conventions and notations}
We work throughout with the theory of \categories{} as developped thoroughly by Lurie in \cite{HTT,HA}, although we try to take a mostly model-independent approach. 
\begin{itemize}
\item Unless there is a specific need for it, we do not adress set-theoretic issues. $\widehat{\Cat}$ denotes the \category{} of large \categories.
\item We use $\Ss$ for the \category{} of spaces/anima and $\Sp$ for the \category{} of spectra. 
    \item We use $\PrL$ to denote the (very large) \category{} of presentable \categories, and $\PrL_{\st}$ the full subcategory spanned by the stable ones; we often view it as a symmetric monoidal \category{} through the Lurie tensor product - this is the only one we consider, so we keep it implicit throughout. We use $\Fun^L$ to denote the \category{} of colimit-preserving functors. For a regular cardinal $\kappa$, $\PrL_\kappa$ denotes the non-full subcategory of $\PrL$ spanned by $\kappa$-compactly generated \categories{} and $\kappa$-compact preserving left adjoints (equivalently, left adjoints whose right adjoint preserves $\kappa$-filtered colimits).
    \item For a symmetric monoidal \category{} $\M$, we let $\CAlg(\M)$ denote the \category{} of commutative algebras in $\M$, and for $A\in\CAlg(\M)$, we let $\Mod_A(\M)$ denote its \category{} of modules\footnote{Canonically, we mean $\mathbb E_\infty$-modules, although they are equivalent to left modules.}. For a non-necessarily commutative algebra $A$, we write $\LMod$ or $\RMod$ to indicate whether we consider left or right modules. 
    \item We use $\Map$ for mapping spaces, $\map$ for mapping spectra in stable \categories, and $\hom$ for internal homs, or enriched homs.
\end{itemize}

From now on, the word ``category'' means \category{} by default, and we write $1$-category when we want to single out ordinary categories.

\subsection{Acknowledgements}
Lior Yanovski is the person who led me to think about these questions, and specifically who asked me whether \Cref{thm:main} was true. This work would not have existed without him, and I thank him for suggesting I think about it, and for many discussions about internal left adjoints.

I wish to thank Dustin Clausen, Sasha Efimov, Achim Krause and Thomas Nikolaus for \emph{many} helpful discussions about the present subject. This work benefitted immensely from their help. 

I thank Robert Burklund for urging me to also write about (locally) rigid categories, leading to the writing of \cite{companion}, and for helpful conversations about basepoints; Shay Ben Moshe for some helpful feedback and Thomas Blom for interesting conversations related to monads and I thank Lars Hesselholt for feedback on an earlier draft of the proof of the main result. 

This work was supported by the Danish National Research Foundation through the
Copenhagen Centre for Geometry and Topology (DNRF151). Part of the results were obtained while I was visiting Thomas Nikolaus in Münster, and I thank both Thomas and Lisa for their hospitality, as well Münster Universität for its hospitality. Finally, I thank Harvard University and Mike Hopkins for their hospitality while I was in the finishing stages of writing this document. 

\setcounter{secnumdepth}{3}
\section{Preliminaries}\label{section:prelim}
In this section, we set the stage of the whole paper by introducing the basic objects: dualizable $\V$-modules, and $\V$-internal left adjoints. We set up their basic properties and their relations to other categories of $\V$-modules. We conclude by saying a few words about compact assembly. 

    Throughout this section, $\V$ will be a presentably symmetric monoidal \category, i.e. an object of $\CAlg(\PrL)$. Our main reference for enriched categories is \cite{heine}. 
    
    We use $\Cat_\V$ to denote the category of $\V$-enriched categories\footnote{$\widehat{\Cat}_\V$ if we want to allow large ones.}, $\Fun_\V$ to denote enriched functor (ordinary) categories, and $\Fun^L_\V$ to denote categories of functors in $\Mod_\V(\PrL)$, and $\PP_\V$ to denote $\V$-enriched presheaves. 
    
    We will often compare the general case to the specific case $\V=\Sp$, where a rich source of $\V$-enriched \categories{} are stable \categories{}. 
\subsection{Presentable $\V$-modules}
The goal of this subsection is to set up in the world of $\V$-modules the usual constructions and properties that we know from bare $\PrL$ or $\PrL_{\st}$. Most of what we prove is either contained in, or derived from \cite{heine,Shay,hinichday,hinichyoneda}. 
We start by recalling a result of Heine comparing $\V$-modules with certain enriched \categories: \
\begin{thm}[{\cite[Theorem 1.2]{heine}}]\label{thm: enrichedtensored}
    There is a canonical functor $$\chi: \Mod_\V(\PrL)\to \widehat{\Cat_\V}$$ which is fully faithful on hom-categories. Given two $\V$-modules $\M,\N$, the functor $$\Fun^L_\V(\M,\N)\to \Fun_\V(\chi(\M),\chi(\N))$$ has essential image those $\V$-functors that are left adjoints as $\V$-functors. 
\end{thm}
\begin{nota}
    Given $\M\in\Mod_\V(\PrL)$, we view $\chi(\M)$ as essentially the same object as $\M$, so we will abuse notation and denote both by the same symbol. 
\end{nota}
Furthermore, Heine and Hinich \cite{heine,hinichday} prove the expected universal property of presheaf $\V$-categories, namely: 
\begin{thm}[{\cite[Theorem 1.13]{heine}}]\label{thm: UPVPsh}
    Let $\M_0$ be a small $\V$-\category, and $\N$ a presentable $\V$-module. Restriction along the Yoneda embedding $\M_0\to \PP_\V(\M_0)$ induces an equivalence $$\Fun^L_\V(\PP_\V(\M_0), \N)\simeq \Fun_\V(\M_0,\N)$$
\end{thm}
\begin{ex}
    For $\V= \Ss$, $\Ss$-enrichment is no data, and so we find that $\PP_\Ss$ is simply the usual presheaf category. 
\end{ex}
\begin{ex}
    For $\V=\Sp$, a natural class of $\Sp$-enriched categories is that of stable categories. Between two stable categories, the category of $\Sp$-enriched functors is equivalent to  that of exact functors by \cite[Theorem 1.11]{heine}. In particular, in this case, we find that $\PP_\Sp(C)= \Ind(C)$ when $C$ is small stable and canonically $\Sp$-enriched. 
\end{ex}
\begin{ex}
Suppose $\V$ is such that the canonical functor $\Psh(\V^\dbl)\to \V$ is a localization, i.e. $\V$ is canonically generated by its dualizable objects. In this case, if $C$ is $\V$-enriched and $\V^\dbl$-tensored, $\PP_\V(C) = \Fun_\V(C\op,\V) = \Fun_{\Mod_{\V^\dbl}}(C\op,\V)$ where the latter is the category of $\V^\dbl$-linear functors. This can be deduced from Heine's work in \cite{heine}, and the fact that a $\V$-enriched functor between $\V^\dbl$-tensored categories is automatically $\V^\dbl$-linear. 
\end{ex}
\begin{cor}\label{cor:locVmod}\label{lm:ffrightadjointV}
Assume $\V$ is $\kappa$-compactly generated, and the tensor product of $\V$ preserves $\kappa$-compacts. 
    Let $\M\in \Mod_\V(\PrL_{\lambda})$ for some $\lambda\geq \kappa$, and let $\M^\lambda$ denote the full $\V$-subcategory of $\M$ spanned by the $\lambda$-compacts in $\M$. 

    The inclusion $\M^\lambda\to \M$ induces a functor $\PP_\V(\M^\lambda)\to \M$ by \Cref{thm: UPVPsh}. In this situation, this functor is a localization of $\V$-modules, i.e. its right adjoint is $\V$-fully faithful.
\end{cor}
\begin{proof}
    Let $y: \M^\lambda\to \PP_\V(\M^\lambda)$ be the Yoneda embedding, and let $W$ consist of the following maps in $\PP_\V(\M^\lambda)$: 
    \begin{enumerate}
        \item First, maps of the form $v\otimes y(m)\to y(v\otimes m)$ for $v\in\V^\lambda, m\in \M^\lambda$; 
        \item Second, maps of the form $\colim_J y(m_j)\to y(\colim_J m_j)$ for $\lambda$-small diagrams $m_\bullet:~J\to~\M^\lambda$. 
    \end{enumerate}
    Close $W$ under tensors by $\V^\lambda$, and let $\M'$ be the accessible localization of $\PP_\V(\M^\lambda)$ at $W$. This acquires a canonical $\V$-module structure, and the map $\PP_\V(\M^\lambda)\to \M$ clearly factors as a $\V$-module map through $\M'$. We claim that the map $\M'\to \M$ is an equivalence. 

    Using \Cref{thm: UPVPsh}, we find that for any $\V$-module $\N$, $$\Fun^L_\V(\M',\N)\simeq \Fun^L_{\V,W}(\PP_\V(\M^\lambda), \N)\simeq \Fun_{\V,W}(\M^\lambda, \N)$$ where the latter is the full subcategory of $\V$-enriched functors spanned by those functors that invert the maps in $W$. Now by assumption, $\M^\lambda$ is closed under tensors by $\V^\lambda$ and $\V$ is $\lambda$-compactly generated, so $\V$-enriched maps $\M^\lambda\to \N$ are equivalent to weakly $\V^\lambda$-tensored maps by \Cref{thm: enrichedtensored} (or rather its non-presentable variant, cf. \cite[Theorem 1.1]{heine}). The condition that $W$ be sent to equivalences, via the first half of maps in $W$, corresponds to the condition that the corresponding weakly $\V^\lambda$-tensored map be a map of $\V^\lambda$-modules; and via the second half of maps in $W$, to the condition that the underlying functor preserve $\lambda$-small colimits.

    Thus the claim now follows from the fact that $\Fun^{\lambda-\mathrm{rex}}_{\V^\lambda}(\M^\lambda,\N)\simeq \Fun^L_\V(\M,\N)$ (where $\Fun^{\lambda-\mathrm{rex}}$ is the subcategory of $\lambda$-small colimit preserving functors). 
\end{proof}

\begin{obs}\label{obs:naturalcolim}
    The inclusion $\M^\lambda\to \M$ is natural in $\M\in\Mod_\V(\PrL_\lambda)$, thus by \cite[Theorem A]{Shay}, so is the induced map $\PP_\V(\M^\lambda)\to \M$. 
\end{obs}
\subsection{Internal left adjoints}
We now introduce internal left adjoints, which will be the relevant morphisms between dualizable $\V$-modules. They are suppposed to be analogues/generalizations of ``compact-preserving functors between compactly generated categories''. Much of the material here can also be found in \cite{BMS} or \cite{Shay}, in either greater or smaller generality. We refer to those papers for the proofs we omit (though the reader is encouraged to work them out as exercises). 
\begin{defn}
    Let $f:\M\to \N$ be a functor in $\Mod_\V(\PrL)$, with right adjoint $f^R$. Following \cite[Example 7.3.2.8., Remark 7.3.2.9.]{HA}, if the canonical projection map $$v\otimes f^R(n)\to f^R(v\otimes n)$$ is an equivalence for all $v\in\V,n\in\N$, we find that $f^R$ induces a right adjoint $\V$-module map $\N\to \M$. 
    
If it is furthermore colimit-preserving, we call $f$ an \emph{internal left adjoint}. 
\end{defn}
\begin{rmk}
    If $f^R$ is colimit-preserving, one can rephrase the condition about the projection map as the statement that the following square is horizontally right adjointable\footnote{See \cite[Definition 3.1]{HSS} for the notion of adjointability.}: 
    \[\begin{tikzcd}
	\V\otimes\M & \V\otimes\N \\
	\M & \N
	\arrow[from=2-1, to=2-2]
	\arrow[from=1-1, to=1-2]
	\arrow[from=1-1, to=2-1]
	\arrow[from=1-2, to=2-2]
\end{tikzcd}\]
\end{rmk}
\begin{rmk}
    The terminology comes from the fact that this condition is equivalent to $f$ being a left adjoint \emph{internal to} the $(\infty,2)$-category $\Mod_\V(\PrL)$ (equivalently, internal to its homotopy $2$-category). 
\end{rmk}
\begin{obs}
    Internal left adjoints are closed under composition. 
\end{obs}
\begin{ex}
    If $\V=\Ss$ or $\Sp$, the condition that the projection map be an equivalence in fact \emph{follows} from the fact that $f^R$ preserves colimits. This is more generally the case whenever $\V$ is a \emph{mode}, i.e. an idempotent algebra in $\PrL$ \cite[Proposition 5.2.15]{AmbiHeight}. A different generalization is that this is the case whenever $\V$ is generated under colimits by its dualizable objects, as is spelled out in \cite[Corollary 3.8]{companion}
\end{ex}
\begin{ex}
    In the case $\V=\Sp$, if the source $\M$ is compactly generated, this is equivalent to $f$ preserving compact objects - in general, this condition only \emph{implies} that $f$ preserves compact objects but the converse need not hold. See also \Cref{cor:atomicimpliesinternal}.
\end{ex}
\begin{ex}
    Let $f:I\to J$ be a functor between small \categories, and let $f_!:~\V^I\to~\V^J$ be the left Kan extension functor. This is left adjoint to $f^*$, and it is easy to see that this implies that $f_!$ is an internal left adjoint. 

    It is rarer for $f^*$ to be an internal left adjoint: its right adjoint is $f_*$, right Kan extension, and so this condition boils down to the limit functors over $I\times_J J_{j/}$ all preserving colimits and $\V$-tensors - this is some kind of ``finiteness relative to $\V$'' condition. The situation where $I,J$ are spaces is studied in some detail in \cite{CCRY}. 
\end{ex}
\begin{ex}
    Let $f: A\to B$ be a map of algebras in $\V$. Basechange along $f$, as a functor $\LMod_A(\V)\to \LMod_B(\V)$ is an internal left adjoint. 
\end{ex}
\begin{ex}
    If $v\in\V$ is dualizable, then for any $\V$-module $\M$, tensoring with $v$ is an internal left adjoint $\M\to \M$. 
\end{ex}
\begin{ex}
    In the case where $\V$ is stable, consider a localization sequence $$\M\to \N\to \N/\M$$ where $\M\to \N$ is fully faithful. In this case, the inclusion functor $i:\M\to \N$ is an internal left adjoint if and only if the projection functor $p:\N\to \N/\M$ is. This follows from the co/fiber sequence $ii^R\to \id_\N\to p^R p$ of endofunctors of $\N$ (see, e.g. \Cref{lm:fundamentalsequence}).
\end{ex}
\begin{ex}
    Let $\M\xrightarrow{f} \N_0\overset{i}{\hookrightarrow} \N$ be a composite in $\Mod_\V(\PrL)$ where $\N_0\to \N$ is fully faithful. If $i\circ f$ is an internal left adjoint, then so is $f$. Indeed, the right adjoint to $f$ can be written as $(i\circ f)^R \circ i$, as $i^Ri\simeq \id_{\N_0}$. 
\end{ex}
\begin{nota}
   Let $\M,\N\in\Mod_\V(\PrL)$. We let $\Fun^{iL}_\V(\M,\N) \subset \Fun^L_\V(\M,\N)$ denote the full subcategory spanned by internal left adjoints. 
\end{nota}
\begin{lm}\label{lm:smallinternalleft}
    Let $\M,\N\in\Mod_\V(\PrL)$. The \category{} $\Fun^{iL}_\V(\M,\N)$ of internal left adjoints from $\M$ to $\N$ is essentially small.
\end{lm}
\begin{proof}
    Recall that $\Fun^L_\V(\M,\N)$ is presentable. In particular, for $\M=\N$, we find that $\id_\M$ is $\kappa$-compact for some $\kappa$. 

    If $f$ is an internal left adjoint, we find that for $g\in\Fun^L_\V(\M,\N)$, $\Map(f,g)\simeq \Map(\id_\M, f^Rg) $, so that $f$ is also $\kappa$-compact in $\Fun^L_\V(\M,\N)$. We therefore have $\Fun^{iL}_\V(\M,\N)\subset \Fun^L_\V(\M,\N)^\kappa$, and the latter is essentially small. 
\end{proof}

A crucial special case of internal left adjoints is when $\M=\V$. The following definition appeared in \cite{Shay,Sep}
\begin{defn}\label{defn:atomicob}
    An object $x\in \M$ is called $\V$-atomic, or simply atomic if the base $\V$ is understood, if the $\V$-linear functor $\V\xrightarrow{-\otimes x} \M$ it classifies\footnote{Under the equivalence $\Fun^L_\V(\V,\M)\simeq \M$.} is an internal left adjoint, i.e. if $\hom_\M(x,-): \M\to \V$ preserves colimits and the canonical map $v\otimes~\hom_\M(x,y)\to~\hom_\M(x,v\otimes~y)$ is an equivalence for all $v\in \V, y\in\M$. 
\end{defn}
\begin{ex}\label{ex:VatVdbl}
    $\V$-atomic objects in $\V$ are exactly the dualizable objects. It follows by considering the composite $\V\xrightarrow{-\otimes v}\V\xrightarrow{-\otimes m}\M$ that if $m\in\M$ is $\V$-atomic and $v\in\V$ is dualizable, then $v\otimes m$ is $\V$-atomic.
\end{ex}
\begin{ex}
    $\Sp$-atomic objects are exactly compact objects in the usual sense. More generally, if $\V$ is a mode, an object $m$ in $\M$ is atomic if and only if $\hom(m,-):\M\to\V$ preserves colimits. This case was studied in \cite{BMS}. 
\end{ex}
\begin{nota}
    Let $\M\in\Mod_\V(\PrL)$. We let $\M^\at\subset \M$ denote the full subcategory spanned by the atomic objects.
\end{nota}
\begin{cor}\label{cor:atimplcpt}
    Let $\M$ be a $\V$-module. The \category{} $\M^\at$ of atomic objects of $\M$ is essentially small. In fact, if the unit $\one_\V$ is $\kappa$-compact, we have $\M^\at\subset \M^\kappa$. 
\end{cor}
\begin{proof}
    This follows from \Cref{lm:smallinternalleft} together with the definitional equivalence $\Fun^{iL}_\V(\V,\M)\simeq~\M^\at$. 

    The second part follows from the proof of \Cref{lm:smallinternalleft} together with the fact that under the equivalence $\Fun^L_\V(\V,\V)\simeq \V$, the identity $\id_\V$ corresponds to $\one_\V$. 
\end{proof}
\begin{defn}
   Let $\M\in\Mod_\V(\PrL)$. $\M$ is said to be $\V$-atomically generated if the smallest full $\V$-submodule of $\M$ closed under colimits and containing the atomics of $\M$ is $\M$ itself. 
\end{defn}
\begin{obs}\label{obs:descatgen}
    The tensoring of $\V$ on $\M$ restricts to a weak tensoring of $\V$ on $\M^\at$. By \cite{heine}, the canonical map of weakly tensored \categories{} $\M^\at\to \M$ induces a similar map of $\V$-enriched \categories{} between objects with the same name, and then by \Cref{thm: UPVPsh} a $\V$-linear, colimit preserving functor $i:\PP_\V(\M^\at)\to \M$. 

 For any given $y$, the collection of $x$'s such that the map $$\hom_{\PP_\V(\M^\at)}(x,y)\to \hom_\M(i(x),i(y))$$ is an equivalence is closed under colimits and $\V$-tensors. For $x\in \M^\at$ (viewed as an object of $\PP_\V(\M^\at)$ under the Yoneda embedding), the collection of $y$'s for which this is an equivalence is similarly closed under colimits and $\V$-tensors, and by construction (together with the Yoneda lemma, cf.\cite{hinichyoneda}) it contains the image of $\M^\at$. Putting these two together, we find that the functor $\PP_\V(\M^\at)\to \M$ is a fully faithful $\V$-linear left adjoint. It follows that if $\M$ is atomically generated, this functor is an equivalence. 

 We now observe that its inverse is explicit : it is given by the restricted Yoneda embedding.

 From this construction and this result, we deduce that $\M$ is atomically generated if and only if it is of the form $\PP_\V(\M_0)$ for some $\V$-enriched \category{} $\M_0$. 
\end{obs}

\begin{cor}\label{cor:atimpliescpct}
    Let $\V\in\CAlg(\PrL_\kappa)$, and let $\M$ be an atomically generated $\V$-module (thus of the form $\PP_\V(\M_0)$ for some small $\V$-\category{} $\M_0$). 

    In this case, $\M$ is $\kappa$-compactly generated, and $\kappa$-compacts in $\M$ are closed under tensoring with $\kappa$-compacts in $\V$. 
\end{cor}
\begin{proof}
    For each $v\in\V^\kappa$ and $m\in\M^\at$, $v\otimes m$ is $\kappa$-compact by direct inspection and these objects generate $\M$ under \emph{colimits} by definition. 

    As these objects therefore generate $\kappa$-compacts in $\M$ under $\kappa$-small colimits, and as $\V^\kappa$ is closed under tensor products by assumption, we find that tensoring with $\V^\kappa$ preserves $\M^\kappa$.
\end{proof}
\begin{lm}\label{lm:colimdetection}
Let $F: \M\to \N$ be a map in $\Mod_\V(\PrL)$; and let $\iota_i: \M_i\to \M$ be a family of internal left adjoints such that the right adjoints $(\iota_i)^R : \M\to \M_i$ are jointly conservative. 

In this case, $F$ is an internal left adjoint if and only if each $F\circ \iota_i$ is. 
\end{lm}
\begin{proof}
We need to check two things about $F^R$, namely that it preserves colimits, and that its canonical projection maps $v\otimes F^R(n)\to F^R(v\otimes n)$ are equivalences. Both statements are of the form ``this map is an equivalence'', and thus can be checked after applying the family $(\iota_i)^R$. 

In both cases, we reduce to the appropriate comparison map for $(F\circ \iota_i)^R$ and for $(\iota_i)^R$, but these are assumed to be equivalences, so we are done. Let us give more details, for the projection map - the case of colimits is completely analogous. 

The claim is that $\iota_i^R$ applied to the projection map $v\otimes F^R(n)\to F^R(v\otimes n)$ fits into the following commutative diagram, where the top horizontal map is the projection map for $F\circ \iota_i$, and the leftmost bottom horizontal map is the projection map for $\iota_i$: \[\begin{tikzcd}
	{v\otimes (F\circ \iota_i)^R(n)} && {(F\circ \iota_i)^R(v\otimes n)} \\
	{v\otimes (\iota_i^R\circ F^R)(n)} & {\iota_i^R(v\otimes F^R(n))} & {\iota_i^R\circ F^R(v\otimes n)}
	\arrow["{\simeq }"', from=1-1, to=2-1]
	\arrow[from=2-1, to=2-2]
	\arrow[from=2-2, to=2-3]
	\arrow["\simeq"', from=2-3, to=1-3]
	\arrow[from=1-1, to=1-3]
\end{tikzcd}\]

Thus, if both the ones for $\iota_i$ and for $F\circ \iota_i$ are equivalences, it follows that also $\iota_i^R$ applied to the one for $F$ is, as claimed. 
  \end{proof}
\begin{cor}\label{cor:atomicimpliesinternal}
    Let $\M,\N\in\Mod_\V(\PrL)$. Any internal left adjoint $f:\M\to \N$ preserves atomic objects. If $\M$ is atomically generated, the converse holds: if $f$ preserves atomic objects, then $f$ is an internal left adjoint.   
\end{cor}
\begin{proof}
    The first part is obvious as internal left adjoints are closed under composition. 

    For the second part, we note that atomic generation amounts to the statement that the various $\V\xrightarrow{-\otimes m}\M$, for $m\in\M^\at$, are internal left adjoints and have jointly conservative right adjoints. The result thus follows from \Cref{lm:colimdetection} and the assumption that $f$ preserves atomic objects. 
\end{proof}
A more relevant corollary is:
\begin{cor}\label{cor:colimintleft}
Let $\Mod_\V(\PrL)^{iL}$ denote the (non-full) wide subcategory of $\Mod_\V(\PrL)$ whose morphisms are the internal left adjoints. The source of the forgetful functor $$\Mod_\V(\PrL)^{iL}\to \Mod_\V(\PrL)$$ has all (small) colimits, and they are preserved by the forgetful functor. 
\end{cor}
\begin{proof}
    Let $\M_\bullet: I^\triangleright \to \Mod_\V(\PrL)$ be a colimit diagram whose restriction to $I$ lands in $\Mod_\V(\PrL)^{iL}$. Let $\infty \in I^\triangleright$ denote the cone point. 

    We first note that $\M_\infty$ is equivalent to the \emph{limit} of the diagram $I\op\to \Mod_\V(\PrL)$ obtained by taking right adjoints, in such a way that the limit projection maps $\M_\infty\to \M_i$ are the right adjoints to the canonical colimit inclusion maps $\M_i\to \M_\infty$. In particular, these right adjoints are colimit-preserving and $\V$-linear. In particular, the whole diagram $I^\triangleright\to \Mod_\V(\PrL)$ lands in $\Mod_\V(\PrL)^{iL}$. 

    Furthermore, these right adjoints are jointly conservative. Thus we are in position to apply \Cref{lm:colimdetection}: a $\V$-linear functor $\M_\infty\to \N$ is an internal left adjoint if and only if each $\M_i\to \M_\infty\to \N$ is - this is precisely what we need to check to prove that a forgetful functor from a wide subcategory creates colimits. 
\end{proof}
\begin{rmk}
    The category of presentable $\V$-modules and internal left adjoints enjoys more stability properties, related in particular to absolue limits in $\V$. They should be explored further, but we do not do this here. 
\end{rmk}
\subsection{Characterizations of dualizability}
Recall that our main object of study in this paper is dualizable $\V$-modules. We finally define them explicitly:
\begin{defn}
    A $\V$-module $\M$ is called dualizable if it is so for the (relative) Lurie tensor product $\otimes_\V$ on $\Mod_\V(\PrL)$.
    
    The \category{} $\Dbl{\V}$ is the full subcategory of $\Mod_\V(\PrL)^{iL}$ spanned by dualizable $\V$-modules, in other words the (non-full) subcategory of $\Mod_\V(\PrL)$ spanned by dualizable $\V$-modules, and with morphisms the internal left adjoints. 

    When $\V=\Ss$, we simply write $(\PrL)^\dbl$, and when $\V=\Sp$, we simply write $\Prdbl$. 
\end{defn}
In the case of the base $\V =\Sp$, Lurie proves in \cite[Proposition D.7.3.1.]{SAG} the following characterizations of dualizability:
\begin{thm}[Lurie]\label{thm:lurie}
    Let $\M\in\PrL_{\st}$. The following are equivalent: 
    \begin{itemize}
        \item $\M$ is dualizable.
        \item $\M$ is a retract, in $\PrL_\st$, of a compactly generated stable \category.
        \item $\M$ is the kernel of a compact-preserving localization between compactly generated stable \categories.
        \item There is an internally left adjoint fully faithful embedding of $\M$ in a compactly generated stable \category.
        \item The colimit functor $\Ind(\M)\to \M$ admits a left adjoint; equivalently, for any $\lambda$ such that $\M$ is $\lambda$-compactly generated, the colimit functor $\Ind(\M^\lambda)\to \M$ admits a left adjoint.
    \end{itemize}
\end{thm}
\begin{ex}
    A key example (which is an input to the previous theorem) is compactly generated stable \categories: they are dualizable over $\Sp$. For example, $\Mod_R(\Sp)$ for any $R\in\Alg(\Sp)$.  
\end{ex}
\begin{ex}\label{ex:Shisdbl}
    Let $X$ be a locally compact Hausdorff space. One can prove that $\Sh(X;\Sp)$ is dualizable - there are many perspectives for this, we will mention one involving so-called compact maps in \Cref{pfofShdbl}; but it also follows from Lurie's covariant Verdier duality \cite[§5.5.5]{HA} that one can provide an explicit duality datum for $\Sh(X;\Sp)$. 

    However this \category{} is rarely compactly generated, cf. \cite{Oscarcompact} for a modern treatment. 
\end{ex}
\begin{ex}
Let $L_n:\Sp\to L_n\Sp$ denote localization at Morava $E$-theory at height $n$ (and an implicit prime $p$). The smashing theorem proves that this is a compact-preserving localization between compactly generated categories, hence its kernel $\ker(L_n)=~\{X\in~\Sp\mid~E_n\otimes~X=~0\}$ is dualizable. The negative solution to the telescope conjecture \cite{NoTel} implies in particular that it is not compactly generated ($\Ind(\ker(L_n)^\omega)$ is by definition $\ker(L_n^f)$). 
\end{ex}
\begin{ex}
    When $A$ is an (ordinary) ring and $I\subset A$ a finitely generated two-sided ideal, $I^2=I$ implies that $I$ is generated by a single central idempotent element. However, without the finite generation hypothesis, there are many nontrivial examples, and there are also examples for which $I$ is flat over $A$. In particular it follows that $A\to A/I$ induces a localization $\Mod_A\to \Mod_{A/I}$, whose kernel is rarely compactly generated, but by the above it is dualizable. 

    Such examples are the subject of almost mathematics \cite{almost}, but can also be found in functional analysis. For example, let $A = C^0(X)$ be the ring of continuous complex (or real) valued functions on a topological space $X$, and let $I$ be the ideal of functions vanishing at $x\in X$. One easily verifies $I^2=I$, and its flatness can also be proved. Another example is, given a separable Hilbert space $H$, the algebra $B(H)$ of bounded operators on $H$, with the ideal $K(H)$ of compact operators on $H$. These examples are discussed in \cite{youtubeDustin2}, and go back to the work of Karoubi \cite{karoubiconj}, Suslin-Wodzicki \cite{suslinwodzicki} and Higson \cite{higson}.  
\end{ex}
Our aim in this section is to give a generalization of Lurie's characterization to an arbitrary base $\V$, except for the third bullet point in the unstable case. The analogue of compactly generated stable \categories{} will be $\V$-atomically generated $\V$-modules. For this, we first need the crucial input, due in this generality to Berman \cite{BermanI}: 
\begin{prop}\label{prop:atgenimpliesdbl}
    Any atomically generated $\V$-module is dualizable. 
\end{prop}
By \Cref{obs:descatgen}, any atomically generated $\V$-module is of the form $\PP_\V(\M_0)$ for some small $\V$-\category{} $\M_0$. By \cite[Theorem 1.7]{BermanI}, these are dualizable. For convenience, we provide a different proof that does not rely on Berman's argument. In fact, we prove the following more general statement:
\begin{prop}\label{prop:dbldetect}
    Suppose $f_i:\M_i\to \M$ is a family of internal left adjoints such that the $f_i^R$ are jointly conservative, and such that each $\M_i$ is dualizable. In this case, $\M$ is dualizable.
\end{prop}
The previous proposition can be obtained by putting all the $\M_i=\V, f_i= -\otimes i$ for $i\in\M^\at$. We will simply need the following technical lemma:
\begin{lm}\label{lm:technicaladj}
    Let $\mathcal B$ be a $2$-category, $f:x\to y$ be map over some object $z$ in $B$ such that $f$ admits a right adjoint and the corresponding triangle is horizontally right adjointable:
    \[\begin{tikzcd}
	x & y \\
	& z
	\arrow[from=1-1, to=1-2]
	\arrow[from=1-1, to=2-2]
	\arrow[from=1-2, to=2-2]
\end{tikzcd}\]
Finally, assume that both $x\to z$ and $y\to z$ are conservative, that is, for every $b\in B, \hom(b,-)$ applied to them is a conservative functor. Then $f$ is an equivalence. 
\end{lm}
\begin{proof}
    By applying $\hom(b,-)$ for every $b\in B$ we reduce to the case $B=\Cat$. In that case, we simply need to prove that the unit and counit $\id_x\to f^Rf$ and counit $ff^R\to \id_y$ are equivalences. By conservativity, this can be checked after mapping to $z$, where the claim follows by adjointability: letting $p:x\to z,q:y\to z$ denote the corresponding maps, we have $pf^Rf = qf=p= p\id_x$ (and one easily checks that the maps are the correct ones).
\end{proof}
\begin{proof}
We wish to show that for all $\N$, the map $\Fun^L_\V(\M,\V)\otimes_\V\N\to \Fun^L_\V(\M,\N)$ is an equivalence, so fix an arbitrary $\N$. 

We note that the transformation $\Fun^L_\V(-,\V)\otimes_\V\N\to \Fun^L_\V(-,\N)$ participates in a $2$-functor $\Mod_\V(\PrL)\op\to \Mod_\V(\PrL)^{\Delta^1}$ and so it sends our internal left adjoints $\M_i\to \M$ to left adjoints, i.e. to vertically left adjointable squares: 
\[\begin{tikzcd}
	{\Fun^L_\V(\M,\V)\otimes_\V\N} & {\Fun_\V^L(\M,\N)} \\
	{\Fun^L_\V(\M_i,\V)\otimes_\V\N} & {\Fun^L_\V(\M_i,\N)}
	\arrow[from=1-1, to=1-2]
	\arrow[from=1-1, to=2-1]
	\arrow[from=1-2, to=2-2]
	\arrow["\simeq"', from=2-1, to=2-2]
\end{tikzcd}\]
where the bottom horizontal arrows are equivalences since the $\M_i$'s are dualizable. 

Note that the $f_i^*: \Fun^L_\V(\M,\N)\to \Fun_\V^L(\M_i,\N)$ are jointly conservative since the images of the $f_i$'s generate $\M$ under colimits and $\V$-tensors. Thus, their left adjoints generate $\Fun^L_\V(\M,\N)$ under colimits and $\V$-tensors, and this is so also for $\N=\V$. This property is preserved by tensoring with $\N$. It follows then from \Cref{lm:colimdetection} that the top horizontal arrows are internal left adjoints. 

Now, the squares are vertically left adjointable, and hence horizontally right adjointable too. Since the bottom horizontal maps are equivalences, it follows from \Cref{lm:technicaladj} that the top map is an equivalence. 
\end{proof}
\begin{lm}
    Let $f:\M\to \N\in\PrL$ be a localization. For any $\D\in\PrL$, $\D\otimes f$ is a localization. 
\end{lm}
\begin{proof}
For a class $S$ of maps, we let $\Fun_S$ denote the full subcategory of the functor category spanned by functors sending $S$ to equivalences.

    Let $W$ be a generating set of equivalences for the localization $f$. In this case, we note that for any $\E\in\PrL$ restriction along $ f$ induces an equivalence $$\Fun^L(\D,\Fun^L(\N,\E))\to \Fun^L(\D,\Fun^L_W(\M,\E))$$
    Thus letting $W'= \D^\kappa\otimes W$ for some $\kappa$ such that $\D$ is $\kappa$-compactly generated, we find that restriction along $\D\otimes f$ induces an equivalence $$\Fun^L(\D\otimes\N,\E)\to \Fun^L_{W'}(\D\otimes \M,\E)$$ 
    \end{proof}
    
\begin{warn}
    One may be tempted to conclude a similar statement about fully faithful functors. This is not so, for example the fully faithful inclusion $\Sp_{\geq 0}\to \Sp$ tensored with $\Set$ yields the unique map $\Ab\to 0$ (we learned of this example from Lior Yanovski). One can also produce stable counterexamples, cf. \cite[Theorem 2.2]{sashaI}.
\end{warn}
    \begin{lm}
        Let $f,g:I\to \PrL$ be small diagrams and $L:f\to g$ a natural transformation which is pointwise a localization. The induced functor $\colim_I f\to\colim_I g$ is also a localization.
    \end{lm}
    \begin{proof}
        It suffices to prove that its right adjoint is fully faithful. But now this follows from the description of colimits in $\PrL$ (cf. \cite[Theorem 5.5.3.18, Corollary 5.5.3.4]{HTT}) and the fact that limits of fully faithful functors are fully faithful. 
    \end{proof}
\begin{cor}\label{cor:tensloc}
    Let $f:\M\to \N\in\Mod_\V(\PrL)$ be a functor whose underlying functor in $\PrL$ is a localization. For any $\D\in\Mod_\V(\PrL)$, $\D\otimes_\V f$ is also a localization. 
\end{cor}
\begin{proof}
    This map is the colimit of the maps $\D\otimes\V^{\otimes n}\otimes \M\to\D \otimes \V^{\otimes n}\otimes \N$. The result follows from combining the two previous lemmas.
\end{proof}
The final preparation we need is the following general lemma which can be found in \cite[Lemma 21.1.2.14]{SAG}: 
\begin{lm}\label{lm:retradj}
    Let $B$ be an $(\infty,2)$-category, and let $f:x\to y$ be a $1$-morphism in $B$. For $f$ to admit a right adjoint, it suffices that :
    \begin{enumerate}
        \item $B(z,x)$ is idempotent-complete for all $z\in B$; 
        \item $f$ is a retract, in $(\iota_1B)^{\Delta^1}$ of a $1$-morphism that admits a right adjoint.  
    \end{enumerate}
\end{lm}
\begin{proof}
We first prove the corresponding property when $B=\Cat$. In this case, let $f:C\to D$ be a functor which is a retract, in $\Cat^{\Delta^1}$, of a left adjoint $f': C'\to D'$. Let $i_C,i_D$ denote the ``inclusions'', and $r_C,r_D$ the corresponding retractions. 

Now fix $d\in D$ and consider $\Map(f(-),d): C\op\to\Ss$ : we wish to prove it is representable, so by idempotent completeness, it suffices to prove it is the retract of a representable functor. 

We have maps $$\Map(f(-),d)\to \Map(i_D f(-), i_D(d))\simeq \Map(f' i_C(-), i_D(d))\simeq \Map(i_C(-), (f')^Ri_D(d))\to \Map(-,r_C(f')^Ri_D(d))$$ This last functor is clearly representable, so it suffices to produce a retraction. 

It is simply given by $$\Map(-,r_C(f')^Ri_D(d))\to \Map(f(-), fr_C(f')^R, i_D(d))\simeq \Map(f(-),r_D f'(f')^Ri_D(d))$$

$$\to \Map(f(-),r_Di_D(d))\simeq \Map(f(-),d)$$ and it is a simple matter of diagram chasing to prove that it is indeed a retraction.

We note that, in particular, the right adjoint $f^R$ is a retract of $r_C(f')^Ri_D$. 

We can now prove the general case: let $f:x\to y$ be a retract of $f':x'\to y'$, and we use the notations $i_x,i_y,r_x,r_y$ as above. In this case, for every $z\in B$, $B(z,f)$ is a retract of $B(z,f')$ and so admits a right adjoint, by the case of $\Cat$ and our idempotent-completeness assumption. It now suffices to prove that for any $g:z'\to z$, the following square is horizontally right adjointable:
\[\begin{tikzcd}
	{B(z,x)} & {B(z,y)} \\
	{B(z',x)} & {B(z',y)}
	\arrow["{B(g,x)}"', from=1-1, to=2-1]
	\arrow["{B(z',f)}"', from=2-1, to=2-2]
	\arrow["{B(g,y)}", from=1-2, to=2-2]
	\arrow["{B(z,f)}", from=1-1, to=1-2]
\end{tikzcd}\]

We note that this square is a retract of the corresponding square for $f'$, we are left with observing that retracts of horizontally right-adjointable diagrams are horizontally right adjointable, which is the object of the next lemma. 
\end{proof}
\begin{lm}\label{lm:retractBC}
    In $\Cat^{\square}$, retracts of horizontally right adjointable squares are horizontally right adjointable. 
\end{lm}
\begin{proof}
    Consider two squares $\square_i$, for 
    $i=0,1$, of the form : 
    \[\begin{tikzcd}
	{C_i} & {D_i} \\
	{E_i} & {F_i}
	\arrow["{q_i}"', from=1-1, to=2-1]
	\arrow["{h_i}"', from=2-1, to=2-2]
	\arrow["{f_i}", from=1-2, to=2-2]
	\arrow["{p_i}", from=1-1, to=1-2]
\end{tikzcd}\]
where $p_i,h_i$ have right adjoints $p_i^R,h_i^R$, and assume $\square_0$ is a retract of $\square_1$.

We have the Beck-Chevalley map $q_ip_i^R\to h_i^R f_i$. The claim is that the one for $i=0$ is a retract of a variant of the one for $i=1$. 

Specifically, call $i_C,i_D,i_E,i_F$ the functors participating in the ``inclusion'' functor of the retraction, and similarly $r_C,r_D,r_E,r_F$ the functors participating in the retraction. 

We claim that the map $r_E q_1p_1^R i_D\to r_E\circ h_1^Rf_1\circ i_D$ given by applying $r_E\circ (-)\circ i_D$ to the Beck-Chevalley map retracts onto the Beck-Chevalley map for $\square_0$. This is a simple but tedious diagram chase, which is in fact part of a more general statement about functoriality of Beck-Chevalley maps. 

We note that because equivalences can be checked in homotopy cateories, this statement is insensitive to passing to homotopy categories everywhere, and thus is a $1$-categorical statement. We leave the details to the reader. 
\end{proof}
We can now state and prove: 
\begin{thm}\label{thm:Luriethmgeneralbase}
    Let $\V\in\CAlg(\PrL_\kappa)$ and $\M\in\Mod_\V(\PrL)$. Suppose $\M \in \Mod_\V(\PrL_\lambda)$ for some $\lambda \geq \kappa$. The following are equivalent: 
    \begin{enumerate}
        \item $\M$ is a dualizable $\V$-module; 
        \item $\M$ is a retract of an atomically generated $\V$-module;
        \item There is an internally left adjoint fully faithful embedding $\M\to \N$ for some atomically generated $\V$-module $\N$;
        \item The canonical functor $\PP_\V(\M^\mu)\to \M$ admits an internal left adjoint for some $\mu\geq \lambda$ (here $\M^\mu$ denotes the full $\V$-subcategory of $\M$ spanned by the $\mu$-compacts). 
    \end{enumerate}
\end{thm}
\begin{rmk}\label{rmk:boundyhat}
    We will see later (cf. \Cref{thm:CardBound}) that any dualizable $\V$-module is $\max(\kappa,\aleph_1)$-compactly generated (and in fact, is in $\Mod_\V(\PrL_{\max(\kappa,\aleph_1)})$) so that the $\mu$ in the fourth equivalent statement can always chosen to be $\max(\kappa,\aleph_1)$. 
\end{rmk}
\begin{proof}
    (4. implies 3.) : By \Cref{lm:ffrightadjointV}, the right adjoint of $\PP_\V(\M^\lambda)\to \M$ is fully faithful, hence, if a left adjoint exists, it is also fully faithful\footnote{This is standard, see \Cref{cor:leftrightff}.}. Furthermore, $\PP_\V(\M^\lambda)$ is atomically generated, so we are done. 

    (3. implies 2.) : This is clear, as the ($\V$-linear, colimit-preserving) right adjoint of a fully faithful embedding $\M\to \N$ provides a retraction (indeed, the unit of the adjunction is an equivalence if the left adjoint is fully faithful). 

    (2. implies 1.) : As dualizable objects are closed under retracts in any idempotent-complete \category{} such as $\Mod_\V(\PrL)$, this follows at once from \Cref{prop:atgenimpliesdbl}.  

    To prove that 1. implies 4., we prove first that 1. implies 2., that 4. is satisfied in the case of atomically generated $\V$-modules, and finally that 4. is closed under retracts. 

    (1. implies 2.) : Let $\M$ be dualizable. In that case, $\Fun^L_\V(\M,-)\simeq \M^\vee \otimes_\V -$ preserves $\V$-linear localizations by \Cref{cor:tensloc}, in particular it sends them to (essential) surjections. 

    Thus the identity functor $\id_\M$ has a preimage under $\Fun^L_\V(\M,\PP_\V(\M^\lambda))\to \Fun^L_\V(\M,\M)$. This provides a section of the localization functor $\PP_\V(\M^\lambda)\to \M$\footnote{We warn the reader that this section need a priori not be itself the left adjoint that we are after.}, thus proving 2..

    Now, let $\M$ be $\V$-atomically generated, we prove 4. in this case. Let $p:\PP_\V(\M^\lambda)\to \M$ be the canonical functor. For $x\in \M^\at\subset \M^\lambda$, we have a map $$\hom_{\PP_\V(\M^\lambda)}(y(x),F)\to \hom_\M(x,p(F))$$ natural in $F\in\PP_\V(\M^\lambda)$. Both $x$ and $y(x)$ are atomic in $\M$ and $\PP_\V(\M^\lambda)$ respectively, so both sides are $\V$-linear colimit-preserving functors $\PP_\V(\M^\lambda)\to \V$.  Therefore this map is an equivalence for all $F$ if and only if its restriction to $\M^\lambda$ is, in which case this follows from the enriched Yoneda lemma applied to $\M^\lambda$. 
    
Therefore, for every $v\in\V$, there is a natural equivalence $$\Map(v\otimes x, p(-))\simeq \Map(v\otimes y(x), -)$$ which proves that $p$ admits a local left adjoint at $v\otimes x$ for every $x\in \M^\at$. It follows from cocompleteness that $p$ admits a left adjoint $p^L$. Furthermore, the previous construction at $v\otimes x$ shows that the canonical map $p^L(v\otimes x)\to v\otimes p^L(x)$ is an equivalence for all $x\in \M^\at$ and $v\in \V$. It follows by 2-out-of-3 that the following map is an equivalence for all $v,w\in\V, x\in\M^\at$: $p^L(v\otimes~w\otimes~x)\to~v\otimes~p^L(w\otimes~x)$, and thus, by completion under colimits and the fact that $\M$ is atomically generated, it follows that the canonical map $p^L(v\otimes x)\to~v\otimes~p^L(x)$is an equivalence for all $x\in \M$ and $v\in \V$, so that in total, $p^L$ is an internal left adjoint to $p$.
     
    Finally, 4. follows from (the dual of) \Cref{lm:retradj}: functors admitting a left adjoint are closed under retracts as long as the source is idempotent-complete (which presentable categories certainly are).
\end{proof}
\begin{rmk}\label{rmk:cardindep}
    Let $\M$ be a $\lambda$-compactly generated $\V$-module, and let $\mu\geq \lambda$. In this case, the canonical functor $\PP_\V(\M^\mu)\to \M$ factors as $\PP_\V(\M^\mu)\to \PP_\V(\M^\lambda)\to \M$, where the first map is restriction, and the second is the canonical functor.  Indeed, the composite $\M^\mu\to~\PP_\V(\M^\mu)\to~\PP_\V(\M^\lambda)$ is then the restriction to $\M^\mu$ of the restricted Yoneda embedding, and so mapping back to $\M$ yields the inclusion of $\M^\mu$ in $\M$, and this is enough by \Cref{thm: UPVPsh}. 

    In particular, if both canonical functors admit left adjoints, then the one for $\mu$ factors through the one for $\lambda$.
\end{rmk}
\begin{nota}
    Let $\M$ be a $\lambda$-compactly generated dualizable $\V$-module. Point 4. in \Cref{thm:Luriethmgeneralbase} provides the existence of a left adjoint to the canonical functor $\PP_\V(\M^\lambda)\to \M$. We let $\hat y$ denote this left adjoint\footnote{This notation is a bit abusive as $\hat y$ depends in principle on $\lambda$. However this is rather harmless, cf. \Cref{rmk:boundyhat}, and cf. \Cref{rmk:cardindep}}.
\end{nota}
\begin{rmk}\label{rmk:hatyff}
    The canonical functor $\PP_\V(\M^\lambda)\to \M$ is a localization by \Cref{cor:locVmod}, so its left adjoint $\hat y$ is fully faithful, and thus also $\V$-fully faithful because it is $\V$-linear. 
\end{rmk}
\begin{ex}
    From the proof, we see immediately that if $m\in \M^\at$, $\hat y(m) = y(m)$. More precisely, let $y: \M\to \PP_\V(\M^\lambda)$ be \emph{right} adjoint to the canonical functor. Since $y$ and $\hat y$ are both fully faithful, we obtain a comparison $\hat y\to y$ which is an equivalence when evaluated on atomic objects. 
\end{ex}
\begin{cor}
    Let $\M$ be a dualizable $\V$-module, and $T$ a $\V$-linear colimit-preserving monad on $\M$. The category $\Mod_T(\M)$ of $T$-modules is also dualizable, and the free $T$-module functor is an internal left adjoint.  
\end{cor}
\begin{proof}
We first point out that the internal left adjoint claim is clear. In particular, since its right adjoint is conservative, \Cref{prop:dbldetect} applies, and we are done. 
\end{proof}
\begin{ques}
    Is the analogue of this statement true for $\V$-linear comonads ?
\end{ques}
\begin{prop}\label{prop:internalleftBC}
Assume $\V\in\CAlg(\PrL_\kappa)$, and let $\M,\N$ be $\lambda$-compactly generated dualizable $\V$-modules where $\lambda\geq \kappa$. 

A $\V$-linear, $\lambda$-compact preserving map $f:\M\to \N$ is an internal left adjoint if and only if the following square is vertically left adjointable: 
\[\begin{tikzcd}
	{\PP_\V(\M^\lambda)} & {\PP_\V(\N^\lambda)} \\
	\M & \N
	\arrow["f"', from=2-1, to=2-2]
	\arrow[from=1-1, to=2-1]
	\arrow[from=1-2, to=2-2]
	\arrow["{\PP_\V(f^\lambda)}", from=1-1, to=1-2]
\end{tikzcd}\]

In fact, if there \emph{exists a} homotopy filling the square: 
\[\begin{tikzcd}
	{\PP_\V(\M^\lambda)} & {\PP_\V(\N^\lambda)} \\
	\M & \N
	\arrow["f"', from=2-1, to=2-2]
	\arrow["{\PP_\V(f^\lambda)}", from=1-1, to=1-2]
	\arrow["{\hat y}", from=2-1, to=1-1]
	\arrow["{\hat y}"', from=2-2, to=1-2]
\end{tikzcd}\]
then $f$ is an internal left adjoint.
\end{prop}
\begin{rmk}
    The subtlety in the ``in fact'' part is that we do not require the homotopy to be the canonical Beck-Chevalley transformation as in the definition of left adjointable. 
\end{rmk}
\begin{proof}
Let $p_\M:\PP_\V(\M^\lambda)\to \M$ (resp. $p_\N$) denote the canonical functor.

    First, assume that there is a homotopy making the following square commute (e.g. the Beck-Chevalley transformation, but not necessarily): 
    \[\begin{tikzcd}
	{\PP_\V(\M^\lambda)} & {\PP_\V(\N^\lambda)} \\
	\M & \N
	\arrow["f"', from=2-1, to=2-2]
	\arrow["{\PP_\V(f^\lambda)}", from=1-1, to=1-2]
	\arrow["{\hat y}", from=2-1, to=1-1]
	\arrow["{\hat y}"', from=2-2, to=1-2]
\end{tikzcd}\]
Note that by design, $\PP_\V(f^\lambda)$ preserves atomic objects in $\PP_\V(\M^\lambda)$, so it is itself an internal left adjoint by \Cref{cor:atomicimpliesinternal}. 

Taking right adjoints everywhere, we find that $f^R\circ p_\N\simeq p_\M\circ \PP_\V(f^\lambda)^R$. Precomposing by $\hat y_\N$ and using \Cref{rmk:hatyff}, we find that $f^R\simeq f^R\circ p_\N\circ \hat y_\N \simeq p_\M\circ \PP_\V(f^\lambda)^R \circ \hat y_\N$, all these equivalences being as $\V$-functors, or as lax $\V$-linear functors. As the last one is $\V$-linear and preserves colimits, we conclude that the same holds for $f^R$, thus proving the claim. 

Conversely, assume $f$ is an internal left adjoint, we wish to prove that the canonical natural transformation in the following square is an equivalence: 
\[\begin{tikzcd}
	{\PP_\V(\M^\lambda)} & {\PP_\V(\N^\lambda)} \\
	\M & \N
	\arrow["{\hat y}", from=2-1, to=1-1]
	\arrow["f"', from=2-1, to=2-2]
	\arrow["{\hat y}"', from=2-2, to=1-2]
	\arrow["{\PP_\V(f^\lambda)}", from=1-1, to=1-2]
	\arrow[shorten <=8pt, shorten >=8pt, Rightarrow, from=2-2, to=1-1]
\end{tikzcd}\]

Up to enlarging $\lambda$, we may assume that $f^R$ preserves $\lambda$-compacts - note that this does not affect $\hat y$, and thus the adjointability of this square, by \Cref{rmk:cardindep}. Now, pass to right adjoints in the whole diagram to reach: 
\[\begin{tikzcd}
	{\PP_\V(\M^\lambda)} & {\PP_\V(\N^\lambda)} \\
	\M & \N
	\arrow["{p_\M}"', from=1-1, to=2-1]
	\arrow["{p_\N}", from=1-2, to=2-2]
	\arrow["{\PP_\V((f^R)^\lambda)}"', from=1-2, to=1-1]
	\arrow["{f^R}", from=2-2, to=2-1]
	\arrow[shorten <=8pt, shorten >=8pt, Rightarrow, from=1-1, to=2-2]
\end{tikzcd}\]
For general reasons, see \cite[Remark 4.7.4.14]{HA}, the natural transformation in this adjointed diagram is the Beck-Chevalley transformation for the horizontal right adjointability of the original square: 
\[\begin{tikzcd}
	{\PP_\V(\M^\lambda)} & {\PP_\V(\N^\lambda)} \\
	\M & \N
	\arrow["f"', from=2-1, to=2-2]
	\arrow[from=1-1, to=2-1]
	\arrow[from=1-2, to=2-2]
	\arrow["{\PP_\V(f^\lambda)}", from=1-1, to=1-2]
\end{tikzcd}\]
Thus we can ask instead if the latter is horizontally right adjointable. But now, because $f^R$ preserves colimits and $\V$-tensors, this can be tested after restricting to $\N^\lambda$, where the conclusion is clear. 
\end{proof}
The adjointability of these squares provides canonical naturality squares. As the canonical map $\PP_\V(\M^\lambda)\to \M$ is itself natural, the above gives us:
\begin{cor}\label{cor:yhatnatural}
Let $\V\in\CAlg(\PrL_\kappa)$.
    The maps $\hat y: \M\to \PP_\V(\M^\lambda)$ assemble canonically into a natural transformation of functors defined on $\Dbl{\V}\cap \Mod_\V(\PrL_\lambda)$ for $\lambda\geq \kappa$; which is a natural section for the natural map $\PP_\V(\M^\lambda)\to \M$. 
\end{cor}
\begin{rmk}
    As before, the restriction concerning $\lambda$ is quite harmless in view of \Cref{thm:CardBound} and \Cref{rmk:cardindep}, cf. \Cref{cor:yhatnatural2}. 
\end{rmk}
\begin{proof}
\cite[Theorem 4.6]{Runemonad} provides a colax natual transformation $\hat y$, and an equivalence of colax natural transformations $p\hat y \simeq \id_\M$. \Cref{prop:internalleftBC} proves that this colax natural transformation is in fact a natural transformation.
\end{proof}
We obtain the following which, in light of \Cref{thm:Luriethmgeneralbase} and the prevalence of ``atomic preserving functors'' suggests that internal left adjoints between dualizable $\V$-modules is the correct notion of morphisms:
\begin{cor}\label{cor:retractfunctor}
A $\V$-linear map $f:\M\to \N$ between dualizable $\V$-modules is an internal left adjoint if and only if it is a retract in $\Mod_\V(\PrL)^{\Delta^1}$ of an internal left adjoint\footnote{Equivalently atomic-preserving.} functor between atomically generated $\V$-modules.
\end{cor}
\begin{proof}
By \Cref{prop:internalleftBC}, we obtain ``only if''. 

Now suppose $f: \M\to \N$ is a retract of $f_0:\M_0\to \N_0$, an internal left adjoint between atomically generated $\V$-modules. This retraction happens in $\Mod_\V(\PrL_\kappa)$ for $\kappa$ sufficiently large, and therefore, by naturality of the map $\PP_\V(C^\kappa)\to C$, we obtain that the square from \Cref{prop:internalleftBC} for $f$ is a retract of the same square for $f_0$. By (the dual of) \Cref{lm:retractBC}, we obtain vertical left adjointability for the square 
\[\begin{tikzcd}
	{\PP_\V(\M^\kappa)} & {\PP_\V(\N^\kappa)} \\
	\M & \N
	\arrow[from=1-1, to=1-2]
	\arrow[from=1-1, to=2-1]
	\arrow[from=1-2, to=2-2]
	\arrow[from=2-1, to=2-2]
\end{tikzcd}\]
    which, by \Cref{prop:internalleftBC} implies that $f$ is strongly continuous. 
\end{proof}
There is another proof of the following result based on analyzing in detail the equivalence $\colim_I \M_i \simeq \lim_{I\op}\M_i$, but the above gives us access to another proof of:
\begin{prop}\label{prop:colimdbl}
    $\Dbl{\V}$ admits all small colimits, and the forgetful functor $$\Dbl{\V}\to \Mod_\V(\PrL)$$ preserves them.
\end{prop}
\begin{proof}
    By \Cref{cor:colimintleft}, it suffices to prove that the forgetful functor $\Dbl{\V}\to~\Mod_\V(\PrL)^{iL}$ preserves colimits, i.e. that dualizable $\V$-modules are closed under colimits along internal left adjoints. We first observe that this is true for atomically generated $\V$-modules. To prove this, let $\M_\bullet: I\to \Mod_\V(\PrL)^{iL}$ be a diagram of atomically generated $\V$-modules, and consider its colimit $\M_\infty$. By \Cref{cor:colimintleft}, each of the maps $\M_i\to\M_\infty$ is an internal left adjoint and thus sends atomics in $\M_i$ to atomics in $\M_\infty$. Furthermore, it is clear that the images of the $\M_i$ generate $\M_\infty$ under tensors and colimits, which all in all proves that $\M_\infty$ is atomically generated.

    Second, we note that for any diagram $\M_\bullet : I\to \Dbl{\V}$, there exists a $\lambda$ such that it lands in $\Dbl{\V}\cap \Mod_\V(\PrL_\lambda)$, so that  $\hat y: \M_i \rightleftarrows \PP_\V(\M_i^\lambda) $ provides a natural retraction (cf. \Cref{cor:yhatnatural}) of an atomically generated $\V$-module onto $\M_i$, and the maps in the diagram of the $\PP_\V(\M_i^\lambda)$ clearly preserve atomic objects. In particular, upon passing to the colimit, we find that $\colim_I \M_i$ is a retract of $\colim_I \PP_\V(\M_i^\lambda)$, and the latter is atomically generated by the previous discussion. By \Cref{thm:Luriethmgeneralbase}(2), this proves that $\colim_I\M_i$ is dualizable.  
\end{proof}
    \subsection{An adjunction}
    
    In \Cref{prop:colimdbl}, we saw that the functor $\Dbl{\V}\to \Mod_\V(\PrL)$ preserves all colimits; since this is also the case for $\Mod_\V(\PrL_{\kappa})$ for $\kappa$ large enough, it follows that for any $\kappa$ large enough, the inclusion $\Dbl{\V}\cap \Mod_\V(\PrL_{\kappa}) \to \Mod_\V(\PrL_\kappa)$ also does. In what follows, we will see that the source category is presentable, and hence this implies that this functor has a right adjoint. However, we can also produce a right adjoint explicitly, and in fact, this will be convenient\footnote{To avoid circularity.} because we \emph{use} this right adjoint to prove said presentability. 
    
    The main result, which is a general version of \cite[Proposition 1.90]{sashaI}, is: 
    \begin{thm}\label{thm:nonstandadj}
        Let $\V\in\CAlg(\PrL_\kappa)$, and consider the inclusion $$\Dbl{\V}\cap\Mod_\V(\PrL_\kappa)\to \Mod_\V(\PrL_\kappa)$$ It admits a right adjoint, given by $\N\mapsto \PP_\V(\N^\kappa)$. 
    \end{thm}
    \begin{rmk}\label{rmk:nonstandadjsymmon}
        We note that the inclusion in question is canonically symmetric monoidal, hence this upgrades to a symmetric monoidal adjunction, and thus to an adjunction between categories of commutative algebras. 
    \end{rmk}
    \begin{proof}
        We set up this adjunction using a unit and a counit. 
    
        (The unit.) By \Cref{cor:yhatnatural}, there is a natural transformation $\hat y: \M\to \PP_\V(\M^\kappa)$ on $\Dbl{\V}\cap\Mod_\V(\PrL_\kappa)$, we take this to be our unit. 
    
        (The counit.) There is a natural map of $\V$-categories $\N^\kappa\to \N$ which induces, by \Cref{thm: UPVPsh}, a natural map $c_\N: \PP_\V(\N^\kappa)\to \N$, we take this to be our counit. Note that it does live in $\Mod_\V(\PrL_\kappa)$ because it sends atomic generators in $\PP_\V(\N^\kappa)$ to $\kappa$-compacts in $\N$. 
    
        We now need to check two triangle identities. 
    
        Let $i$ denote the inclusion, and $P: \N\mapsto \PP_\V(\N^\kappa)$ our candidate right adjoint. The first triangle identity looks like: 
        \[\begin{tikzcd}
    	i\M & iPi\M \\
    	& i\M
    	\arrow["i\eta", from=1-1, to=1-2]
    	\arrow["{\epsilon i}", from=1-2, to=2-2]
    	\arrow["\id"', from=1-1, to=2-2]
    \end{tikzcd}\]
    which unravels to: 
    \[\begin{tikzcd}
    	\M & {\PP_\V(\M^\kappa)} \\
    	& \M
    	\arrow["{\hat y}", from=1-1, to=1-2]
    	\arrow["c", from=1-2, to=2-2]
    	\arrow["\id"', from=1-1, to=2-2]
    \end{tikzcd}\]
    However, we have seen in \Cref{cor:yhatnatural} that $\hat y$ was a natural section for $c$, i.e. this diagram \emph{does} naturally commute. 
    
    The second triangle identity, while psychologically confusing, is also essentially clear from our work so far. It looks like: 
    \[\begin{tikzcd}
    	{P(\N)} & {PiP(\N)} \\
    	& {P(\N)}
    	\arrow["{\eta P}", from=1-1, to=1-2]
    	\arrow["P\epsilon", from=1-2, to=2-2]
    	\arrow["\id"', from=1-1, to=2-2]
    \end{tikzcd}\]
    which unwinds to: 
    \[\begin{tikzcd}
    	{\PP_\V(\N^\kappa)} & {\PP_\V(\PP_\V(\N^\kappa)^\kappa)} \\
    	& {\PP_\V(\N^\kappa)}
    	\arrow["{\hat y_{\PP_\V(\N^\kappa)}}", from=1-1, to=1-2]
    	\arrow["{\PP_\V(c^\kappa)}", from=1-2, to=2-2]
    	\arrow["\id"', from=1-1, to=2-2]
    \end{tikzcd}\]
    
    We note that it suffices to prove that this triangle commutes (naturally) after precomposing with $y:\N^\kappa\to \PP_\V(\N^\kappa)$ (again by \Cref{thm: UPVPsh}). But now, by the way $\hat y$ was constructed in \Cref{thm:Luriethmgeneralbase}, we find that for $\PP_\V(C)$ for some small $\V$-enriched category $C$, we have a natural\footnote{Strictly speaking, we have not proved naturality in \Cref{thm:Luriethmgeneralbase}. It is, however, not hard to do so but also unnecessary: in this case a pointwise analysis suffices, cf. \cite[2.1.11.]{RV}.} commutative diagram: 
    \[\begin{tikzcd}
    	C & {\PP_\V(C)^\kappa} \\
    	{\PP_\V(C)} & {\PP_\V(\PP_\V(C)^\kappa)}
    	\arrow["y"', from=1-1, to=2-1]
    	\arrow["{\hat y}"', from=2-1, to=2-2]
    	\arrow["y", from=1-1, to=1-2]
    	\arrow["y", from=1-2, to=2-2]
    \end{tikzcd}\]
    and thus a commutative diagram: 
    \[\begin{tikzcd}
    	{\N^\kappa} & {\PP_\V(\N^\kappa)^\kappa} \\
    	{\PP_\V(\N^\kappa)} & {\PP_\V(\PP_\V(\N^\kappa)^\kappa)} \\
    	& {\PP_\V(\N^\kappa)}
    	\arrow["y"', from=1-1, to=2-1]
    	\arrow["y", from=1-2, to=2-2]
    	\arrow["y", from=1-1, to=1-2]
    	\arrow["{\hat y}"', from=2-1, to=2-2]
    	\arrow["{\PP_\V(c^\kappa)}", from=2-2, to=3-2]
    \end{tikzcd}\]
    Combining this with the definitional commutative diagram:
    \[\begin{tikzcd}
    	{\PP_\V(\PP_\V(\N^\kappa)^\kappa)} & {\PP_\V(\N^\kappa)^\kappa} \\
    	{\PP_\V(\N^\kappa)} & {\N^\kappa}
    	\arrow["{\PP_\V(c^\kappa)}"', from=1-1, to=2-1]
    	\arrow["{c^\kappa}", from=1-2, to=2-2]
    	\arrow["y"', from=1-2, to=1-1]
    	\arrow["y", from=2-2, to=2-1]
    \end{tikzcd}\]
    and finally using that the composite $\N^\kappa\xrightarrow{y}\PP_\V(\N^\kappa)^\kappa\xrightarrow{c^\kappa}\N^\kappa$ is the identity, we may conclude. 
    \end{proof}
    
    \begin{rmk}\label{rmk:nonstandadjInd}
    When $\V=\Sp$, $\PP_\V$ is better known as $\Ind(-)$. Unlike in the general case, $\Ind(-)$ can be defined also for presentable categories (as was e.g. used in the statement of \Cref{thm:lurie}), and the above theorem has an analogue with $\Ind(\N)$ in place of $\Ind(\N^\kappa)$. It becomes set-theoretically scarier to phrase it in terms of adjunctions\footnote{In what category does $\Ind$ of a large category live ? }, but one can prove that for $\M\in\Prdbl$, there is an equivalence: 
    $$\Fun^L(\M,\N)\simeq \Fun^{iL}(\M,\Ind(\N))$$
    \end{rmk}
    \begin{cor}
    Let $\M,\N\in\Mod_\V(\PrL_\kappa)$ and assume $\M,\N$ are further dualizable. Let $f:\M\to \N$ be a $\V$-linear functor, $\kappa$-compact-preserving functor and let $\hat f: \M\to \PP_\V(\N^\kappa)$ be the internal left adjoint corresponding to it via \Cref{thm:nonstandadj}. 
    
    $f$ is an internal left adjoint if and only if $\hat f$ factors through $\hat y_\N$. 
    \end{cor}
    \begin{proof}
        From the proof of \Cref{thm:nonstandadj}, we see that $\hat f = \PP_\V(f^\kappa)\circ \hat y_\M$. If it factors through $\hat y_\N$, then it follows from applying $p_\N:\PP_\V(\N^\kappa)\to \N$ that the factorization $\M\to \N$ is equivalent to $f$ as a functor, so that the ``in fact'' part of \Cref{prop:internalleftBC} implies that $f$ is an internal left adjoint. 
    
        The converse is clear also because of \Cref{prop:internalleftBC}. 
    \end{proof}
\subsection{Compact assembly}
In the case $\V=\Ss$, dualizability is a rather strong condition to impose on an \category{} $\M\in \PrL$; one of the key examples, stably, is the \category{} of sheaves of specta on a locally compact Hausdorff space $X$, but unstably, it is not dualizable\footnote{See however the example given by Harpaz in \cite{Yonatan}, which is closely related to sheaves.}. However, it is still \emph{compactly assembled} in the sense of \cite[Definition 21.1.2.1]{SAG}, which amounts to saying that it is a retract of a compactly generated \category. 

This notion, while weaker than dualizability, is still relevant and we will study it in some detail too. In fact, our main results suggest that ``compactly assembled'' is a better notion than ``dualizable'', even in the stable case where they are equivalent. 
\begin{rmk}\label{rmk:compassdbl}
    In \cite[Remark 1.37]{sashaI}, Efimov sketches a proof of the fact that, while the notions of (unstably) dualizable and compactly assembled are not equivalent $(\PrL)^\dbl$ is (abstractly) equivalent to a certain \category{} of compactly assembled \categories{}, which we discuss later. In particular, our results about the \category{} $(\PrL)^\dbl$ itself carry over to this \category{} of compactly assembled \categories. 
\end{rmk}
\begin{defn}\label{defn:compactass}
    An \category{} $\M\in \widehat{\Cat_\infty}$ is called compactly assembled if it is a retract of a compactly generated \category; equivalently if it admits small filtered colimits and the colimit functor $\Ind(\M)\to \M$ admits a left adjoint. 
\end{defn}
In particular we do not assume that $\M$ is cocomplete, it only has to admit filtered colimits. 

There is an obvious variant with $\kappa$-compact generation, and the equivalence of definition holds just as well in this case. In particular, we mention: 
\begin{obs}\label{obs:kcpctass}
    An \category{} $\M\in\widehat{\Cat_\infty}$ is $\kappa$-compactly assembled if and only if there exists a $\kappa$-compactly generated \category{} $\N$, and a fully faithful embedding $i:~\M\to~\N$ admitting a $\kappa$-filtered colimit-preserving right adjoint. 
\end{obs}

We point out here one of the key features of compactly assembled categories: 
\begin{prop}\label{prop:cpctasshyper}
    Let $\M\in\PrL$ be compactly assembled, and let $\mathcal X$ denote an $\infty$-topos. If $\mathcal X$ has enough points, then the collection of points $x^*:\mathcal X\to \Ss$ induces a jointly conservative family of functors $\Sh(\mathcal X;\M):=\Fun^R(\mathcal X\op,\M)\simeq\mathcal X\otimes\M\to\M$.
\end{prop}
To prove this in the case of a general $\infty$-topos, we need the following observation:
\begin{lm}\label{lm:points}
    Let $\mathcal X$ be an $\infty$-topos, and let $x^*:\mathcal X\to \Ss$ be a point, i.e. a left exact cocontinuous functor. There exists a cofiltered diagram $U_\bullet: I\to \mathcal X$ such that $x^*$ is naturally equivalent to $\colim_{I\op}\Gamma(U_i,-) := \colim_{I\op}\Map(U_i,-)$. 

    Furthermore, for any $\mathcal M\in \PrL$, the functor $\Fun^R(\mathcal X\op,\M)\simeq\mathcal X\otimes\M\xrightarrow{x^*\otimes\M}\M$ is equivalent to $F\mapsto \colim_I F(U_i)$. 
\end{lm}
\begin{proof}
    Let $\kappa$ be such that $\mathcal X$ is $\kappa$-compactly generated and such that $\kappa$-compact objects in $\mathcal X$ are closed under finite limits. 
    
    In this case, $x^*$ is left Kan extended from $\mathcal X^\kappa$, and the restriction to it is a finite limit-preserving functor $\mathcal X^\kappa\to \Ss$, i.e. a pro-object in $\mathcal X^\kappa$. The claim follows for $x^*_{\mid\mathcal X^\kappa}$, but now both $x^*$ and $\colim_{I\op}\Map(U_i,-)$ are left Kan extended from $\mathcal X^\kappa$, so the claim follows on all of $\mathcal X$.

    For the furthermore part, we write $\M$ as a localization of a presheaf category $\PP(\M_0)$ for some small $\M_0$. It follows that the map $\Fun^R(\mathcal X\op,\M)\to \M$ is given by $$\Fun^R(\mathcal X\op,\M)\to~\Fun^R(\mathcal X\op,\PP(\M_0))\to \PP(\M_0)\to \M$$ so we are reduced to the case of $\PP(\M_0)$. But now we can pull out $\Fun(\M_0\op,-)$ everywhere and reduce to the case of $\Ss$, where it is what we just proved. 
\end{proof}
\begin{proof}[Proof of \Cref{prop:cpctasshyper}]
    We first prove it for $\M$ compactly generated. In this case, we note that the collection of functors $\Map(m,-), m\in\M^\omega$ has the following properties:
    \begin{enumerate}
        \item They are jointly conservative;
        \item They preserve sheaves; 
        \item They preserve $x^*$.
    \end{enumerate}
    The last fact follows from \Cref{lm:points}. The first fact is the definition of compactly generated, and the second follows from the fact that they preserve arbitrary limits. 

From this, the claim follows. 

Finally, we note that the given property for $\M$ is closed under retracts in $\PrL$, so the claim follows for general compactly assembled categories.
\end{proof}

\begin{rmk}\label{rmk:hypercomplete}
    This fact is not true with no assumptions on $\M$. For example, fix a manifold $M$. It is not hard to prove that $U\mapsto \Sh(U)$ and $U\mapsto \Sh^{loc.cst}(U)$\footnote{The full subcategory spanned by locally constant sheaves.} are both sheaves on $M$ with values in $\PrL_{\omega_1}$, and that the canonical inclusion induces an equivalence on stalks, yet is not an equivalence.
    
    It seems harder to find an example where filtered colimits are left exact, but we conjecture that they exist. 
\end{rmk}
    
\section{Compact maps}\label{section:atomic}
In this section, we introduce the concept of compact maps, as well as higher cardinal variants of it. They allow for an \emph{internal} characterization and understanding of compact assembly, and therefore allow to manipulate dualizable categories in a way closer to the way we usually manipulate compactly generated categories. We learned of these ideas from Dustin Clausen (see \cite{youtubeDustin2}), but they also exist in the work of Johnstone--Joyal \cite{johnstonejoyal}. 

For our study of dualizability in general, they will mostly be helpful in proving \Cref{cor:cardboundV}, which is the key content of \Cref{thm:DblIsPrL}. In the companion paper \cite{companion}, we introduce a variant of the notion of compact maps, namely that of $\V$-atomic maps which is more adapted to the specific base $\V$.

The definition of compact maps is supposed to be a way of encoding, using only the map $f:x\to y$, the property that it has a ``compact image'', or perhaps factors through a compact object, even in the absence of compact objects. We generalize this to $\kappa$-compact maps, which will be technically convenient later on, even though the analogous notion of ``$\kappa$-compactly assembled category'' is uninteresting for uncountable $\kappa$, cf. \Cref{prop:cardboundunstable}. 
\subsection{Generalities}
\begin{defn}
    Let $\M$ be a category with $\kappa$-filtered colimits, and $f:x\to y$ a map in $\M$. $f$ is said to be a weakly-$\kappa$-compact morphism, or weakly-$\kappa$-compact map if for any $\kappa$-filtered diagram $z_\bullet : I\to \M$ and any map $h: y\to \colim_I z_i$, there exists an index $j\in I$ and a factorization as follows: 
    \[\begin{tikzcd}
	{z_j} & {\colim_Iz_i} \\
	x & y
	\arrow["f"', from=2-1, to=2-2]
	\arrow["h"', from=2-2, to=1-2]
	\arrow[from=1-1, to=1-2]
	\arrow[dashed, from=2-1, to=1-1]
\end{tikzcd}\]
When $\kappa=\omega$, we simply call them weakly compact maps.
\end{defn}
\begin{ex}\label{ex:O(X)}
    Let $X$ be a topological space, and $\Oo(X)$ its poset of opens - note that it has filtered colimits. 
    
    Let $U\subset V$ be two opens of $X$. If there is a compact subset $K$ of $X$ such that $U\subset K\subset V$, then the corresponding map in $\Oo(X)$ is a compact map - in this case, we write $U\Subset V$. This is perhaps the prototypical example. If $X$ is locally compact Hausdorff, the converse holds. 
\end{ex}
\begin{defn}
    Let $\M$ be a category with $\kappa$-filtered colimits, and $f:x\to y$ a map in $\M$. $f$ is said to be a $\kappa$-compact morphism, or $\kappa$-compact map if for any $\kappa$-filtered diagram $z_\bullet : I\to \M$, there exists a diagonal filler in the following commutative square: 
    \[\begin{tikzcd}
	{\colim_I\Map(y,z_i)} & {\Map(y,\colim_Iz_i)} \\
	{\colim_I\Map(x,z_i)} & {\Map(x,\colim_Iz_i)}
	\arrow[from=1-1, to=2-1]
	\arrow[from=2-1, to=2-2]
	\arrow[from=1-2, to=2-2]
	\arrow[from=1-1, to=1-2]
	\arrow[dashed, from=1-2, to=2-1]
\end{tikzcd}\]
\end{defn}
\begin{rmk}
    This idea of ``commuting symboles that are not supposed to commute, up to a map'' is one of the key ideas in the theory of compact maps and, as we shall see in \cite{companion}, atomic maps, and trace-class maps, and thus in studying compactly-assembled categories, dualizable $\V$-modules, and rigid $\V$-algebras. 
\end{rmk}
\begin{defn}
     Let $\M$ be a category with $\kappa$-filtered colimits, and $f:x\to y$ a map in $\M$. $f$ is said to be a strongly $\kappa$-compact morphism, or strongly $\kappa$-compact map if there exists a $\kappa$-filtered colimit preserving functor $C$ equipped with a natural transformation $\eta:C\to \Map(x,-)$ and a lift $\tilde f\in C(y)$ of $f\in\Map(x,y)$ along $\eta_y$.
\end{defn}
Before giving examples, let us point out that the terminology is reasonable:
\begin{lm}
    Any $\kappa$-compact map is weakly-$\kappa$-compact, and any strongly $\kappa$-compact map is $\kappa$-compact. 
\end{lm}
\begin{proof}
    The fact that $\kappa$-compact implies weakly $\kappa$-compact is obvious by taking $\pi_0$ of the diagram defining $\kappa$-compactness. 

    So let $f:x\to y$ be strongly $\kappa$-compact and let $\eta: C\to \Map(x,-), \tilde f\in C(y)$ be witnesses of that. Now let $z_\bullet: I\to \M$ be a $\kappa$-filtered diagram, and consider the following map: $$\Map(y,\colim_I z_i)\xrightarrow{\id\times\lceil\tilde f\rceil} \Map(y,\colim_I z_i)\times C(y)\to C(\colim_I z_i) \xleftarrow{\simeq} \colim_I C(z_i)\to \colim_I \Map(x,z_i)$$
    We claim that this is the desired diagram filler.

    Let us first note that because of the commutative square :
\[\begin{tikzcd}
	{\colim_I C(z_i)} & {\colim_I \Map(x,z_i)} \\
	{C(\colim_Iz_i)} & {\Map(x,\colim_Iz_i)}
	\arrow[from=2-1, to=2-2]
	\arrow["\simeq"', from=1-1, to=2-1]
	\arrow[from=1-2, to=2-2]
	\arrow[from=1-1, to=1-2]
\end{tikzcd}\]
the lower triangle commutes. The argument for the upper triangle is similar and involves the commutative diagram: 
\[\begin{tikzcd}
	{\colim_I\Map(y,z_i)} & {\colim_I\Map(y,z_i)\times C(y)} & {\colim_IC(z_i)} \\
	{\Map(y,\colim_Iz_i)} & {\Map(y,\colim_Iz_i)\times C(y)} & {C(\colim_Iz_i)}
	\arrow[from=1-1, to=2-1]
	\arrow[from=2-1, to=2-2]
	\arrow[from=1-1, to=1-2]
	\arrow[from=1-2, to=2-2]
	\arrow[from=1-2, to=1-3]
	\arrow["\simeq", from=1-3, to=2-3]
	\arrow[from=2-2, to=2-3]
\end{tikzcd}\]
Pasting these together also shows that the homotopy one gets is the correct one. 
\end{proof}
\begin{rmk}
    We will see in \Cref{ex:embtocheckcpct} that in a \emph{$\kappa$-compactly assembled} category, both implications can be reversed. We will also see that the weakly $\kappa$-compact maps are enough to work with in the criterion for $\kappa$-compact assembly, if one assumes ahead of time that $\kappa$-filtered colimits are left exact, cf. \Cref{thm:compactassembly}\footnote{See also \Cref{prop:cardboundunstable} for the case of $\kappa>\omega$. }. However, we do not expect the converse to hold in general.
\end{rmk}
\begin{ex}
    In the example of $\Oo(X)$, $X$ a topological space, we find that if $U\Subset V$, then the corresponding map is strongly compact. Indeed, for $U\subset K\subset V$, we can choose for $C\to \hom(U,-)$ the subfunctor with value $\pt$ if $K\subset W$ and $\emptyset$ else\footnote{Something like ``$\hom(K,-)$''.}. 
\end{ex}

\begin{ex}\label{ex:cpctinSh}
    The compact maps of \Cref{ex:O(X)} are preserved by the embedding $\Oo(X)\to~\Sh(X)$, at least when $X$ is locally compact Hausdorff. In fact, it is not hard to prove that they are strongly compact maps: let $U\subset K\subset V$, and let $i:K\to X$ denote the inclusion. On $\Sh(K)$, we have the functors $\Gamma(K,-),\Gamma(U,-)$, and a map $\Gamma(K,-)\to \Gamma(U,-)$, which gives us a map $\Gamma(K,i^*(-))\to \Gamma(U,i^*(-))$ on $\Sh(X)$. Now, restricting to $K$, and evaluating on $U$ is the same as restricting to $U$ and taking global sections, which is the same as simply evaluating at $U$ from the beginning, so that $\Gamma(U,i^*(-))\simeq \Map_{\Sh(X)}(U,-)$. Finally, the map $U\to V$ in $\Sh(X)$ is in the image of $\Gamma(K,K)= \Gamma(K,i^*V)\to \Gamma(U,i^*(V))=\Map(U,V)$, and by \cite[Remark 7.3.1.5, Theorem 7.3.1.16]{HTT}, $\Gamma(K,i^*(-))$ preserves filtered colimits. 
\end{ex}
\begin{ex}\label{ex:cptobject}
    An object $x$ is $\kappa$-compact if and only if $\id_x$ is a $\kappa$-compact map, and in this case $\id_x$ is also strongly $\kappa$-compact. If $\id_x$ is weakly $\kappa$-compact and if $\M$ has finite limits which commute with $\kappa$-filtered colimits, then $x$ is $\kappa$-compact. We prove this converse below, as a prelude to \Cref{lm:basicnuc}.
\end{ex}
\begin{lm}
    Let $\M$ be a category with finite limits and left exact $\kappa$-filtered colimits, and let $x\in\M$ be such that $\id_x$ is weakly $\kappa$-compact, i.e., for all $\kappa$-filtered systems $z_\bullet:I\to\M$, the map $\pi_0\colim_I\Map(x,z_i)\to \pi_0\Map(x,\colim_Iz_i)$ is surjective. In this case, $x$ is $\kappa$-compact.
\end{lm}
\begin{proof}
    We first prove that the map on $\pi_0$ is an isomorphism, i.e. that it is injective. Let $x\to z_i, x\to z_j$ be two maps such that the composites $x\to z_i\to z, x\to z_j\to z$ agree, where $z=\colim_I z_\bullet$. In this case we get a map $x\to z_i\times_z z_j\simeq \colim_{I_{i,j/}}z_i\times_{z_k}z_j$ - here we use that $I_{i,j/}$ is $\kappa$-filtered, that the map $I_{i,j/}\to I$ is cofinal, and that $\kappa$-filtered colimits in $\M$ commute with pullbacks. 

    By weak $\kappa$-compactness, we find a $k$ and a lift $x\to z_i\times_{z_k}z_j$, so that $$x\to z_i\to z_k, x\to z_j\to z_k$$ are already equivalent, and thus represent the same point in $\pi_0\colim_I \Map(x,z_\bullet)$, thus proving injectivity. 

    Next, we observe that $\colim_I\Map(x,z_i) = \Map_{\Ind(\M)}(y(x),\colim_Iy(z_\bullet))$ where $y:~\M\to~\Ind(\M)$ is the Yoneda embedding, and that our map comes from applying the functor $\colim:~\Ind(\M)\to~\M$. Furthermore, $\Ind(\M)$ has finite limits and $\colim$ commutes with them. We are thus in the situation of the next lemma, and are thereby able to conclude. 
\end{proof}
We learned the following lemma from \cite{KNSh}.  
\begin{lm}\label{lm:pi0enough}
    Let $f:C\to D$ be a left exact functor between categories with finite limits. Let $x\in C$ be such that for all $c\in C$, $\pi_0\Map_C(x,c)\to \pi_0\Map_D(f(x),f(c))$ is an isomorphism. In that case, for all $c\in C$, $\Map_C(x,c)\to \Map_D(f(x),f(c))$ is an equivalence.  
\end{lm}
\begin{proof}
    We already know it is a $\pi_0$-isomorphism, so it suffices to show that for all basepoints and all $n\geq 1$, it induces an isomorphism on $\pi_n$ at that basepoint. We prove this by induction - precisely, we prove by induction on $n$ the statement ``for all $c\in C$ and all $\alpha\in \Map_C(x,c)$, the map $\pi_n(\Map_C(x,c),\alpha)\to\pi_n(\Map_D(f(x),f(c)),f(\alpha))$ is an isomorphism''. 

    So suppose it is true at $n$, and let us prove it at $n+1$. For this, fix $c$ and $\alpha$ as above, and let $c' := \mathrm{eq}(x\rightrightarrows c)$ (note that both arrows are $\alpha$!). We then have a fiber sequence at $\id_x$: $$\Omega_\alpha\Map(x,c)\to \Map_C(x,c')\to \Map_C(x,x)$$ and the rightmost map admits a section. Thus, the long exact sequence on homotopy groups splits into short exact sequences $$0\to \pi_{n+1}(\Map_C(x,c),\alpha)\to \pi_n(\Map_C(x,c'),\tilde \alpha)\to \pi_n(\Map_C(x,x),\id_x)\to 0$$

    Comparing this short exact sequence to the one obtained after applying $f$ and using the induction hypothesis gives the result at $n+1$. 
    \begin{rmk}
        In the case $n=0$, this is an exact sequence of pointed sets rather than of groups, so we need to argue a bit more. Nonetheless, it is generally true that if $F\to Y\to X$ is a fiber sequence at $x\in X$, and $Y\to X$ admits a section, then $\pi_0(F)\to \pi_0(Y)$ is exactly the pointed kernel of $\pi_0(Y)\to \pi_0(X)$. This is enough for our purposes.
    \end{rmk}
\end{proof}
We now point out a slight variant of this lemma which will be convenient later on:
\begin{lm}\label{lm:indprovariant}
     Let $f:C\to D$ be a left exact functor between categories with finite limits. For $x\in\Ind(C)$, consider the functor $\Map(x,-): C\to \Pro(\Ss)$. 

Assume that $x\in \Ind(C)$ is such that for all $c\in C$, $\pi_0\Map_C(x,c)\to \pi_0\Map_D(f(x),f(c))$ is a pro-isomorphism of sets. In that case, for all $c\in C$, all $\alpha \in \Map(x,c)$ and all $n\in\mathbb N$, 
$\pi_n(\Map_C(x,c),\alpha)\to \pi_n(\Map_D(f(x),f(c)),f(\alpha))$ is a pro-isomorphism.

Here, ``$\alpha\in \Map(x,c)$'' means equivalently a map $\pt\xrightarrow{\alpha}\Map(x,c)$ or a point in the image of $\Map(x,c)$ under $\lim:\Pro(\Ss)\to\Ss$. 
\end{lm}
\begin{rmk}
    This lemma is not a special case of the previous one ``applied to $\Ind(C)$'', because our assumption is of a different nature: the limit of the $\pi_0$'s is not the $\pi_0$ of the limit. 
\end{rmk}
\begin{proof}
    The proof works exactly the same, noting that at every step we could have added the word ``pro-''. In particular, note that in those short exact sequences, the only thing we really need is for the leftmost term to be the kernel of the right hand morphism, which is levelwise true as the maps are levelwise split, and levelwise fiber sequences. 
\end{proof}
We leave the following examples as instructive exercises for the reader:
\begin{ex}\label{ex:factorcompact}
    The collection of (strongly, weakly) $\kappa$-compact maps is a $2$-sided ideal : if $f$ is (weakly, strongly) $\kappa$-compact, then for any $g,h$ for which it makes sense, $gfh$ is also $\kappa$-compact.

    In particular, any map factoring through a $\kappa$-compact object is strongly $\kappa$-compact.
\end{ex}
\begin{ex}\label{ex:conversefactorcpt}
    In a $\kappa$-compactly generated \category, any weakly $\kappa$-compact map $f$ factors through a $\kappa$-compact: it suffices to write the target of $f$ as a $\kappa$-filtered colimit of $\kappa$-compacts to prove this. It follows from this and the previous example that in this case, weakly $\kappa$-compact maps are strongly $\kappa$-compact. 
\end{ex}
\begin{ex}\label{ex:reflectcpct}
    If $f: \M\to \N$ preserves $\kappa$-filtered colimits and is fully faithful, then it \emph{reflects} (weakly, strongly) $\kappa$-compact maps. 
\end{ex}
\begin{ex}\label{ex:preservecpct}
    If $f:\M\to \N$ has a $\kappa$-filtered colimit preserving right adjoint $f^R$, then $f$ preserves (weakly, strongly) $\kappa$-compact maps. The proof is the same as the proof that $f$ preserves $\kappa$-compact objects in this situation.
\end{ex}
\begin{ex}\label{ex:embtocheckcpct}
    Combining \Cref{ex:factorcompact,,ex:conversefactorcpt,,ex:reflectcpct,,ex:preservecpct}, we see that if $i:\M\to \N$ is a fully faithful embedding with a $\kappa$-filtered colimit preserving right adjoint, and if $\N$ is $\kappa$-compactly generated, then all variants of $\kappa$-compact maps in $\M$ are exactly those maps $f$ such that $i(f)$ factors through a $\kappa$-compact object in $\N$, and in particular they all agree in $\M$. 
    
    That such an $i$ exists is equivalent to $\M$ being $\kappa$-compactly assembled in the sense of \Cref{defn:compactass} - it follows that for a $\kappa$-compactly assembled category $\M$, all notions of $\kappa$-compact maps agree.
\end{ex}
We note that in the previous example, for a $\kappa$-compactly assembled $\M$, using $\hat y$, one can find a canonical way to factor a $\kappa$-compact morphism through a $\kappa$-compact object. Namely:
\begin{lm}
    Let $\M$ be a compactly assembled category with $\hat y : \M\to \Ind(\M)$ a left adjoint to the canonical functor $\Ind(\M)\to \M$. Given a morphism $f:x\to z$, the following are equivalent: 
    \begin{enumerate}
        \item $f$ is compact;
        \item $y(f)$ factors through the canonical map $\hat y(z)\to y(z)$; 
        \item $\hat y(f)$ factors through the canonical map $\hat y(x)\to y(x)$. 
    \end{enumerate}
\end{lm}
\begin{proof}
We first prove that 1. is equivalent to 3.. Suppose 1. holds. In this case by \Cref{ex:preservecpct} $\hat y(f)$ is also a compact map thus, as $\Ind(\M)$ is compactly generated,  it factors through a compact object $y(m)$. But now any map $\hat y(x)\to y(m)$ factors through $\hat y(x)\to y(x)$ by adjunction. Thus we find 3. 

Conversely, if 3. holds then by \Cref{ex:factorcompact}, $\hat y(f)$ is compact and thus by \Cref{ex:reflectcpct} $f$ is compact too. 

We now point out that 2. and 3. are equivalent for formal reasons because of the chain of adjunctions $\hat y\dashv \colim\dashv y$: in the square \[\begin{tikzcd}
	{\hat y(x)} & {\hat y(z)} \\
	{y(x)} & {y(z)}
	\arrow[from=2-1, to=2-2]
	\arrow[from=1-1, to=2-1]
	\arrow[from=1-1, to=1-2]
	\arrow[from=1-2, to=2-2]
	\arrow[dashed, from=2-1, to=1-2]
\end{tikzcd}\] 
the dotted lift makes the lower triangle commute if and only if it makes the upper triangle commute by adjunction. 
\end{proof}
This example will be generalized and made more precise in \cite[Example 2.14]{companion}, where we will see that a reasonable name for $\Map(y(x),\hat y(z))$ would be the ``space of compact maps from $x$ to $z$'' (making ``$f$ is compact'' an extra structure on $f$ rather than a property).  
\begin{rmk}
    The notion of compact maps will be compared in \cite{companion} to the special case $\V=\Sp$ of a general notion of $\V$-atomic maps. However, in the same way that ($\kappa$-)compact objects are relevant to presentability even for other bases, ($\kappa$-)compact maps will also be relevant, see for example the proof of \Cref{prop:cardboundunstable} (which is involved in \Cref{cor:cardboundV}).
\end{rmk}
\subsection{Compact exhaustions}
In this section, we prove one of the key features of $\kappa$-compact maps: ``in the limit'' they behave like compact objects. We make this precise via the following terminology. 
\begin{defn}
    Let $\M$ be a category with filtered colimits, and let $x_\bullet: \mathbb N\to \M$ be a diagram. Suppose that each transition map $x_n\to x_{n+1}$ is $\kappa$-compact. 

    We say that $x_\bullet$ is a $\kappa$-compact exhaustion of $x:=\colim_\mathbb Nx_n$, and for $x\in\M$, we say that $x$ is $\kappa$-compactly exhaustible if there exists a $\kappa$-compact exhaustion of it. 

    We define weak $\kappa$-compact exhaustion and weakly $\kappa$-compactly exhaustible in the obvious way. 
\end{defn}
The main result of this subsection is: 
\begin{lm}\label{lm:basicnuc}
    Let $\M$ be a category admitting filtered colimits, and let $c:\Ind_\kappa(\M)\to \M$ be the canonical functor. Let $x_\bullet:\mathbb N\to \M$ be a weakly $\kappa$-compact exhaustion of $x=\colim_\mathbb Nx_n$. 
    
    In this case, for any $\kappa$-ind-object $y=\{y_i\}_{i\in I}$, the following map is an equivalence: $$\Map(\{x_n\}, y)\to \Map(\colim_\mathbb N x_n, p(y)) = \Map(x,\colim_I y_i)$$
\end{lm}

    We give ``two proofs'' of this lemma. One of them is a proof of the lemma as stated, but before that we give a proof of a weaker lemma, where we assume that we have a $\kappa$-compact exhaustion of $x$, rather than a weakly $\kappa$-compact exhaustion. We start by the latter, as it is more natural and also much simpler, and then explain the necessary modifications for the former.  
\begin{proof}[Proof of the weaker lemma]
    Fix $x_\bullet: \mathbb N\to \M$ a sequential diagram where each of the transition maps is $\kappa$-compact.

    Note that the map in question is the canonical map $$\lim_\mathbb N\colim_I\Map(x_n,y_i)\to\lim_\mathbb N\Map(x_n,\colim_I y_i)$$
    In other words, it is $\lim_\mathbb N$ of the map $\colim_I\Map(x_n,y_i)\to\Map(x_n,\colim_I y_i)$. We prove a stronger claim than the claim that this is an equivalence on limits, namely we prove that this is a pro-equivalence $\{\colim_I\Map(x_n,y_i)\}_n\to \{\Map(x_n,\colim_I y_i)\}_n$.

    The general statement is that in an \category{} $C$, if $A_\bullet, B_\bullet: \mathbb N\op\to C$ are two inverse sequential systems in $C$, and $f: A_\bullet\to B_\bullet$ is a natural transformation such that for each $n$, there exists a dashed arrow making the two triangles below commute, then $f$ induces a pro-equivalence:
\[\begin{tikzcd}
	{A_{n+1}} & {B_{n+1}} \\
	{A_n} & {B_n}
	\arrow[from=1-1, to=2-1]
	\arrow[from=2-1, to=2-2]
	\arrow[from=1-2, to=2-2]
	\arrow[from=1-1, to=1-2]
	\arrow[dashed, from=1-2, to=2-1]
\end{tikzcd}\]
This is an instructive exercise in cofinality which we leave to the reader. From this, the claim follows immediately of course, from the definition of a $\kappa$-compact map. 
\end{proof}
The proof of the actual lemma looks similar, except that we cannot actually obtain lifts at the space level as above, because the definition of a weakly $\kappa$-compact map is too ``$\pi_0$-ish''. To prove it, we use two reductions: firstly, we use \Cref{lm:pi0enough} to reduce to a $\pi_0$-claim, secondly, we use some kind of unstable Milnor sequence to prove this $\pi_0$-claim. 
\begin{proof}[Proof of the lemma]
       First note that applying \Cref{lm:pi0enough} to $c:\Ind(\M)\to \M$, we can reduce to proving that $\pi_0(\lim_\mathbb N \colim_I \Map(x_n, y_i))\to \pi_0(\lim_\mathbb N\Map(x_n,\colim_I y_i))$ is an isomorphism. 
    
    Second, by \Cref{lm:unstableMilnor}, it suffices to prove that \begin{itemize}
        \item[(a)]$\pi_0(\colim_I\Map(x_n,y_i))\to \pi_0(\Map(x_n,\colim_Iy_i))$ is a pro-isomorphism of sets,
        \item[(b)] for any $f\in \lim_\mathbb N\colim_I\Map(x_n,y_i)$, the induced map $\pi_1(\colim_I \Map(x_n,y_i),f_n)\to~\pi_1(\Map(x_n,\colim_I y_i),f_n)$ is a pro-isomorphism of groups. 
    \end{itemize} 
   \begin{rmk}
       If $\M$ is stable, then (b) follows from (a) by using $\Omega$. Thus the complications related to (b) below are only relevant in the unstable case. 
   \end{rmk}

    To prove (a), we use the same general statement as in the previous proof, except that we replace $A_{n+1}$ (resp. $B_{n+1}$) by $A_{n+2}$ (resp. $B_{n+2}$), namely we prove that there exist dotted lifts in the following diagrams: 
    
\[\begin{tikzcd}
	{\colim_I \pi_0\Map(x_{n+2},y_i)} & {\pi_0\Map(x_{n+2},\colim_I y_i)} \\
	{\colim_I \pi_0\Map(x_n,y_i)} & {\pi_0\Map(x_n,\colim_I y_i)}
	\arrow[from=1-1, to=1-2]
	\arrow[from=2-1, to=2-2]
	\arrow[from=1-1, to=2-1]
	\arrow[from=1-2, to=2-2]
	\arrow[dashed, from=1-2, to=2-1]
\end{tikzcd}\]

This is good enough because the pro-system $\{A_{2n}\}$ is naturally isomorphic to the pro-system $\{A_n\}$.

So let $f: x_{n+2}\to \colim_I y_i$ be any map. By weak compactness of $x_{n+1}\to x_{n+2}$, the restriction to $x_{n+1}$ factors through some $y_j$. We want to pove that, after restricting to $x_n$ any two factorizations of $x_{n+1}\to x_{n+2}\to \colim_I y_i$ give the same result in $\colim_I\pi_0\Map(x_n,y_i)$. 

The point is that if we have two factorizations, say $f_0$ and $f_1$, by filteredness of $I$, we may assume that they factor through the same $y_j$, and then we get a map $x_{n+1}\to y_j\times_y y_j$. As filtered colimits are left exact in $\M$, this target is the filtered colimit $\colim_{I_{j/}} y_j\times_{y_i}y_j$, so that the compactness of $x_n\to x_{n+1}$ guarantees  that the map $x_n\to x_{n+1}\to y_j\times_y y_j$ factors through some $y_j\times_{y_i}y_j$. Therefore, the composites $x_n\to x_{n+1}\xrightarrow{f_0,f_1}y_i$ agree. Thus we get a well-defined map of sets $\pi_0\Map(x_{n+2},\colim_I y_i)\to \colim_I\pi_0\Map(x_n,y_i)$. The uniqueness proved here shows that the upper triangle commutes, and clearly the bottom triangle in the above square commutes, by construction. 

Thus the pro-objects are isomorphic, as claimed. Now, the claim about $\pi_1$ simply follows from \Cref{lm:indprovariant}.
\end{proof}
A key corollary of this is the following:
\begin{cor}\label{cor:basicnuccpct}
Let $\M$ be a category with filtered colimits. 
    Let $x$ be a $\kappa$-compactly exhaustible object of $\M$. In this case, $x$ is $\max(\kappa,\omega_1)$-compact.

    The same holds for weakly $\kappa$-compactly exhaustible objects if $\M$ has finite limits and $\kappa$-filtered colimits in $\M$ are left exact.
\end{cor}
\begin{proof}
    \Cref{lm:basicnuc} shows that for a $\kappa$-filtered diagram $y_\bullet$, $\Map_\M(x,\colim_I y_i)\simeq \Map_{\Ind_\kappa(\M)}(\{x_n\}, \{y_i\})$ where $x_\bullet$ is some $\kappa$-compact exhaustion of $x$. 

In turn, the latter is $\lim_\mathbb N\colim_I\Map(x_n,y_i)$. If $I$ is also $\omega_1$-filtered (so, in total, $\max(\kappa,\omega_1)$-filtered), the colimit and the limit commute so that $$\Map(x,\colim_Iy_i)\simeq \colim_I \lim_\mathbb N\Map(x_n,y_i)\simeq \colim_I \Map(x,y_i)$$ as required. 
\end{proof}
\subsection{$\kappa\geq \omega_1$}
We have already noted that in the case of $\kappa$-compactly generated categories, $\kappa$-compact maps were a bit boring: they are exactly the maps factoring through a $\kappa$-compact object. We prove that for uncountable $\kappa$, the same holds for $\kappa$-compactly assembled categories - in fact, we prove that for uncountable $\kappa$, $\kappa$-compactly assembled implies $\kappa$-compactly generated. This serves two purposes: firstly, this will imply a greatly practical cardinal bound which is the key ingredient in proving the presentability of $\Dbl{\V}$, and secondly, this indicates that compact assembly is really only an interesting phenomenon at $\omega$. 

Alternatively, one can view the following proof as an indication of the usefulness of the notion of $\kappa$-compact maps, as well as a demonstration of some of the tricks in this business.

\begin{prop}\label{prop:cardboundunstable}
Let $i: \M\to \N$ be a fully faithful functor admitting a right adjoint $i^R$, where $\N$ is $\kappa$-compactly generated,  for some \emph{uncountable} cardinal $\kappa$. Assume $\M,\N$ admit sequential colimits. If the right adjoint $i^R$ preserves $\kappa$-filtered colimits, then $\M$ is also $\kappa$-compactly generated. 
\end{prop}

The proof is much simpler in the stable case, where it really is a higher cardinal version of the generalized telescope conjecture which turns out to be correct. 

For the reader to get an idea of the argument, we begin with this simpler case:
\begin{prop}\label{prop:telescope}
    Let $\kappa$ be an \emph{uncountable} cardinal, and let $L:\D\to\E$ be a $\kappa$-compact preserving localization between $\kappa$-compactly generated stable \categories. Its kernel $\ker(L)$ is $\kappa$-compactly generated. 
\end{prop}
The proof of this proposition is relatively simple: one tries to prove the statement for $\kappa=\omega$ and, failing to prove that (wrong) statement, one then simply observes that it works for higher cardinals.
\begin{proof}
The right adjoint of $\D\to \E$ preserves $\kappa$-filtered colimits, hence so does the right adjoint to the inclusion of the kernel by \Cref{cor:colimleftright}. Therefore the inclusion of the kernel preserves $\kappa$-compacts. It also reflects $\kappa$-compactness because it is fully faithful and colimit preserving. 

By \Cref{lm:kappacpctgen}, we are thus trying to prove that any nonzero $X\in\ker(L)$ receives a nonzero map from some $Y\in\ker(L)\cap \D^\kappa = \ker(L)^\kappa$. So let $X\in \ker(L)$ be nonzero. We can write $X = \colim_I X_i$ as a $\kappa$-filtered colimit of $\kappa$-compact objects of $\D$. Of course, the $X_i$'s themselves have no reason to be in $\ker(L)$. 

We first find an $i_0$ such that the canonical map $X_{i_0}\to X$ is nonzero, which exists by \Cref{prop:nonzerocolim}. 

Applying $L$ to it, and because $L(X_{i_0})$ is $\kappa$-compact by assumption, there is some $i_1 > i_0$ such that $L(X_{i_0})\to L(X_{i_1})$ is already zero. Start again with $X_{i_1}$ and find some $i_2$, and iterate to get a sequence $i: \mathbb N\to I$ with $i(0) = i_0$ such that for every $n$, the map $L(X_{i(n)}\to X_{i(n+1)})$ is zero.  

Take $Y := \colim_\mathbb N X_{i(n)}$. Then, because $L$ preserves colimits, and every map in the diagram of the $L(X_{i(n)})$'s is $0$, we find $LY = 0$, i.e. $Y\in \ker(L)$. Furthermore, $\mathbb N$ is $\kappa$-small\footnote{This is where we crucially use that $\kappa > \omega$. }, and hence $Y$ is $\kappa$-compact, as a $\kappa$-small colimit of $\kappa$-compacts. In other words, $Y\in \ker(L)^\kappa$. Finally, we have a factorization $X_{i_0}\to Y\to X$, and $X_{i_0}\to X$ is nonzero, so $Y\to X$ is also nonzero. 
\end{proof}
We recall that the kernel $\ker(L)$ is dualizable if $\kappa=\omega$, but it need not be compactly generated. On the other hand, the above proof affords the following interpretation of compact maps in $\ker(L)$:
\begin{cor}
   Let $L:\D\to\E$ be a compact preserving localization between compactly generated stable \categories, and let $i^R: \D\to \ker(L)$ be the right adjoint to the inclusion. Let $f:x\to y$ be a map in $\D$ such that $L(f) = 0$, and such that $x$ is compact. In that case, $r(f)$ is compact in $\ker(L)$. 
\end{cor}
\begin{proof}
    Let $i$ denote the inclusion and $L^R$ be the right adjoint to the inclusion. In $\D$, we have the following diagram:
    \[\begin{tikzcd}
	{ii^R(x)} & x & {L^RL(x)} \\
	{ii^R(y)} & y & {L^RL(y)}
	\arrow[from=2-1, to=2-2]
	\arrow[from=2-2, to=2-3]
	\arrow[from=1-2, to=2-2]
	\arrow[from=1-3, to=2-3]
	\arrow[from=1-1, to=2-1]
	\arrow[from=1-1, to=1-2]
	\arrow[from=1-2, to=1-3]
	\arrow[dashed, from=1-2, to=2-1]
\end{tikzcd}\]
where there is a dotted map making the bottom right triangle commute because the lower row is a fiber sequence and the right-most map is null by assumption. By adjunction, the dotted map also makes the top left hand triangle commute and thus $ii^R(x)\to ii^R(y)$ factors through $x$ and is therefore compact. By \Cref{ex:reflectcpct}, $i^R(f)$ is also compact. 
\end{proof}
\begin{cor}
    Let $i:\mathcal K\to\D$ be a fully faithful map in $\PrL_{\st}$, such that $i^R$ preserves $\kappa$-filtered colimits and assume $\D$ is $\kappa$-compactly generated. If $\kappa$ is uncountable, then $\mathcal K$ is $\kappa$-compactly generated. 
\end{cor}
\begin{proof}
    Consider the localization $\E:=\D/\mathcal K$, with localization functor $L:\D\to \E$. By standard arguments, the fact that $i^R$ preserves $\kappa$-filtered colimits implies the same for the right adjoint $L^R$, and in particular $L$ preserves $\kappa$-compact objects. It follows that $\E$ is $\kappa$-compactly generated, and thus $L:\D\to\E$ fits in the situation of the previous proposition, so that its kernel $\mathcal K$ is $\kappa$-compactly generated, as was to prove. 
\end{proof}
The point of the proof of this corollary is that (at least in the presentable case) stability allows us to rephrase the problem as a problem of limits in $\PrL_\kappa$. In fact, the proof of \Cref{prop:telescope} is a special case of the more general following statement:
\begin{prop}\label{prop:smalllim}
    For any uncountable cardinal $\kappa$, the inclusion $\PrL_\kappa\to \PrL$ preserves $\kappa$-small limits. 
\end{prop}
\begin{proof}
    It suffices to prove so for pullbacks and $\kappa$-small products. The case of $\kappa$-small products does not require the uncountability of $\kappa$: let $I$ be a $\kappa$-small set, and $C_i$ a family of small categories. The canonical map $\prod_IC_i\to \prod_I\Ind_\kappa(C_i)$ lands in $\kappa$-compact objects: indeed, $\Map_{\prod_I}((x_i), -) = \prod_I\Map(x_i,-)$ is a $\kappa$-small limit which thus commutes with $\kappa$-filtered colimits in spaces. It follows that the induced functor $\Ind_\kappa(\prod_IC_i)\to \prod_I\Ind_\kappa(C_i)$ is fully faithful. 
    
    Now fix an object $(x_i)$ in the target, and for each $i$, write $x_i=\colim_{K_i}x_i^k$ for a $\kappa$-filtered poset $K_i$ and a diagram $x_i^\bullet: K_i\to C_i$. Note that $\prod_I K_i$ is also $\kappa$-filtered, and each projection map $\prod_I K_i\to K_i$ is cofinal, so that $(x_i)= \colim_{(k_i)\in \prod_I K_i} (x_i^{k_i})$ in $\prod_I \Ind_\kappa(C_i)$ proves essential surjectivity.

    We now move on to pullbacks, where uncountability will be needed. Let $C\xrightarrow{p} D\xleftarrow{q} E$ be a diagram in $\PrL_\kappa$. The inclusion $C^\kappa\times_{D^\kappa}E^\kappa\to C\times_D E$ clearly lands in $\kappa$-compact objects, so we have a fully faithful embedding $\Ind_\kappa(C^\kappa\times_{D^\kappa}E^\kappa)\to C\times_D E$ and we need to prove it is essentially surjective. 

    We first describe the informal idea, and make it precise later: take $(c,e,\alpha:~p(c)\simeq~q(e))~\in~C\times_D~E$, and we want to approximate it by similar objects where $c\in C^\kappa, e\in E^\kappa$. Write $c=\colim_I c_i, e\simeq \colim_J e_j$ with $c_i\in C^\kappa,e_j\in E^\kappa$ and $I,J$ $\kappa$-filtered posets. 

    Fix any $i_0\in I$ - as $p(c_{i_0})$ is $\kappa$-compact, the map $p(c_{i_0})\to p(c)\simeq q(e)$ factors through some $q(e_{j_0})$. Similarly, the map $q(e_{j_0})\to q(e_j)\simeq p(c)$ factors through some later $p(c_{i_1})$, and we can now iterate this game to find a ``ladder'': 
    \[\begin{tikzcd}
	{p(c_{i_0})} & {q(e_{j_0})} \\
	{p(c_{i_1})} & {q(e_{j_1})} \\
	{p(c_{i_2})} & {q(e_{j_2})} \\
	\dots & \dots
	\arrow[from=1-1, to=1-2]
	\arrow[from=1-1, to=2-1]
	\arrow[from=1-2, to=2-1]
	\arrow[from=1-2, to=2-2]
	\arrow[from=2-1, to=2-2]
	\arrow[from=2-1, to=3-1]
	\arrow[from=2-2, to=3-1]
	\arrow[from=3-1, to=3-2]
	\arrow[from=2-2, to=3-2]
	\arrow[from=3-1, to=4-1]
	\arrow[from=3-2, to=4-2]
	\arrow[from=3-2, to=4-1]
	\arrow[from=4-1, to=4-2]
\end{tikzcd}\]
whose colimit gives us some $c_\infty,e_\infty$ with an equivalence $p(c_\infty)\simeq q(e_\infty)$ compatible with the map $c_\infty\to c, e_\infty\to e$ and the equivalence $p(c)\simeq q(e)$. As $\kappa$ was uncountable and we only took a countable colimit, $c_\infty\in C^\kappa,e_\infty\in E^\kappa$. Furthermore, we started with an arbitrary $i_0\in I$, so we can ``exhaust'' $(c,e,\alpha)$ with these $c_\infty,e_\infty$. 

We now make the idea precise. Firstly, we note that we may assume without loss of generality that $I=J$, and that the equivalence $p(c)\simeq q(e)$ is induced by a map of diagrams $\alpha_i : p(c_i)\to q(e_i)$, as $p(c_i)$ and $q(e_j)$ are all $\kappa$-compact and $I$ is $\kappa$-filtered. 

We then left Kan extend both diagrams $c_\bullet$ and $e_\bullet$ from $I$ to $Down_\kappa(I)$, the set of $\kappa$-small downwards closed filtered subsets of $I$. The resulting diagrams have the same colimit, so we may assume without loss of generality that $I$ admits $\kappa$-small filtered colimits and that both $c_\bullet$ and $e_\bullet$ preserve them.

In that situation, the previous argument essentially shows that the sub-poset $I^{eq}$ of $I$ spanned by those $i$'s for which $\alpha_i: p(c_i)\to q(e_i)$ is an equivalence is unbounded in $I$, and hence cofinal. Thus $(c,e,\alpha)\simeq\colim_{I^{eq}} (c_i, e_i,\alpha_i) $, as needed. 
 \end{proof}
\begin{rmk}
    In the above, we used filtered posets as we find them easier to manipulate than arbitrary filtered categories. One could give an alternative proof of the above by essentially following the approach from \cite[Section 5.4.6]{HTT}, keeping more careful track of cardinalities. 
\end{rmk}

In the case of \Cref{prop:cardboundunstable}, we are unable to reduce to a statement about limits in $\PrL_\kappa$ (or in $\kappa$-accessible categories), even if we assume that $\M,\N$ are presentable, so we really need an extra proof, although it will look similar. 
\begin{proof}[Proof of \Cref{prop:cardboundunstable}]

    We first observe the following general phenomenon: if $i(m)\simeq~\colim_J n_j$ in $\N$, where $J$ is $\kappa$-filtered, then in fact the co-unit maps $ii^R(n_j)\to n_j$ induces equivalences $\colim_J ii^R(n_j)\simeq \colim_J~n_j \simeq~i(m)$\footnote{This trick is very useful and will be used again later.}.

    Indeed, it is not hard to verify that the total composite $\colim_J~ii^R(n_j)\to~\colim_J n_j\to~i(m)$ is $i(-)$ applied to $\colim_J i^R(n_j)\to m$. The latter is $i^R(-)$ applied to our original equivalence, using that $i^R$ preserves $\kappa$-filtered colimits, so that it is indeed an equivalence. By 2-out-of-3, all the maps involved in the composite are thus equivalences. 

    Having observed this, fix $m\in \M$, and write $i(m)\simeq \colim_J n_j$ for some $\kappa$-filtered diagram $n_\bullet: J\to\N^\kappa$, which exists by assumption. We thus have $m\simeq \colim_J i^R(n_j)$.

    Let $Down_\omega(J)$ denote the poset of downwards closed filtered subsets of $J$ that have a countable cofinal subset - note that this is a filtered subset and we have an inclusion $J\to Down_\omega(J), j\mapsto \{\leq j\}$. Left Kan extend the diagram $i^R(n_\bullet)$ to $Down_\omega(J)$ to get a diagram $\tilde n_\bullet: Down_\omega(J)\to \M$. 

    We claim that the subposet $P\subset Down_\omega(J)$ spanned by those $F$'s such that $\tilde n_F$ is $\kappa$-compactly exhaustible is cofinal. Once we know this, we will be done: it will follow that $\M$ is generated under colimits\footnote{Surprisingly, the colimit we produce is not a priori $\kappa$-filtered. This is not a problem to obtain $\kappa$-compact generation.} by $\kappa$-compactly exhaustible objects, and hence, by \Cref{cor:basicnuccpct}, by $\kappa$-compact objects. 

To prove this cofinality, since $Down_\omega(J)$ is a filtered poset, it suffices to produce, for every $F_0$, a bigger $F$ with the desired property. Now $F_0\subset J$ has a countable cofinal subset, and $J$ is $\kappa$-filtered, so we can find a $j_0\in J$ which is an upper bound for $F_0$. Thus we may assume $F_0= \{\leq j_0\}$. 

In this case, $\tilde n_{j_0} = i^R(n_{j_0})$. Now we note that since $n_{j_0}$ is $\kappa$-compact, the map $n_{j_0}\to~m\simeq~\colim_J~ii^R(n_k)$ factors through some $ii^R(n_{j'_1})$. Using again $\kappa$-compactness, there is a $j_1\geq j'_1,j_0$ such that the composite $n_{j_0}\to ii^R(n_{j_1})\to n_{j_1}$ is the map induced by $n_\bullet$. Thus, for any $j_0$, we find a $j_1\geq j_0$ such that $n_{j_0}\to n_{j_1}$ factors through the counit $ii^R(n_{j_1})\to n_{j_1}$. 

The claim is now that this forces $i^R(n_{j_0})\to i^R(n_{j_1})$ to be $\kappa$-compact. Once we prove this, we will be done: indeed, we can then iterate from $j_1$ to find a $j_2$, and then a $j_n$ and so on, and we can let $F$ be the downwards closure $\{j_n, n\in\mathbb N\}$. For this $F$, we will have $\tilde n_F= \colim_{j\in F} i^R(n_j)= \colim_n i^R(n_{j_n})$, which will be $\kappa$-compactly exhaustible, and by construction $\{\leq j_0\}\subset F$.  

Now this claim is in turn easy to prove : in the following diagram, if there exists a dotted lift making the lower right triangle commute 
\[\begin{tikzcd}
	{ii^R(n_0)} & {n_0} \\
	{ii^R(n_1)} & {n_1}
	\arrow["f", from=1-2, to=2-2]
	\arrow["{\epsilon_{n_0}}", from=1-1, to=1-2]
	\arrow["{ii^R(f)}"', from=1-1, to=2-1]
	\arrow["{\epsilon_{n_1}}"', from=2-1, to=2-2]
	\arrow[dotted, from=1-2, to=2-1]
\end{tikzcd}\]
then the upper left triangle automatically commutes by adjunction: maps $ii^R(n_0)\to ii^R(n_1)$ are entirely determined by the composition $ii^R(n_0)\to ii^R(n_1)\to n_1$. Thus $ii^R(f)$ factors through the $\kappa$-compact object $n_0$, and is therefore $\kappa$-compact by \Cref{ex:factorcompact}. Thus, by \Cref{ex:reflectcpct}, $i^R(n_0)\to i^R(n_1)$ is also $\kappa$-compact, as was claimed.
\end{proof}

\begin{cor}\label{cor:cardboundcompass}
    Let $\M$ be a compactly assembled category. $\M$ is $\omega_1$-compactly generated. 

   If $\M'$ is $\kappa$-compactly assembled and admits colimits, then it is $\max(\kappa,\omega_1)$-compactly generated. 
\end{cor}
\begin{proof}
For the first part of the statement, by compact assembly, we find a fully faithful left adjoint $i:\M\to \N$, where $\N$ is compactly generated and $i^R$ preserves filtered colimits. It follows that $\N$ is $\omega_1$-compactly generated and that $i^R$ preserves $\omega_1$-filtered colimits. 
    By \Cref{prop:cardboundunstable}, $\M$ is $\omega_1$-compactly generated. 

    For the second half of the statement, we again note (by \Cref{obs:kcpctass}) that we can find a fully faithful left adjoint $i:\M'\to \N$ where $\N$ is $\kappa$-compactly generated \emph{and} admits colimits (this is because $\N$ can be chosen to be $\Ind_\kappa((\M')^\lambda)$ for $\lambda$ large enough), so that \Cref{prop:cardboundunstable} applies again, and $\M'$ is $\kappa$-compactly generated. 
\end{proof}
In fact we can completely characterize the $\omega_1$-compacts in $\M$ if $\M$ is compactly assembled and presentable: 
\begin{cor}\label{cor:characcptex}
    Suppose $\M$ is compactly assembled. An object $x\in\M$ is $\omega_1$-compact if and only if it is compactly exhaustible. 
\end{cor}
\begin{proof}
    One direction is \Cref{cor:basicnuccpct}. For the other direction, we note that the proof in \Cref{prop:cardboundunstable} shows that every object of $\M$ can be written as an $\omega_1$-filtered colimit of compactly exhaustible objects. Indeed, we need to show that, in the notation of that proof, $P$ is $\omega_1$-filtered. Since it is cofinal in the poset $Down_\omega(J)$, it suffices to show it for the latter. Now for this, if $F_i, i\in\NN$ is a family of elements of $Down_\omega(J)$, we may choose, for every $i\in\NN$, a cofinal subset of the form $\{x_n^i, n\in\NN\}$. Fix an enumeration $s:\NN^4\to \N$. For every $(n,i), (m,j)\in\NN^4$, find some $\beta_{s(n,i,m,j)}\geq x_n^i, x_m^j, \beta_{s(n,i,m,j)-1}$ (which exists, as $J$ is filtered), and let $F:=\{y\in J\mid \exists k\in\NN, y\leq \beta_k\} $. $F$ is clearly in $Down_\omega(J)$ and each $F_i\subset F$, thus proving the claim. 

    It follows from this that every $\omega_1$-compact object is a retract of a compactly exhaustible object. Now let $x_n, n \in \NN$ be a compact exhaustion of $x$, and let $p:x\to m, i: m\to x$ be a retraction of $x$ onto $m$. 

    Letting $e:= ip$, we have that $m\simeq \colim_\NN(x\xrightarrow{e}x\xrightarrow{e}\dots )$. Furthermore, for every $n$, the composite $x_n\to x\xrightarrow{e}x$ is compact and hence factors through some $x_m$. Composing with $x_m\to x_{m+1}$, we may assume the factorization is itself compact. Up to a cofinal subset, we may assume $m=n+1$, and we let $\tilde e_n: x_n\to x_{n+1}$ be a chosen compact factorization. We then find $m=\colim_\NN x_{2n}$ along $x_{2n}\xrightarrow{i_{2n+1}\circ \tilde e_{2n}}x_{2n+2}$ where $i_m: x_m\to x_{m+1}$ is the map in the original compact exhaustion. 
    \end{proof}
    \begin{rmk}
        A ``moral'' corollary of this corollary is that in the analogy between compactly generated \categories{} and compactly assembled ones, we should not replace compacts by compactly exhaustible objects (since these are exactly the $\omega_1$-compacts), but rather by the compact \emph{exhaustions}. 
    \end{rmk}
\begin{cor}\label{cor:cardboundV}
    Let $\V\in\CAlg(\PrL_\kappa)$ and let $\M$ be a dualizable $\V$-module. In this case, $\M$ is $\max(\kappa,\omega_1)$-compactly generated.
\end{cor}
\begin{proof}
    By \Cref{thm:Luriethmgeneralbase}, there is an internally left adjoint fully faithful embedding $\M\to \N$ for some atomically generated $\V$-module $\N$. By \Cref{cor:atimpliescpct}, $\N$ is $\kappa$-compactly generated and thus by \Cref{prop:cardboundunstable}, $\M$ is $\max(\kappa,\omega_1)$-compactly generated. 
\end{proof}
\begin{cor}\label{cor:boundyhat}
    Let $\M$ be a compactly assembled category. The left adjoint $\hat y:~\M\to~\Ind(\M)$ factors through $\Ind(\M^{\omega_1})$. 

    If $\M$ is a dualizable $\V$-module, where $\V\in\CAlg(\PrL_\kappa)$, then $\hat y$ can be chosen to have values in $\PP_\V(\M^{\max(\kappa,\omega_1)})$. 
\end{cor}
\begin{proof}
In the proof of the existence of $\hat y:\M\to\PP_\V(\M^\lambda)$, all we needed was for $\M$ to be $\lambda$-compactly generated. By \Cref{cor:cardboundV}, $\M$ is $\max(\kappa,\omega_1)$-compactly generated, and thus $\hat y$ exists there. 

By \Cref{rmk:cardindep}, this $\hat y$ is the same as the one for any chosen $\lambda$. 
\end{proof}

By \Cref{cor:yhatnatural}, we find: 
\begin{cor}\label{cor:yhatnatural2}
Let $\V\in\CAlg(\PrL_\kappa)$. 
    The functors $\hat y_\M, \M\in\Dbl{\V}$ promote to a natural transformation $\M\xrightarrow{\hat y}\PP_\V(\M^{\max(\omega_1,\kappa)})$ between functors $\Dbl{\V}\to \Dbl{\V}$. 
\end{cor}
\subsection{Characterization of compact assembly}
The goal of this section is to characterize compact assembly in terms of compact maps, in as similar a way to compact generation. Note that we now focus on compact = $\omega$-compact maps because \Cref{prop:cardboundunstable} implies that for uncountable $\kappa$, $\kappa$-compactly assembled implies $\kappa$-compactly generated. 

\begin{thm}\label{thm:compactassembly}
    Let $\M$ be a category with filtered colimits. $\M$ is compactly assembled if and only if compactly exhaustible objects generate $\M$ under filtered colimits, if and only if filtered colimits in $\M$ are left exact and weakly compactly exhaustible objects generate $\M$ under filtered colimits. 
\end{thm}
This result has a more ``effective'' stable variant, which we will discuss in the next subsection.

\begin{proof}
    First assume that $\M$ is compactly assembled. In this case, we can find a fully faithful embedding $i:\M\to \N$ with a filtered colimit preserving right adjoint $i^R$, where $\N$ is compactly generated. Following the proof of \Cref{prop:cardboundunstable}, for a given $m\in \M$, we write $m\simeq \colim_J n_j$ for some filtered diagram $J\to \N^\omega$. Still following that proof, we see that any $m\in \M$ can be written as $\colim_{F\in P}\colim_{i\in F}i^R(n_i)$, where $P\subset Down_\omega(J)$ was the subposet spanned by the $F\in Down_\omega(J)$ such that $\colim_{i\in F}i^R(n_i)$ is compactly exhaustible. There we concluded that each of these sequential colimits was $\kappa$-compact, but in our situation, we can simply use the definition: each of these colimits is compactly exhaustible, and $P$ is filtered, so we are done. 
    
    We also need to prove that filtered colimits in $\M$ are left exact. For this we note that filtered colimits \emph{and} limits can be computed by applying $i$, taking the correspond co/limit, and then applying $i^R$, so that we are reduced to the compactly generated case. In the compactly generated case, we can test left exactness of filtered colimits by mapping out of all compact objects, as the collection of functors $\Map(c,-), c\in \M^\omega$ is jointly conservative, and they each preserve filtered colimits and finite limits. In other words, we reduce to the case of $\Ss$, where it is true that filtered colimits are left exact \cite[Proposition 5.3.3.3]{HTT}.

    Conversely, assume $\M$ is generated under filtered colimits by (weakly) compactly exhaustible objects, and let $c:\Ind(\M)\to \M$ denote the colimit functor, which we wish to prove admits a left adjoint $\hat y$. 
Our assumption, together with the fact that $\Ind(\M)$ admits filtered colimits, shows that it suffices to prove that $\hat y$ exists on (weakly) compactly exhaustible objects, i.e. that for any (weakly) compactly exhaustible $x$, the functor $y\mapsto \Map_\M(x, c(y))$ is corepresentable in $\Ind(\M)$. Indeed, as $\Ind(\M)$ is cocomplete, the collection of $x$'s such that this functor is corepresentable is closed under colimits. 

This case is now covered by \Cref{lm:basicnuc} - either with no assumption if we are in the compact case, or with the left exactness of filtered colimits if we are in the weakly compact case.  
\end{proof}
We can now characterize functors between compactly assembled categories whose right adjoints preserve filtered colimits in a similar way to how they are characterized between compactly generated categories. 
\begin{cor}\label{cor:strongadjcompactass}
    Let $\M,\N\in\PrL$, and assume $\M$ is compactly assembled, and let $f:~\M\to~\N$ be a left adjoint. In this case, $f^R$ preserves filtered colimits if and only if $f$ preserves compact maps. If filtered colimits in $\N$ are left exact, this is the case if and only if $f$ preserves weakly compact maps. 

More generally, $f^R$ preserves $\kappa$-filtered colimits if and only if $f$ sends compact maps to $\kappa$-compact maps.
\end{cor}

We use the following lemma: 
\begin{lm}\label{lm:colim-detection}
    Let $\M$ be compactly assembled, and let $z_\bullet :I\to \M$ be a filtered diagram with a cocone $z_\bullet \to z$. 

    This cocone is a colimit cocone if and only if for any sequential diagram $x_\bullet:\mathbb N\to\M$ with compact\footnote{By \Cref{ex:embtocheckcpct}, all variants of compact maps agree in $\M$.} transition maps, the canonical map $$\Map_{\Ind(\M)}(\{x_n\}_n,\{z_i\})\to \Map_\M(\colim_\mathbb N x_n, z)$$ is an equivalence.
\end{lm}
\begin{proof}
The canonical map is a composite $$\Map_{\Ind(\M)}(\{x_n\}_n,\{z_i\})\to \Map_\M(\colim_\mathbb N x_n,\colim_I z_i)\to  \Map_\M(\colim_\mathbb N x_n, z)$$ where the first map is an equivalence by \Cref{lm:basicnuc}. 

Thus it is an equivalence if and only if the second one is, i.e. if and only if $\Map_\M(x,\colim_I z_i)\to~\Map_\M(x,z)$ is an equivalence for all compactly exhaustible $x$. Because these generate $\M$ under colimits by \Cref{thm:compactassembly}, this map is an equivalence for all of them if and only if $\colim_Iz_i\to z$ is an equivalence. 
\end{proof}
\begin{proof}[Proof of \Cref{cor:strongadjcompactass}]
Only if is \Cref{ex:preservecpct} and it does not use that $\M$ is compactly assembled. 

    Now assume that $\M$ is compactly assembled, and that $f$ preserves compact maps. 

Consider a filtered diagram $z_\bullet: I\to \N$, and the cocone $f^R(z_\bullet)\to f^R(\colim_I z_i)$, which we want to prove is a colimit-cocone. By \Cref{lm:colim-detection}, it suffices to prove that for any sequential system $x_\bullet$ along compact maps, $\Map_{\Ind(\M)}(\{x_n\}, \{f^R(z_i)\})\to \Map_\M(\colim_\mathbb N x_n, f^R(\colim_I z_i))$ is an equivalence. This map is now equivalent to the map $\Map_{\Ind(\N)}(\{f(x_n)\}, \{z_i\})\to~\Map_\N(\colim_\mathbb N f(x_n), \colim_Iz_i)$.

That this is an equivalence follows from \Cref{lm:basicnuc} and the assumption that $f$ preserves compact maps, so that $f(x_\bullet)$ is still a sequential system along compact maps. 

If filtered colimits in $\N$ are left exact, the same argument goes through, using at the end this assumption to apply \Cref{lm:basicnuc}. 

The proof for the $\kappa$-variant is the same. 
\end{proof}
From the proof, it is clear that we have the following addendum:
\begin{add}\label{add:enoughcpct}
    In the sitation of \Cref{cor:strongadjcompactass}, for $f^R$ to preserve filtered colimits, it suffices that $f$ preserve ``enough'' compact maps, that is, that there exist a set $S$ of compact maps that $f$ sends to compact maps, and such that sequential colimits along maps in $S$ generate $\M$ under colimits. 
\end{add}
Away from presentability, we can still give a characterization of functors that preserve compact morphisms \emph{between} compactly assembled \categories: 
\begin{cor}
    Let $f:\M\to \N$ be a filtered colimit preserving functor between compactly assembled \categories. The functor $f$ preserves compact morphisms if and only if the following diagram is vertically left adjointable:
    \[\begin{tikzcd}
	\M & \N \\
	{\Ind(\M)} & {\Ind(\N)}
	\arrow["f", from=1-1, to=1-2]
	\arrow["{p_\M}", from=2-1, to=1-1]
	\arrow["{\Ind(f)}"', from=2-1, to=2-2]
	\arrow["{p_\N}"', from=2-2, to=1-2]
\end{tikzcd}\]
i.e. if the canonical map $\hat y_\N \circ f\to \Ind(f)\circ \hat y_\M$ is an equivalence. 
\end{cor}
\begin{proof}
    On the one hand, suppose $\hat y_\N \circ f\simeq \Ind(f)\circ \hat y_\M$, and let $\alpha$ be a compact morphism in $\M$. In this case, $\Ind(f)\circ \hat y_\M(\alpha)$ is a compact morphism, and hence so is $\hat y_\N\circ f(\alpha)$, and by \Cref{ex:reflectcpct}, we find that $f(\alpha)$ is compact. 

    Conversely, suppose $f$ preserves compact morphisms. In this case it sends compact exhaustions to compact exhaustions.Using this, one proves (e.g. by comparing $\hat y$ and $y$) that the canonical map $\hat y_\N \circ f\to \Ind(f)\circ \hat y_\M$ is an equivalence on compactly exhaustible objects, and therefore (by filtered colimit preservation and \Cref{thm:compactassembly}) on all objects, as needed. 
\end{proof}
\begin{rmk}
    If $\M,\N$ do not have all colimits or if $f$ does not preserve them, $f$ need not have a right adjoint so we cannot really phrase this in terms of right adjoints. 
\end{rmk}
\begin{defn}
A filtered colimit preserving functor $f:\M\to \N$ between compactly assembled categories is called compactly assembled if it satisfies the equivalent conditions of the previous corollary. 
\end{defn}
\begin{rmk}
    Note that it follows from the corollary that any compactly assembled functor is a retract of a compact preserving, filtered colimit preserving functor between compactly generated \categories. 
\end{rmk}
The following definition is now relatively justified:
\begin{defn}
    Let $\CompAss$ denote the \category{} of compactly assembled \categories{} and compactly assembled functors.  

    Let $\PrL_\ca\subset\CompAss$ denote the non-full subcategory spanned by the presentable \categories{} and colimit-preserving functors\footnote{That are also compactly assembled.}
\end{defn}
\begin{rmk}
    $\CompAss$ is the \category{} of compactly assembled \categories{} mentioned in \Cref{rmk:compassdbl}.
\end{rmk}
We may now prove an analogue of \Cref{prop:colimdbl}:
\begin{prop}\label{prop:colimCompAss}
    The \category{} $\PrL_\ca$ is cocomplete, and furthermore the forgetful functor to $\PrL$ preserves colimits.
\end{prop}
\begin{proof}
    The proof is similar to \Cref{prop:colimdbl}, and we actually leave it as an exercise.
\end{proof}

We now arrive at a non-obvious fact:
\begin{prop}\label{prop:Spcpt}
    The functor $(-)^\omega:\PrL_\ca\to \Cat$ preserves filtered colimits. 
\end{prop}
The proof will use the following, which is also not obvious a priori:
\begin{prop}\label{prop:fffiltered}
    Filtered colimits in $\PrL_\ca$ preserve fully faithful embeddings.
\end{prop}
\begin{proof}
    Suppose $C_i\to D_i$ is a filtered diagram in $\PrL_\ca$ of fully faithful embeddings. 

    We have a commutative square of filtered diagrams: 
    \[\begin{tikzcd}
	{C_i} & {D_i} \\
	{\Ind(C_i^{\omega_1})} & {\Ind(D_i^{\omega_1})}
	\arrow[from=1-1, to=1-2]
	\arrow["{\hat y}"', from=1-1, to=2-1]
	\arrow["{\hat y}", from=1-2, to=2-2]
	\arrow[from=2-1, to=2-2]
\end{tikzcd}\]
where the diagram $C_i\to \Ind(C_i^{\omega_1})$ is a fully faithful left adjoint \emph{in the category of diagrams}, because each of the naturality squares: 
\[\begin{tikzcd}
	{C_i} & {C_j} \\
	{\Ind(C_i^{\omega_1})} & {\Ind(C_j^{\omega_1})}
	\arrow[from=1-1, to=1-2]
	\arrow["{\hat y}"', from=1-1, to=2-1]
	\arrow["{\hat y}", from=1-2, to=2-2]
	\arrow[from=2-1, to=2-2]
\end{tikzcd}\] is vertically right adjointable (and the vertical left adjoints are fully faithful). Thus, when passing to colimits, we have a square of the form: 
\[\begin{tikzcd}
	{\colim_I C_i} & {\colim_I D_i} \\
	{\colim_I\Ind(C_i^{\omega_1})} & {\colim_I\Ind(D_i^{\omega_1})}
	\arrow[from=1-1, to=1-2]
	\arrow["{\hat y}"', from=1-1, to=2-1]
	\arrow["{\hat y}", from=1-2, to=2-2]
	\arrow[from=2-1, to=2-2]
\end{tikzcd}\]
where the vertical functors are fully faithful by adjointability, the bottom horizontal functor is fully faithful by the corresponding statement in $\PrL_\omega$, and therefore the top horizontal functor is fully faithful too.
\end{proof}

\begin{proof}[Proof of \Cref{prop:Spcpt}]
Let $\M_\bullet: I\to \PrL_\ca$ be a filtered system, with colimit $\M_\infty$ and canonical maps $\iota_i: \M_i\to \M_\infty$. 

We first prove that the functor $\colim_I \M_i^\omega\to \M_\infty^\omega$ is fully faithful. This map is the map induced on compacts by $\colim_I \Ind(\M_i^\omega)\to \colim_I \M_i$, which is a filtered colimit of fully faithful embeddings and hence is fully faithful by \Cref{prop:fffiltered}. 

Now we prove that the functor $\colim_I \M_i^\omega\to \M_\infty^\omega$ is essentially surjective. For this, we first set up the diagram of the $\hat y_i:\M_i\to \Ind(\M_i^{\omega_1})$ and take its colimit $\hat j_\infty: \M_\infty\to~\colim_I \Ind(\M_i^{\omega_1})$. Note that this is not $\hat y_{\M_\infty}$ in general. We use $\hat\iota_i:\Ind(\M_i^{\omega_1})\to \colim_I \Ind(\M_j^{\omega_1})$ denote the canonical functor. 

Let $x\in \M_\infty$ be compact. Since $\hat j_\infty$ is compactly assembled, $\hat j_\infty(x)\in \colim_I \Ind(\M_i^{\omega_1})$ is compact. By the analogous claim for $\PrL_{\omega}$, we find that there is an $i$ and an $x_i\in \M_i^{\omega_1}$ such that $\hat \iota_i(y_{\M_i}(x_i))= \hat j_\infty(x)$. Note that $x_i = p_i(j_i(x_i))$, and we have $$x=(\hat j_\infty)^R \hat j_\infty (x) = (\hat j_\infty)^R \hat\iota_i y_i(x_i) = \iota_i p_i y_i(x_i) = \iota_i x_i$$ 

Thus also $\hat j_\infty(x) = \hat j_\infty \iota_i (x_i) = \hat \iota_i \hat y_i(x_i)$: the canonical map $\hat y_i(x_i)\to y_i(x_i)$ becomes an equivalence after applying $\hat \iota_i$. Our goal is to ``witness that equivalence'' at a finite stage $j\geq i$. 

Since $x_i$ is $\omega_1$-compact, by \Cref{cor:characcptex}, it is compactly exhausted, so fix $x_{i,0}\to~\dots\to~x_{i,n}\to~\dots$ a compact exhaustion of $x_i$. We have $\hat y_i(x_i) = \colim_\NN y(x_{i,n})$ and so $\hat \iota_i y_i(x_i) = \hat j_\infty(x) = \colim_\NN \hat\iota_i y_i(x_{i,n})$ - but $\hat j_\infty(x)$ is compact, so we may find a retraction of the form $\hat \iota_i y_i(x_i)\to~\hat \iota_i y_i(x_{i,n})\to~\hat \iota_i y_i(x_i)$. We may furthermore freely insert one more composite in the colimit, so we find that the following composite is equivalent to the identity of $\hat \iota_i y_i(x_i)$:
$$\hat \iota_i y_i(x_i)\to\hat \iota_i y_i(x_{i,n})\to \hat \iota_i y_i(x_{i,n+1})\to \hat\iota_i $$
This is now an equivalence of maps between compacts in $\colim_I \Ind(\M_j^{\omega_1})$ that all come from the $i$th stage, and thus, by the analogous claim for $\PrL_\omega$, it is witnessed at a finite stage: up to picking a larger $j$, we may assume that this equivalence happens in $\M_j^{\omega_1}$. 

But $x_{j,n}\to x_{j,n+1}$ is compact, as it is the image under the compactly assembled functor $\M_i\to \M_j$ of a compact map. By \Cref{ex:factorcompact}, it follows that the identity of $x_j$ is compact, i.e. that $x_j$ is compact. In total, this proves essential surjectivity since $\iota_j(x_j) =~\iota_j \iota_i^j(x_i) =~\iota_i (x_i) = x$. 
\end{proof}
\begin{cor}
    Let $\V\in\CAlg(\PrL_\omega)$. The functor $(-)^\omega: \Dbl{\V}\to \Cat$ preserves filtered colimits. 
\end{cor}
\begin{proof}
    The functors $\Dbl{\V}\to \Mod_\V(\PrL)$ and $\Mod_\V(\PrL)\to \PrL$ preserve all colimits, the second for general reasons and the first by \Cref{prop:colimdbl}. 

    Since $\V$ is in $\CAlg(\PrL_\omega)$, combining \Cref{thm:Luriethmgeneralbase} and \Cref{cor:atimpliescpct}, the composite $\Dbl{\V}\to \PrL$ factors through $\PrL_\ca$. It follows from \Cref{prop:colimCompAss} (and conservativity of $\PrL_\ca\to \PrL$) that $\Dbl{\V}\to~\PrL_\ca$ preserves colimits. 

    Now we can conclude with \Cref{prop:Spcpt} that the composite $\Dbl{\V}\to~\PrL_\ca\xrightarrow{(-)^\omega}~\Cat$ preserves filtered colimits. 
\end{proof}
\begin{rmk}
    This result can be almost equivalently stated as ``If $\V\in \CAlg(\PrL_\omega)$, then $\V$ is compact in $\Dbl{\V}$'' (this statement does follow from the result). 
\end{rmk}
 
We conclude this section with a proof of \Cref{ex:Shisdbl}: 
\begin{prop}\label{pfofShdbl}
    Let $X$ be a locally compact Hausdorff space. The \category{} $\Sh(X)$ of sheaves on $X$ is compactly assembled. In particular, $\Sh(X;\Sp)$ is dualizable in $\PrL_{\st}$.
\end{prop}
\begin{proof}

    It follows from \Cref{ex:cpctinSh} that for all $U\Subset V$, the induced map of sheaves $h_U\to h_V$ is compact in $\Sh(X)$. 

In particular, whenever $U$ is an open in $X$ such that $U=\bigcup_n U_n$ for some sequence of opens with $U_n\Subset U_{n+1}$, $h_U$ is compactly exhaustible. So it suffices to prove that these $h_U$'s generate $\Sh(X)$ under colimits. It therefore suffices to prove that these opens form a basis of $X$. 

It suffices to prove that for any $x\in X$, and and any $V\ni x$ open containing $x$, there exists such a $U$ with $x\in U\subset V$. But we have the following two facts:
\begin{enumerate}
    \item For any $x\in V$, one can find a relatively compact open $U$ with $x\in U\Subset V$;
    \item For any good inclusion $W_0\Subset W_1$, there exists a refinement to $W_0\Subset W_2 \Subset W_1$.
\end{enumerate}
Together they lead to the following construction: start with $x\in U_0\Subset V$ with $U_0$ as in the first point; then construct $U_1$ so that $U_0\Subset U_1\Subset V$, $U_2$ so that $U_1\Subset U_2\Subset V$, etc. so that ultimately, $U = \bigcup_n U_n\subset V$ is one of these opens. 
\end{proof}
\subsection{Stable compact assembly and continuous $K$-theory}
In this section, we specialize to the case of stable \categories. In this case, we will be able to rephrase the criterion from \Cref{thm:compactassembly} in terms of ``zero-ness'', similarly to the case of compact generation. We will also survey the theory of compactly assembled stable \categories, as well as the basics of Efimov's continuous $K$-theory. Note that by combining \Cref{defn:compactass} and \Cref{thm:lurie}, we find that compactly assembled stable \categories{} are exactly dualizable presentable stable \categories. 

 We learned the following statement from Dustin Clausen:
 \begin{thm}\label{thm:dblviacpctmaps}
    Let $\M$ be a stable presentable \category. $\M$ is compactly assembled if and only if there exists a small set $S$ of compact maps in $\M$ such that for any $x\in \M$, if $x$ is nonzero, then there exists some nonzero map to $x$ that factors through $S$, and such that either of the following two conditions hold:
    \begin{enumerate}
        \item Any map in $S$ factors as a composite of two maps in $S$; 
        \item All the objects appearing as sources of maps in $S$ are in the colimit-completion of $S$-telescopes in $\M$. Here, an $S$-telescope is a sequential colimit along maps in $S$. 
    \end{enumerate}
     \end{thm}
These conditions help to guarantee that the compactly exhaustible objects generate $\M$ under colimits, leading us to \Cref{thm:compactassembly}. 
\begin{proof}
  (Only if) Suppose that $\M$ is compactly assembled, and let $\hat y: \M\to\Ind(\M^{\omega_1})$ denote the left adjoint to the canonical functor $c: \Ind(\M)\to \M$ (cf. \Cref{cor:boundyhat}), and $y$ the right adjoint to $c$, i.e. the Yoneda embedding. 

We let $S$ denote the set of all compact maps between $\omega_1$-compacts in $\M$. Let $x$ be a nonzero object of $\M$, and write $\hat y(x)\simeq\colim_I y(x_i)$ for some filtered diagram $I\to \M^{\omega_1}$. As $\Ind(\M^{\omega_1})$ is compactly generated, one of the maps $y(x_i)\to \hat y(x)$ must be nonzero by \Cref{prop:nonzerocolim}. 

Furthermore, as in the proof of \Cref{prop:cardboundunstable}, we also have that the natural map $\hat y(x_j)\to y(x_j)$ induces an equivalence $\colim_I \hat y(x_j)\to \hat y(x)$: it is a map between objects in the image of $\hat y$, and applying $c$ clearly yields an equivalence. Compactness of $y(x_i)$ guarantees that the map $y(x_i)\to \hat y(x)$ factors through some $\hat y(x_j)$, and the map $\hat y(x_j)\to \hat y(x)$ factors, by construction, through $y(x_j)\to \hat y(x)$. 

It follows that $x_i\to x_j$ is a compact map between $\omega_1$-compacts, and that $x_i\to x_j\to x$ is nonzero.

The same kind of argument immediately shows that any compact map between $\omega_1$-compacts factors as a composite of two such compact maps. 

Finally, the proof of \Cref{thm:compactassembly} shows that all objects of $\M$ are filtered colimits of compactly exhaustible objects, and so this proves that both 1. and 2. are satisfied. 
\newline 

(If) By \Cref{thm:compactassembly}, it suffices to prove that $\M$ is generated under colimits by compactly exhaustible objects. 

We need to prove that for any nonzero $x\in \M$, there is a nonzero map from a compactly exhaustible object $y$. For this, use our first assumption, namely that there is a nonzero composite $y_0\to y_\infty\to x$, with $y_0\to y_\infty\in S$. 

If we first assume 1., one can factor this as $y_0\to y_1\to  y_\infty$, with both maps in $S$, and then we can similarly factor $y_1\to y_\infty$ as $y_1\to y_2\to y_\infty$ and so on. Ultimately, we find a sequential diagrams with transition maps in $S$ whose colimit $y_\omega$ fits as $y_0\to y_\omega\to y_\infty$, so that the map $y_\omega\to y_\infty\to x$ is nonzero, as $y_0\to x$ factors through it. $y_\omega$ is clearly compactly exhaustible, so we are done by \Cref{lm:kappacpctgen}. 

We now assume 2. in place of 1.. In this case, $y_0$ is by assumption in the colimit-completion of $S$-telescopes, and so we are simply done by \Cref{add:kappacpctgen}.   
\end{proof}
\begin{rmk}\label{rmk:compactmapsplitunstable}
    The proof of (Only if) works without stability, in particular, in any compactly assembled category, any compact map factors as a composite of two compact maps. 
\end{rmk}
A corollary of this remark is the following: 
\begin{cor}\label{cor:cpctexhaustQ}
    Let $\M$ be a compactly assembled category, and $x$ a compactly exhaustible object therein. There exists a $\mathbb Q_{\geq 0}$-indexed diagram with colimit $x$ such that each of the transition maps is compact. 
\end{cor}
This will be convenient later on to construct new compactly assembled categories in \Cref{section:limits}, as well as to construct rigidifications in \cite{companion}. To prove this, we first explain how to construct maps out of $\mathbb Q_{\geq 0}$: 
\begin{lm}\label{lm:rationals}
    Let $C$ be a category, and $S$ a collection of maps in $C$ such that for any $f:x\to y$ in $S$, there exist $x\to z, z\to y$, both in $S$, whose composite is $f$. 

    In this case, for any $f:x\to y$ in $S$, there exists a diagram $x_\bullet: \mathbb Q_{\geq 0}^\triangleright\to C$ such that $x_0= x, x_\infty = y$, and each  nontrivial map $r<s$ in $\mathbb Q_{\geq 0}^\triangleright$ is sent to a map in $S$ in $C$.
\end{lm}
\begin{proof}
For each $k$, let $G_k$ denote the full subposet of the dyadics spanned by the dyadics of depth $\leq k$, i.e. of the form $\frac{m}{2^k}$ for some $m\in [0,2^k]$. This is a finite totally ordered set, i.e. a simplex, and so it is a free on its spine. 

Thus, to give a dotted lift as in \[\begin{tikzcd}
	{G_k} & C \\
	{G_{k+1}}
	\arrow[from=1-1, to=2-1]
	\arrow["f", from=1-1, to=1-2]
	\arrow[dashed, from=2-1, to=1-2]
\end{tikzcd}\] it suffices to fill it on the spine, i.e. to provide a factorization of each $f(\frac{m}{2^k})\to f(\frac{m+1}{2^k})$ through some object (to be the image under the dotted map of $\frac{2m+1}{2^{k+1}}$). 

Finally, by Cantor's theorem, $\colim_k G_k\simeq \mathbb Q_{\geq 0}^\triangleright$. 
\end{proof}
\begin{rmk}
    The inclusion $\mathbb N\to \mathbb Q_{\geq 0}$ is cofinal, thus, given a $\mathbb Q_{\geq 0}$-indexed diagram $x_\bullet$, we have $\colim_\mathbb N x_n \simeq \colim_{\mathbb Q_{\geq 0}}x_r$. The relevance of $\mathbb Q_{\geq 0}$ in this story is thus completely captured by the above lemma: given an $x$ which can be written as a sequential colimit of maps in a collection $S$, the fact that it can be written as a $\mathbb Q_{\geq 0}$-indexed colimit of such things is supposed to mean that the maps in the sequential diagram whose colimit is $x$ can themselves be refined to composites of maps in $S$, who themselves can be refined as composites of maps in $S$ etc. 
\end{rmk}

\begin{proof}[Proof of \Cref{cor:cpctexhaustQ}]
    Write $x=\colim_\mathbb N x_n$ along compact maps. Now for each $n$, the map $x_n\to x_{n+1}$ is compact and so, by \Cref{rmk:compactmapsplitunstable}, we can fit a $\mathbb Q^{\triangleright}_{\geq 0}$-diagram with compact transition maps inbetween $x_n$ and $x_{n+1}$. Putting all these diagrams together gives a $\mathbb Q_{\geq 0}$-indexed diagrams in which $\mathbb N$ is cofinal, and hence it has colimit $x$, as desired.   
\end{proof}

We now explore Efimov's continuous $K$-theory, and the structure of $\Prdbl$, the category of dualizable presentable stable categories. We first start by noting that the inclusion $\Prdbl\to \PrL_{\st}$ preserves \emph{some} pullbacks. As a warm-up, we note:
\begin{lm}
    Let $f:\M\to\N$ be an internal left adjoint between dualizable stable categories. The full subcategory $\N_0$ of $\N$ generated under colimits by the image of $\M$ is dualizable, and the inclusion $\N_0\to \N$ is an internal left adjoint. 
\end{lm}
\begin{proof}
    By \Cref{prop:colimdbl}, the pushout $\N\coprod_\M 0$ in $\PrL_{\st}$ lives in $\Prdbl$. But this pushout in $\PrL_{\st}$ is exactly the quotient $\N/\langle \mathrm{im}(\M)\rangle$, where $\langle -\rangle$ means the localizing subcategory generated by a given subcategory. Thus the localization $\N\to \N/\langle \mathrm{im}(\M)\rangle $ is an internal left adjoint, and it follows formally that the inclusion $\langle \mathrm{im}(\M)\rangle \to \N$ is also an internal left adjoint, see, e.g. \Cref{cor:incstrongiffprojstrong}. 
    
    It follows that $\langle \mathrm{im}(\M)\rangle$ is dualizable, as a retract of $\N$. 
\end{proof}
We do not need this right now, but we record the general version, for which the proof needs to be slightly different:
\begin{lm}\label{lm:dblim}
    Let $f:\M\to \N$ be a map in $\Dbl{\V}$. The full subcategory of $\N$ generated under colimits and $\V$-tensors by the image of $\M$ is dualizable, and both its inclusion into $\N$ as well as the corestriction of $f$ are internal left adjoints. 
\end{lm}
\begin{proof}
    Let $\mathcal I$ denote said full subcategory. We have a factorization of $f$ as $\M\xrightarrow{\tilde f} \mathcal I\xrightarrow{i} \N$. As $i$ is fully faithful and $f$ lands in its image, we find that the canonical map $\tilde f^R \to f^R \circ i$ is an equivalence, from which it follows that $\tilde f$ is an internal left adjoint. Furthermore, $\tilde f^R$ is conservative by construction, so that \Cref{lm:colimdetection} implies that $i$ is an internal left adjoint, as $i\circ \tilde f$ is. 

    As $i$ is fully faithful and an internal left adjoint, we have that $\mathcal I$ is a retract of $\N$ and is thus dualizable. 
\end{proof}
\begin{lm}\label{lm:kernel}
    Let $f:\M\to \N$ be map in $\Prdbl$ whose underlying map in $\PrL_{\st}$ is a localization. In this case the forgetful functor $\Prdbl\to \PrL_{\st}$ preserves $\fib(f)$\footnote{In particular this fiber exists - this is clear from the presentability result we will prove, but we do not need this here.}.

    More generally, for any $f:\M\to\N$ in $\Prdbl$, $\fib(f)$ is the largest dualizable subcategory of $\M$ contained in $\ker(f)$ and for which the inclusion functor into $\M$ is an internal left adjoint\footnote{In particular this largest dualizable subcategory exists.}.
\end{lm}
\begin{proof}
For $f:\M\to \N$ as in the first half of the lemma, $\ker(f)$ is dualizable by \Cref{thm:compactassembly} and the inclusion $\ker(f)\to \M$ is an internal left adjoint (again by \Cref{cor:incstrongiffprojstrong}) so it clearly suffices to prove the second half of the lemma. 

Now consider the category (which is a poset) of full subcategories $\M_0\subset \M$ that are dualizable and for which the inclusion $\M_0\to\M$ is an internal left adjoint. We claim that this is a small category. Indeed, they are all $\omega_1$-compactly generated by \Cref{cor:cardboundcompass}, and the inclusion into $\M$ preserves $\omega_1$-compacts, so that this category is actually a full subposet of the poset of full subcategories of $\M^{\omega_1}$, which is clearly small. 

In particular it has a colimit in which, by \Cref{cor:colimintleft} is computed in $\PrL_{\st}$ and is simply the largest element in this poset. We call it $\M_f$. 

Now let $g:\mathcal E\to \M$ be an internal left adjoint with $\mathcal E$ dualizable, such that $f\circ g = 0$. In this case, $f$ restricts to $0$ also on $\langle \mathrm{im}(g)\rangle$ which, by the previous lemma is therefore included in $\M_f$. The result follows formally from this. 
\end{proof}
Recall again from \Cref{prop:colimdbl} that $\Prdbl\to\PrL_{\st}$ preserves all small colimits. Together with the above lemma, this shows that fiber-cofiber sequences in $\Prdbl$ are the same as fiber-cofiber sequences in $\PrL_{\st}$ whose terms and functors lie in $\Prdbl$. In $\PrL_{\st}$, fiber-cofiber sequences are exactly localization sequences by \Cref{prop:loc=fibcofib}.
\begin{defn}\label{defn:locinv}
    Let $E:\Prdbl\to\E$ be a pointed functor to an additive category $\E$. We say $E$ is a localizing invariant if it sends fiber-cofiber sequences in $\Prdbl$ to fiber sequences in $\E$. 

    We let $\Loc(\E)$ denote the \category{} of $\E$-valued localizing invariants, and let $\Loc(\Cat^\perf,\E)$ be defined similarly with $\Cat^\perf$ in place of $\Prdbl$, where $\Cat^\perf$ is the category of idempotent-complete small stable categories.
\end{defn}
The key construction to prove Efimov's \Cref{thm:Efimov} is the Calkin construction:
\begin{defn}
    Let $\M$ be a dualizable stable \category, and $\hat y :\M\to \Ind(\M^\kappa)$ the left adjoint to $c: \Ind(\M^\kappa)\to \M$, for $\kappa\geq \omega_1$. We define $\Calk_\kappa(\M) := \Ind(\M^\kappa)/\hat y(\M)$. With no subscript, we let $\Calk(\M):= \Calk_{\omega_1}(\M)$. 

    Finally, when $\M=\Ind(\M_0)$ is compactly generated, we sometimes abuse notation and write $\Calk_\kappa(\M_0)$ in place of $\Calk_\kappa(\M)$. 
\end{defn}
\begin{obs}\label{obs:natseq}
    For $\kappa\geq\omega_1$, by \Cref{cor:yhatnatural}, the map $\hat y: \M\to\Ind(\M^\kappa)$ is natural in $\M$, and therefore so is the localization sequence $\M\to \Ind(\M^\kappa)\to \Calk_\kappa(\M)$.
\end{obs}
\begin{ex}
    Let $\M=\LMod_R$ for some ring spectrum $R$, and let $M,N\in~\LMod_R^\kappa$. In this case, $\map_{\Calk_\kappa(\LMod_R)}(y(M),y(N))$ can be computed in a familiar way. Indeed, by design it is $$\map_{\Ind(\LMod_R^\kappa)}(y(M),y(N)/\hat y(N)) \simeq\map_{\Ind(\LMod_R^\kappa)}(y(M),y(N))/\map_{\Ind(\LMod_R^\kappa)}(y(M),\hat y(N)) $$ $$\simeq \map(M,N)/\map(y(M),\hat y(N))$$ 

    Now we need to compute $\hat y(N)$. For this, we note that $\hat y$ preserves colimits, and agrees with $y$ on compact objects (by \Cref{lm:basicnuc} : indeed, they are compactly exhaustible with identity transition maps!). Thus the functor $\map(y(M),\hat y(-))$ is the colimit-preserving extension of the functor $\map_R(M,-)$ on compacts, i.e. $\map_R(M,R)\otimes_R - =: M^\vee\otimes_R -$. In total, we find: 
    $$\map_{\Calk_\kappa(\LMod_R)}(y(M),y(N))\simeq \map_R(M,N)/M^\vee\otimes_R N$$ so something one can think of as ``maps modulo finite rank maps''.
\end{ex}
This construction (or a variant thereof) was used originally by Bass to construct negative algebraic $K$-theory (although not in this language), as it provides canonical deloopings for localizing invariants (see also \cite{tabuadasusp}), because of the following lemmas:
\begin{lm}
    Let $E:\Prdbl\to\E$ be a localizing invariant to an additive category $\E$.  For any $\M\in\Prdbl$ and any uncountable $\kappa$, we have $E(\Ind(\M^\kappa)) = 0$. 
\end{lm}
\begin{proof}
    This follows by an Eilenberg swindle, see \Cref{prop:Eilenberg}. Indeed, $\M^\kappa$ admits countable coproducts.
\end{proof}
\begin{cor}\label{cor:deloop}
Let $E:\Prdbl\to\E$ be a localizing invariant and $\kappa \geq \omega_1$. There is a natural equivalence $\Omega E(\Calk_\kappa(\M))\simeq E(\M)$. 
\end{cor}
\begin{proof}
    By \Cref{obs:natseq}, there is a natural fiber-cofiber sequence $\M\to \Ind(\M^\kappa)\to \Calk_\kappa(\M)$, which induces upon applying $E$ a natural fiber sequence $E(\M)\to 0\to E(\Calk_\kappa(\M))$ by the previous lemma, i.e. it induces a natural equivalence as claimed. 
\end{proof}
\begin{lm}
    Let $\mathcal K\xrightarrow{i} \M\xrightarrow{p} \N$ be a fiber-cofiber sequence in $\Prdbl$. For any $\kappa\geq \omega_1$, $\Calk_\kappa(\mathcal K)\to \Calk_\kappa(\M)\to \Calk_\kappa(\N)$ is a fiber-cofiber sequence. 
\end{lm}
\begin{proof}
We first note that $\mathcal K^\kappa\to \M^\kappa\to\N^\kappa$ is a localization sequence of small stable \categories{} by \Cref{prop:kappaloc}.

It follows that $\Ind(\mathcal K^\kappa)\to \Ind(\M^\kappa)\to \Ind(\N^\kappa)$ is a localization sequence. Since the cofiber functor commutes with cofiber sequences, it follows that $\Calk_\kappa(\M)/\Calk_\kappa(\mathcal K)\simeq\Calk_\kappa(\N)$, and since the square: 
\[\begin{tikzcd}
	\mathcal K & \M \\
	{\Ind(\mathcal K^\kappa)} & {\Ind(\M^\kappa)}
	\arrow[from=1-1, to=2-1]
	\arrow[from=2-1, to=2-2]
	\arrow[from=1-2, to=2-2]
	\arrow[from=1-1, to=1-2]
\end{tikzcd}\]
is vertically right adjointable, it follows from \Cref{prop:adjVerdier} that so too is the square: 
\[\begin{tikzcd}
	{\Ind(\mathcal K^\kappa)} & {\Ind(\M^\kappa)} \\
	{\Calk_\kappa(\mathcal K)} & {\Calk_\kappa(\M)}
	\arrow[from=1-1, to=2-1]
	\arrow[from=2-1, to=2-2]
	\arrow[from=1-1, to=1-2]
	\arrow[from=1-2, to=2-2]
\end{tikzcd}\]\footnote{Note that this square is most often not horizontally right adjointable - it is the case if and only if the right adjoint $i^R:\M\to \mathcal K$ is itself an internal left adjoint. }
From this, it follows that $\Calk_\kappa(\mathcal K)\to\Calk_\kappa(\N)$ is fully faithful, and so we are done. 
\end{proof}
A consequence of this is the stability of the \category{} of localizing invariants:
\begin{cor}
    Let $\E$ be an additive category with finite limits. The forgetful functor $\Loc(\Sp(\E))\to \Loc(\E)$ is an equivalence. 
\end{cor}
\begin{proof}
    It is clear that $\Loc(\Sp(\E))\simeq \Sp(\Loc(\E))$, compatibly with the forgetful functor to $\Loc(\E)$. In particular, it suffices to prove that $\Loc(\E)$ is stable. 

    Because it is pointed and has finite limits, by \cite[Proposition 1.4.2.11.(3)]{HA}, it suffices to prove that $\Omega$ induces an equivalence on it - but $\Omega$ is given by postcomposition with $\Omega$ and it commutes with precomposition by $\Calk_\kappa$ (which preserves $\Loc(\E)$ by the previous lemma), and they are inverses to one another by \Cref{cor:deloop}. 
\end{proof}
The key lemma is now:
\begin{lm}\label{lm:calkcpct}
    Let $\M \in\Prdbl$. For any $\kappa$, $\Calk_\kappa(\M)$ is compactly generated. 
\end{lm}
\begin{proof}
The inclusion $\hat y : \M\to \Ind(\M^\kappa)$ is an internal left adjoint, therefore so is the quotient functor $\Ind(\M^\kappa)\to\Calk_\kappa(\M)$. Now an internal left adjoint localization of a compactly generated \category{} is itself compactly generated, as the images of the compact objects are compact and generate.
\end{proof}
We can now state and prove Efimov's theorem:
\begin{thm}\label{thm:Efimov}
    Let $\E$ be an additive category with finite limits. The forgetful functor $\Loc(\E)\to \Loc(\Cat^\perf,\E)$ is an equivalence, with inverse given by $E\mapsto (\M\mapsto \Omega E(\Calk(\M)^{\omega})$. 
\end{thm}
\begin{proof}
Everything we did so far proves that this functor is inverse to the restriction. 
\end{proof}
\begin{nota}
    Given a localizing invariant $E:\Cat^\perf\to \E$, we let $E_{cont}$ (``continuous $E$'') denote its unique extension to a localizing invariant on $\Prdbl$. 
\end{nota}
\begin{rmk}
    There is no variant of this theorem with additive invariants in the sense of \cite{BGT}. Indeed, ind-completions of functors with right adjoints correspond to functors whose right adjoint is itself an internal left adjoint, and for a dualizable $\M$, if there is such an inclusion in a compactly generated category $\N$, then $\M$ is also compactly generated.

    In particular, $\M\mapsto K(\M^\omega)$ would be an additive invariant in any reasonable sense, and it agrees with $K_{cont}$ on compactly generated categories but not on arbitrary dualizable categories (combine \Cref{ex:KofSh} with \cite{Oscarcompact}). 
\end{rmk}
\begin{ex}
    $\THH$ is naturally defined on $\Prdbl$ as a trace, cf. \cite{HSS}; and this functor is localizing by \cite[Theorem 3.4]{HSS}. It follows that it is $\THH_{cont}$.
\end{ex}
We also record the following variant from \cite{Hoyois}. 
\begin{cor}
Let $\E$ be a cocomplete stable \category, and let $\mathcal K$ be a collection of small \categories{} and let $\Loc_\mathcal K(\E)$ denote the full subcategory of $\Loc(\E)$ spanned by localizing invariants preserving $\mathcal K$-shaped colimits. The forgetful functor $\Loc_\mathcal K(\E)\to \Loc_\mathcal K(\Cat^\perf,\E)$ is an equivalence.
\end{cor}
\begin{rmk}
    We could not make sense of the proof sketch in \cite{Hoyois}, cf. \Cref{warn:colimInd}
\end{rmk}
\begin{proof}
It suffices to prove that for any small \category{} $K$, $E$ preserves $K$-shaped colimits if and only if $E_{cont}$ does. Since $\PrL_{\st,\omega}\to \Prdbl$ preserves arbitrary colimits, one direction is clear: if $E_{cont}$ preserves them, then so does $E$. 

Conversely, suppose $E$ preserves $K$-shaped colimits and let $\M_\bullet : K\to\Prdbl$ be an arbitrary diagram. Consider the natural localization sequence $\M_k\to \Ind(\M_k^\kappa)\to \Calk_\kappa(\M_k)$ and consider its colimit: it gives a localization sequence $\colim_K \M_k\to \colim_K\Ind(\M_k^\kappa)\to \colim_K\Calk_\kappa(\M_k)$. 

We now claim that $E_{cont}$ vanishes on the middle category, from which it follows that $E_{cont}(\colim_K\M_k)\simeq \E_{cont}(\colim_K\Calk_\kappa(\M_k))$. Now each $\Calk_\kappa(\M_k)$ is compactly generated, therefore so is the colimit, and therefore $$E_{cont}(\colim_K\Calk_\kappa(\M_k))\simeq E((\colim_K\Calk_\kappa(\M_k))^\omega) \simeq E(\colim_K\Calk_\kappa(\M_k)^\omega)\simeq \colim_KE(\Calk_\kappa(\M_k)^\omega)$$ 

From this it follows that $E_{cont}(\colim_K \M_k)\simeq \colim_K E_{cont}(\M_k)$. 

    To prove that $E_{cont}$ vanishes on $\colim_K\Ind(\M_k^\kappa)$ we note that the latter is compactly generated with compacts $\colim_K \M_k^\kappa$ and hence $E_{cont}$ evaluated on it is $E(\colim_K\M_k^\kappa)\simeq~\colim_K E(\M_k^\kappa)=~0$. 
\end{proof}
\begin{warn}\label{warn:colimInd}
    It is \emph{not} the case in general that $\colim_K\Ind(\M_k^\kappa)\simeq \Ind((\colim_K\M_k)^\kappa)$. 
\end{warn}
\begin{ex}\label{ex:KofSh}
    Algebraic $K$-theory commutes with filtered colimits, it follows that $K_{cont}$ also does. Combining this with formal properties of the assignment $X\mapsto \Sh(X;\M)$ allows for Efimov's calculation $K_{cont}(\Sh(X;\M))\simeq \Gamma(X,K_{cont}(\M))$ for any dualizable $\M$ and compact Hausdorff topological space $X$, cf. \cite{Hoyois,youtubeDustin1}\footnote{Hoyois proves this when $X$ is hypercomplete, e.g. if it is a manifold, but Clausen's proof, as well as Efimov's original proof work in this generality. In fact, these proofs are enough to identify the noncommutative motive of $\Sh(X;\M)$ as opposed to only its $K$-theory.}. 
\end{ex}
\begin{rmk}
    Continuous $K$-theory (and more generally, continuous $E$ for any localizing invariant $E$) can be \emph{constructed} in terms of the Calkin construction, but in concrete cases, it is often helpful to compute it using different ``resolutions'', i.e. given a dualizable $\M$, finding an appropriate compactly-generated $\N$ and an internal left adjoint embedding $\M\to\N$, which need not be $\hat y: \M\to \Ind(\M^{\omega_1})$. 
\end{rmk}
\subsection{An example: smooth categories}\label{section:smooth}
In this section, we exploit a structural property of compact maps to record an analogue of the fact that modules over smooth algebras that are underlying compact, are compact as modules.

This structural property will be helpful later on when we describe a construction of dualizable stable categories. 

\begin{prop}\label{prop:cpctcofib}
    Consider a diagram of vertical cofiber sequences of the form: 
    \[\begin{tikzcd}
	{x_0} & {y_0} & {z_0} \\
	{x_1} & {y_1} & {z_1} \\
	{x_2} & {y_2} & {z_2}
	\arrow["{f_0}", color={rgb,255:red,214;green,92;blue,92}, from=1-1, to=1-2]
	\arrow["{g_0}", color={rgb,255:red,214;green,92;blue,92}, from=1-2, to=1-3]
	\arrow["{f_1}", color={rgb,255:red,214;green,92;blue,92}, from=2-1, to=2-2]
	\arrow["{g_1}", color={rgb,255:red,214;green,92;blue,92}, from=2-2, to=2-3]
	\arrow[from=2-1, to=3-1]
	\arrow[from=1-1, to=2-1]
	\arrow[from=1-2, to=2-2]
	\arrow[from=1-3, to=2-3]
	\arrow[from=2-2, to=3-2]
	\arrow[from=2-3, to=3-3]
	\arrow["{f_2}"', from=3-1, to=3-2]
	\arrow["{g_2}"', from=3-2, to=3-3]
\end{tikzcd}\]
where we have suppressed the nullhomotopies for notational convenience, but the horizontal maps are maps of cofiber sequences; and where $f_0,f_1,g_0,g_1$ are weakly compact maps. 

In that case, $g_2\circ f_2$ is weakly compact. 
\end{prop}
\begin{warn}\label{warn:cpctcofib}
    It is not generally true that $f_2$ (resp. $g_2$) is itself a compact map. To give an example, let $f:x\to y$ be an arbitrary map. It can be written as the cofiber of a canonical map 
    \[\begin{tikzcd}
	{\Omega x} & 0 \\
	0 & y
	\arrow[from=1-1, to=2-1]
	\arrow[from=2-1, to=2-2]
	\arrow[from=1-2, to=2-2]
	\arrow[from=1-1, to=1-2]
\end{tikzcd}\]
where the homotopy making the square commute records $f$. As $0$ is a compact object, both maps $\Omega x\to 0, 0\to y$ are compact.

Note that this example more generally shows that any map is the cofiber of compact maps, so that \Cref{prop:cpctcofib} should not be summarized to quickly as ``A composite of two cofibers of compact maps is compact''. The compatibility of the presentation as such cofibers is essential. 
\end{warn}
\begin{proof}
    Let $z_2\to \colim_I m_i$ be a map to a filtered colimit. The restriction to $y_1$ factors through a finite stage $m_{i_0}$ as $y_1\to y_2$ is compact. 

We further note that the map $y_0\to y_1\to m_{i_0}\to \colim_I m_i$ is $0$, as it factors through $z_0\to z_1\to z_2$. It follows that we get a map $y_0\to \fib(m_{i_0}\to\colim_{I_{i_0/}}m_i)$, whose restriction to $x_0$ factors through a finite stage $\fib(m_{i_0}\to m_{i_1})$, as $x_0\to y_0$ is compact. 

Thus the map $x_1\to y_1\to m_{i_1}$ has a canonical nullhomotopy when restricted to $x_0$, and thus induces a map $x_2\to m_{i_1}$. Furthermore, the composite $x_2\to m_{i_1}\to \colim_I m_i$ corresponds to the nullhomotopy $x_0\to \fib(m_{i_0}\to m_{i_1})\to\fib(m_{i_0}\to \colim_I m_i)$ which in turn comes from the map $x_0\to y_0$ and the nullhomotopy induced by $y_0\to z_0\to z_1\to~\colim_I m_i$. 

Tracing through, we find that the composite $x_2\to \colim_I m_i$ is indeed restricted from $z_2$, because the corresponding nullhomotopy is. 
\end{proof}
We also make the following trivial observation:
\begin{lm}
The collection of (weakly, strongly) compact maps is closed under retracts in $C^{\Delta^1}$. 
\end{lm}
\begin{proof}
    If $f$ is a retract of $g$, then $f$ factors through $g$. Hence this follows from \Cref{ex:factorcompact}.
\end{proof}
We can now explain the desired example:
\begin{thm}\label{thm:smooth}
Let $\M$ be a smooth category, that is, an object of $\Prdbl$ such that the coevaluation map $\Sp\to \M\otimes\M^\vee$ is itself an internal left adjoint; or equivalently such that the identity functor in $\Fun^L(\M,\M)$ is compact. 

We also assume that $\M$ is compactly generated. In this case, there is a compact object $x\in\M$, and a nonnegative integer $n$ such that for any composable sequence of maps $f_1,...,f_n$ in $\M$, if each $\map(x,f_i)$ is compact in $\Sp$, then $f_1\circ ...\circ f_n$ is compact in $\M$. 

More generally, let $\N$ be a dualizable category. For any composable sequence of maps $f_1,...,f_n$ in $\M\otimes\N$, if the functor $\map(x,-)\otimes \id_\N:\M\otimes\N\to \N$ has compact values on the $f_i$'s, then $f_1\circ ... \circ f_n$ is compact in $\M\otimes \N$. 

If $\M=\Mod_A$ for some smooth algebra $A$, $x$ can be chosen to be $A$ itself, so that $U=\map(A,-)$ is the forgetful functor $\M\to\Sp$, and in the relative case, $U\otimes \N$ is identified with the forgetful functor $\Mod_A(\N)\to \N$. 
\end{thm}
\begin{proof}
Consider the identity of $\M$ as an object of $\M\otimes\M^\vee$. As such, it is a colimit of pure tensors, and so, as it is compact, it is a retract of a finite colimit of pure tensors, that is, objects of the form $y_i\boxtimes f_i, y_i\in\M, f_i:\M\to \Sp$. As $\M$ is compactly generated, the $y_i$'s can be chosen to be compact, and the $f_i$'s too, i.e. they can be chosen to be of the form $\map(x_i,-), x_i\in\M^\omega$. 

Furthermore, if $\M=\Mod_A$, $\Fun^L(\Mod_A,\Mod_A) \simeq \Mod_{A\otimes A\op}$ and so up to considering a possibly iterated finite colimit, we may choose $x_i= A, y_i = A$, as claimed. 

The object $x$ we choose in general is $\bigoplus_i x_i$ - it is a finite direct sum, hence still compact. 

Now consider the collection of functors $F:\M\otimes \N\to \M\otimes \N$ for which there exists a natural number $n$ so that the desired statement holds, namely any composable sequence of length $n$ of maps $f_i$ for which $(\map(x,-)\otimes\id_\N)(f_i)$ is compact satisfies that $F(f_1\circ ...\circ f_n)$ is compact in $\M$. This collections contains the functors $(m\otimes \map(x,-))\otimes \id_\N$ for any compact $m$, where the number $n$ is simply $1$. 

The previous lemma shows that this collection is closed under retracts (and thus contains each $(m\otimes\map(x_i,-))\otimes \id_\N, m\in\M^\omega$), and \Cref{prop:cpctcofib} shows that it is closed under cofibers (where, for $\mathrm{cofib}(F_0\to F_1)$, the corresponding number is $n_0+n_1$). As it is clearly closed under desuspensions, in total it contains the thick subcategory generated by the $(m\otimes\map(x_i,-))\otimes \id_\N,m\in\M^\omega$, and thus by smoothness it contains $\id_\M\otimes\id_\N = \id_{\M\otimes \N}$. The result follows. 

We note that the number $n$ is independent of $\N$. 
\end{proof}
\begin{rmk}\label{rmk:relsmooth}
   It is not hard to see that the same holds when replacing $\Sp$ with a rigidly-compactly generated $\V\in\CAlg(\PrL_{\st})$, so we're allowed to consider also relatively smooth algebras. 
   \end{rmk}
\begin{ex}
    We record an example to show that the $n$ is in general $>1$. Let $\M=~\Mod_{\Sph[t]}$, which is smooth. We view objects therein as spectra equipped with endomorphisms; and let $f: (\Sph,\id)\to (\Sigma \Sph,\id)$ be the map given by the $0$ on underlying spectra map, and an interesting homotopy $h$ in the square: \[\begin{tikzcd}
	\Sph & \Sigma\Sph \\
	\Sph & {\Sigma \Sph}
	\arrow["0"', from=2-1, to=2-2]
	\arrow["\id", from=1-2, to=2-2]
	\arrow["\id"', from=1-1, to=2-1]
	\arrow["0", from=1-1, to=1-2]
	\arrow["h"{description}, draw=none, from=1-1, to=2-2]
\end{tikzcd}\]
    
    In this case, the underlying map of $f$ is $0$, and so the underlying map of $\bigoplus_\mathbb N f$ is also $0$ and hence compact. However, $\bigoplus_\mathbb N f$ itself is not compact: if it were, it would factor through some finite direct sum $\bigoplus_F \Sigma\Sph$, but it does not. 
\end{ex}



\begin{rmk}
We do not know whether there exist non-compactly generated smooth categories. In fact, as mentioned in the introduction, we do not even know in general whether there exist non-compactly generated invertible categories! See \cite{Stefanich} for related work. 
\end{rmk}

\section{Comonadicity and presentability}\label{section:comonad}
The goal of this section is to prove our main theorem, namely: 
\begin{thm}\label{thm:DblIsPrL}
    Let $\V\in\CAlg(\PrL)$. The \category{} $\Dbl{\V}$ is presentable. 

    If $\V\in\CAlg(\PrL_\kappa)$, where $\kappa$ is an uncountable regular cardinal, then $\Dbl{\V}$ is $\kappa$-compactly generated. 
\end{thm}

The general strategy is as follows: firstly, we prove that if $\kappa$ is an uncountable regular cardinal such that $\V\in\CAlg(\PrL_\kappa)$ (such a $\kappa$ always exists), then $\Dbl{\V}$ is a (non-full) subcategory of $\Mod_\V(\PrL_\kappa)$. 

Note that this is saying strictly more than only ``each dualizable $\V$-module is $\kappa$-compactly generated'' (which we have proved in \Cref{cor:cardboundV}) - we will unwind what this concretely means when we come to it. As a second step, we prove that the forgetful functor $\Dbl{\V}\to \Mod_\V(\PrL_\kappa)$ (which therefore exists) is \emph{comonadic}, with an explicit, and accessible comonad. As $\PrL_\kappa$ is itself presentably symmetric monoidal, $\Mod_\V(\PrL_\kappa)$ is presentable, and thus, so is $\Dbl{\V}$. 

We now outline a few consequences of this theorem. 
\begin{cor}
    Let $\V\in\CAlg(\PrL)$. The \category{} $\Dbl{\V}$ admits limits, and is closed symmetric monoidal. 
\end{cor}
\begin{proof}
    Both statements follow from the general theory of presentable \categories. For the latter, we simply note that $\Dbl{\V}\to \Mod_\V(\PrL)$ preserves colimits and is strong symmetric monoidal, so that the tensor product is compatible with colimits in the former. 
\end{proof}
These two facts can be established without knowing the presentability of $\Dbl{\V}$\footnote{In fact, they were established by Efimov this way. See also \Cref{section:limitsalwaysexist}.}, but ultimately the same ingredients go into their proof as in the proof of presentability. One fact which does not seem to be easily obtainable without presentability is:
\begin{cor}
    Let $(C,\tau)$ be a (small) $\infty$-site. The inclusion $\Sh_\tau(C;\Dbl{\V})\subset~\Psh(C;\Dbl{\V})$ has a left adjoint, namely sheafification. 
\end{cor}
\subsection{Cardinal bound}
In this subsection, we reach the first goal on our way to proving \Cref{thm:DblIsPrL}, namely:
\begin{thm}\label{thm:CardBound}
   Let $\V\in\CAlg(\PrL_\kappa)$, for an uncountable regular cardinal $\kappa$. As (non-full) subcategories of $\Mod_\V(\PrL)$, we have an inclusion $\Dbl{\V} \subset \Mod_\V(\PrL_\kappa)$. 
\end{thm}
Note that $\PrL_\kappa\subset \PrL$ is a non-full inclusion. The same is true for $\Mod_\V(\PrL_\kappa)\subset~\Mod_\V(\PrL)$: the restriction on morphisms is the same, namely if $\M,\N\in \Mod_\V(\PrL_\kappa)$, then a morphism $\M\to \N$ in $\Mod_\V(\PrL)$ is in $\Mod_\V(\PrL_\kappa)$ if and only if its underlying morphism in $\PrL$ is in $\PrL_\kappa$. 

However, on objects the restriction is more than simply objectwise. Indeed, the structure maps making $\M$ into a $\V$-module must also be in $\PrL_\kappa$, specifically, the multiplication map $\V\otimes\M\to \M$ must be in $\PrL_\kappa$. This is similar to how ``$\V\in\CAlg(\PrL_\kappa)$'' says both that $\V$ is $\kappa$-compactly generated, and that tensor products of $\kappa$-compact objects are $\kappa$-compact (including the nullary tensor product, i.e. the unit of $\V$). 
\begin{lm}
    Let $f:\M\to \N$ be a morphism in  $\Dbl{\V}$. It preserves $\lambda$-compacts for all cardinals $\lambda$.  
\end{lm}
\begin{proof}
    This is because by assumption, the right adjoint of $f$, $f^R$, preserves all colimits, including the $\lambda$-filtered ones. 
\end{proof}
\begin{lm}\label{lm:structuremapcardbound}
Let $\V\in\CAlg(\PrL_\kappa), \M\in\Dbl{\V}$. The structure map $\V\otimes\M\to \M$ preserves $\kappa$-compacts. 
\end{lm}
\begin{proof}
    By \Cref{thm:Luriethmgeneralbase}, there is a fully faithful internal left adjoint $\M\to \N$ where $\N$ is atomically generated over $\V$. By \Cref{cor:atimpliescpct}, $\N$ is $\kappa$-compactly generated and the structure map $\V\otimes\N\to \N$ preserves $\kappa$-compacts. The inclusion $\M\to \N$ is fully faithful and colimit-preserving, so it reflects $\kappa$-compacts, but it is also an internal left adjoint, so it preserves $\kappa$-compacts. 

    It follows by a small diagram chase that $\V\otimes\M\to \M$ preserves $\kappa$-compacts. 
\end{proof}
\begin{proof}[Proof of \Cref{thm:CardBound}]
    Combine \Cref{cor:cardboundV} and \Cref{lm:structuremapcardbound}. 
\end{proof}
\subsection{Comonadicity}
In this section, we conclude the proof of \Cref{thm:DblIsPrL} by proving:
\begin{thm}\label{thm:Comonadic}
    Let $\V\in\CAlg(\PrL_\kappa)$ for some uncountable regular cardinal $\kappa$. The inclusion $\Dbl{\V}\subset \Mod_\V(\PrL_\kappa)$ from \Cref{thm:CardBound} is comonadic, and the comonad $T: \Mod_\V(\PrL_\kappa)\to \Mod_\V(\PrL_\kappa)$ is $\kappa$-accessible. 

    In fact, this comonad is equivalent to the comonad associated with the left adjoint inclusion $\At{\V}\subset \Mod_\V(\PrL_\kappa) $, where $\At{\V}$ is the full subcategory of $\Dbl{\V}$ spanned by atomically generated $\V$-modules\footnote{Equivalently, the subcategory of $\Mod_\V(\PrL)$ spanned by atomically generated $\V$-modules and atomic-preserving functors.}. 
\end{thm}
With this theorem, we can prove our main result, using the following general fact: 

\begin{prop}\label{prop:comodacc}
    Let $\E$ be a presentable $\infty$-category and $T$ an accessible comonad. In this case, the $\infty$-category of $T$-comodules in $\E$ is presentable. 
\end{prop}
\begin{proof}
    For general reasons, the category of comodules is cocomplete (and the colimits are computed underlying, i.e. in $\E$), so this is really an issue of accessibility. 

    But the $\infty$-category of comodules can be expressed as a (partially lax) limit of accessible $\infty$-categories along accessible functors, as $T$ is assumed to be accessible, and hence, is accessible by \cite[5.4.7.3., 5.4.7.11.]{HTT}.
\end{proof}

It it is true generally that if $\kappa$ is an uncountable regular cardinal, $T$ is $\kappa$-accessible and $\E$ is $\kappa$-compactly generated, then so is the category of $T$-comodules. However the proof in this generality is a bit involved so we defer it, and only prove it in our special case. It turns out that in this case, the following simpler statement will be sufficient. We use the following definition:
\begin{defn}\label{defn:coalg}
    Let $T$ be an endofunctor of a category $D$. The category $\coAlg_T(D)$ of coalgebras for $T$ is the lax equalizer of $\id_D$ and $T$ as in \cite[Definition II.1.4]{NS}, that is, the category of pairs $(x, f:x\to Tx)$, or more formally the pullback 
    \[\begin{tikzcd}
	{\coAlg_T(D)} & {D^{\Delta^1}} \\
	D & {D\times D}
	\arrow[from=1-1, to=1-2]
	\arrow[from=1-1, to=2-1]
	\arrow["{(s,t)}", from=1-2, to=2-2]
	\arrow["{(\id_D,T)}"', from=2-1, to=2-2]
\end{tikzcd}\]
\end{defn}
\begin{prop}\label{prop:coalgacc}
    Let $\E$ be a $\kappa$-compactly generated presentable $\infty$-category and $T$ a $\kappa$-filtered colimit preserving endofunctor. If $\kappa$ is uncountable, then the $\infty$-category of $T$-coalgebras in $\E$ is $\kappa$-accessible, and the forgetful functor $\coAlg_T(\E)\to \E$ preserves and reflects $\kappa$-compacts.
\end{prop}
\begin{proof}
    It is clear that $\coAlg_T(\E)\to \E$ reflects $\kappa$-compacts, that is, if $x$ is an object which is underlying $\kappa$-compact, then it is $\kappa$-compact in $\coAlg_T(\E)$. So it suffices to produce enough underlying $\kappa$-compact objects in $\coAlg_T(\E)$. 

    Let $x\to Tx$ be an object therein and write $x\simeq \colim_{d\in \E^\kappa/x}d$ in $\E$. Note that for each $d\in \E^\kappa_{/x}$, the map $d\to x\to Tx$ factors through $Td'$ for some other $d'\in\E^\kappa_{/x}$. Iterating, we find, for each $d\in\E^\kappa_{/x}$, a sequence of $d_i$'s in $\E^\kappa_{/x}$ with maps $d\to d_0, d_i\to d_{i+1}$ so that each $d_i\to x\to Tx$ factors through $Td_{i+1}$. Taking the sequential colimit gives us a map $d_\infty=\colim_n d_n\to \colim_n Td_n\to T(d_\infty)$, and so we have a $T$-coalgebra $d_\infty$ with a map to $x$, and because $\kappa$ is uncountable, $d_\infty$ is still $\kappa$-compact. 

    From there on, the same proof as in \Cref{prop:smalllim} applies to prove the result. 
\end{proof}
\begin{rmk}
    The proof is extremely similar to that of \Cref{prop:smalllim}, but I was not able to make the two proofs fit in a common framework. In particular, note that here there is a some asymmetry between $\id$ and $T$, and the analogous result is \emph{wrong} for $T$-algebras, that is, objects $x$ with a map $Tx\to x$. 
\end{rmk}

\begin{proof}[Proof of \Cref{thm:DblIsPrL}]
By \Cref{thm:Comonadic}, $\Dbl{\V}\simeq \coMod_T(\Mod_\V(\PrL_\kappa))$ is the \category{} of comodules for an accessible comonad $T$ on a presentable \category. By \Cref{prop:comodacc}, it is therefore presentable. 

We now prove the precise accessibility rank. We observe that in our case $\Dbl{\V}= \coMod_T(\Mod_\V(\PrL_\kappa))$ is in fact a \emph{retract} of the simpler category of pointed coalgebras for $T$, that is, the pullback 
\[\begin{tikzcd}
	{\coAlg^{\mathrm{ptd}}_T(\Mod_\V(\PrL_\kappa))} & {\Mod_\V(\PrL_\kappa)} \\
	{\coAlg_T(\Mod_\V(\PrL_\kappa))} & {\coAlg_{\id}(\Mod_\V(\PrL_\kappa))}
	\arrow[from=1-1, to=1-2]
	\arrow[from=1-1, to=2-1]
	\arrow[from=1-2, to=2-2]
	\arrow[from=2-1, to=2-2]
\end{tikzcd}\]
where $\coAlg$ was introduced in \Cref{defn:coalg}, the lower horizontal functor is induced by the counit $T\to\id$, and the right vertical functor is the functor $C\mapsto (C\xrightarrow{\id_C}C)$.

To prove that it is a retract, we first observe that there is an obvious forgetful functor $\Dbl{\V}\simeq \coMod_T(\Mod_\V(\PrL_\kappa))\to \coMod_T^{\mathrm{ptd}}(\Mod_\V(\PrL_\kappa))$. In the other direction, we note that the forgetful functor $\coAlg_T^{\mathrm{ptd}}(\Mod_\V(\PrL_\kappa))\to \Mod_\V(\PrL_\kappa)$ actually factors through $\Dbl{\V}$\footnote{Recall that this is a \emph{property} for a functor.}: indeed, an object in $\coAlg_T^{\mathrm{ptd}}(\Mod_\V(\PrL_\kappa))$ is a $C$ together with a splitting of $p:\PP_\V(C^\kappa)\to C$, so that $C$ is a retract of an atomically generated $\V$-module and so dualizable by \Cref{thm:Luriethmgeneralbase}, and furthermore, a map between two such objects is a retract of an atomic preserving map between atomically generated $\V$-modules, and is therefore itself an internal left adjoint by \Cref{cor:retractfunctor}. 

This proves the retraction claim, and so combining \Cref{prop:coalgacc},  \Cref{prop:smalllim} and \Cref{cor:cardboundcompass}, we may conclude that $\Dbl{\V}$ is $\kappa$-compactly generated. 
\end{proof}
\begin{cor}
    Let $\V\in\CAlg(\PrL_\kappa)$ for some uncountable regular cardinal $\kappa$. The inclusion $\Dbl{\V}\subset \Mod_\V(\PrL_\kappa)$ from \Cref{thm:CardBound} admits a right adjoint, and the induced comonad on $\Mod_\V(\PrL_\kappa)$ is $\kappa$-accessible.
\end{cor}
\begin{proof}
The existence of the right adjoint is \Cref{thm:nonstandadj}, where we also describe explicitly the comonad.

Namely, it is $\N\mapsto \PP_\V(\N^\kappa)$. As $\PP_\V$ is a (partial) left adjoint (see \cite{Shay}) it preserves arbitrary colimits as a functor $\Cat(\V)\to \Mod_\V(\PrL)$. Second, it is clear that $\N\mapsto \N^\kappa$ is $\kappa$-accessible too, so the composite comonad is accessible, as desired. 
\end{proof}
\begin{rmk}
    Note that the comonad is the same as the one for the adjunction $\At{\V}\leftrightarrows~\Mod_\V(\PrL_\kappa)$. In particular, this provides a completely categorical description of $\Dbl{\V}$, using no monoidal structure. In the proof of \Cref{thm:Comonadic}, there is one key instance where the proof fails for $\At{\V}$ and works for $\Dbl{\V}$. 

    We also note that this gives a second proof of the accessibility of the comonad, if one is more confident about $\At{\V}$ being presentable - classically, the latter is equivalent to the full subcategory of $\Cat_\V$ spanned by $\V$-categories that admit absolute $\V$-colimits, and it is overwhelmingly likely that a description like this also holds for \categories. From this, it would be clear that it is presentable. However, it seems to be a challenging task to prove ahead of time that $\At{\V}$ is presentable, as the theory of absolute $\V$-colimits is harder to develop in \category{} theory. For this reason, we include later this fact as a \emph{corollary} of our proof. 
\end{rmk}

\begin{cor}
    The functor $\Dbl{\V}\to \Mod_\V(\PrL_\kappa)$ preserves and reflects $\kappa$-compacts.
\end{cor}
\begin{proof}
    That it reflects $\kappa$-compactness follows directly from comonadicity (as $\kappa\geq \omega_1$).

    That it preserves them follows from the fact that the comonad on $\Mod_\V(\PrL_\kappa)$ (and hence the right adjoint) is $\kappa$-accessible. 
\end{proof}

We are finally ready to prove the comonadicity:
\begin{proof}[Proof of \Cref{thm:Comonadic}]
    We apply the (dual of the) Barr-Beck-Lurie theorem \cite[Theorem 4.7.3.5.]{HA}, so it suffices to prove that the inclusion $i:\Dbl{\V}\subset \Mod_\V(\PrL_\kappa)$ is conservative, and preserves limits of $i$-split cosimplicial diagrams. 
    
The former is trivial, as $\Dbl{\V}$ is a subcategory of $\Mod_\V(\PrL)$, so let us focus on the latter: let $\M^\bullet : \Delta\to \Dbl{\V}$ be a cosimplicial diagram which is split \emph{when viewed as a diagram in $\Mod_\V(\PrL_\kappa)$}; and let $i\M^\bullet: \Delta_{-\infty}\to \Mod_\V(\PrL_\kappa)$ denote a choice of splitting (where we use the notation of \cite[Notation 4.7.2.1.]{HA})\footnote{We warn the reader that this notation is abusive, since on $\Delta_{-\infty}$, this diagram is not $i$ applied to a diagram with values in $\Dbl{\V}$. }. 

First, observe that $i\M^{-1}$ is a retract of $i\M^0$, from which it follows that $i\M^{-1}$ is dualizable\footnote{This is the key step in the proof where we cannot replace the word ``dualizable'' by the word ``atomically generated''. }, so it is, in fact, $i(\M^{-1})$.

Now note that the restriction of $i\M^\bullet$ to $\Delta_+$ is an absolute limit diagam (as it extends to a split cosimplicial diagram), so that $iP(iM^\bullet)$ is also a limit diagram, where $P:\Mod_\V(\PrL_\kappa)\to \Dbl{\V}$ is the right adjoint from \Cref{thm:nonstandadj}. If we restrict further to $\Delta$, we note that then we have an adjunction of diagrams $i\M^\bullet\rightleftarrows iPi\M^\bullet$ by \Cref{prop:internalleftBC} and \cite[Theorem 4.6]{Runemonad}. We therefore get an induced adjunction on their limits, $i\M^{-1}\rightleftarrows iPi\M^\bullet$. 

As the right adjoint $\colim: iPi\M^\bullet \to i\M^\bullet$ is natural on $\Delta_{-\infty}$, we know that the induced functor on limits is $\colim: iPi\M^{-1}\to i\M^{-1}$, with its left adjoint $\hat y$. But the previous argument shows that this left adjoint commutes with the morphisms $\M^{-1}\to \M^n$ for all $n$. More precisely, letting $\iota_n : \M^{-1}\to \M^n$ be the canonical projection, for all $n$ there exist commutative squares: 
\[\begin{tikzcd}
	{\M^{-1}} & {\M^n} \\
	{Pi\M^{-1}} & {Pi\M^n}
	\arrow["\iota_n", from=1-1, to=1-2]
	\arrow["{P(\iota_n)}"', from=2-1, to=2-2]
	\arrow["{\hat y}"', from=1-1, to=2-1]
	\arrow["{\hat y}", from=1-2, to=2-2]
\end{tikzcd}\]

By the ``In fact'' part of \Cref{prop:internalleftBC}, the \emph{existence} of such a square guarantees that $\iota_n$ is an internal left adjoint.

It follows that the limit diagram $i\M^\bullet: \Delta_+\to \Mod_\V(\PrL_\kappa)$ actually lifts (as a diagram) to $\Dbl{\V}$. We let $\M^\bullet :\Delta_+\to \Dbl{\V}$ denote the corresponding cone, and are now left with the task of proving that this is a limit cone \emph{in} $\Dbl{\V}$. Namely given another cone $\N\to \M^\bullet$ on $\Delta$ in $\Dbl{\V}$, we get an induced map $i\N\to i\M^{-1}$ in $\Mod_\V(\PrL_\kappa)$, and have to prove that this map is an internal left adjoint.

As before, it suffices to check that the canonical natural transformation in the following square is an equivalence: 
\[\begin{tikzcd}
	\N & {\M^{-1}} \\
	Pi\N & {Pi\M^{-1}}
	\arrow["{\hat y}"', from=1-1, to=2-1]
	\arrow["{\hat y}", from=1-2, to=2-2]
	\arrow["{P(f)}"', from=2-1, to=2-2]
	\arrow["f", from=1-1, to=1-2]
	\arrow[shorten <=4pt, shorten >=4pt, Rightarrow, from=1-2, to=2-1]
\end{tikzcd}\]
However, we already noted that $iPi\M^\bullet :\Delta_+\to \Mod_\V(\PrL_\kappa)$ is a limit diagram, so that the map $iPi\M^{-1}\to iPi\M^0$ is conservative\footnote{The limit is computed underlying, as $\kappa$ is uncountable, cf. \Cref{prop:smalllim}.}, and thus we may check this after postcomposing with $Pi\M^{-1}\to Pi\M^0$: 
\[\begin{tikzcd}
	\N & {\M^{-1}} & {\M^0} \\
	Pi\N & {Pi\M^{-1}} & {Pi\M^0}
	\arrow["{\hat y}"', from=1-1, to=2-1]
	\arrow["{\hat y}", from=1-2, to=2-2]
	\arrow["{P(f)}"', from=2-1, to=2-2]
	\arrow["f", from=1-1, to=1-2]
	\arrow[shorten <=4pt, shorten >=4pt, Rightarrow, from=1-2, to=2-1]
	\arrow["{\iota_0}", from=1-2, to=1-3]
	\arrow["{P(\iota_0)}"', from=2-2, to=2-3]
	\arrow["{\hat y}", from=1-3, to=2-3]
	\arrow[shorten <=5pt, shorten >=5pt, Rightarrow, from=1-3, to=2-2]
\end{tikzcd}\]
We have already argued that $\iota_0$ was an internal left adjoint, so that the rightmost lax commutative square actually commutes, and thus, by conservativity, the leftmost one does if and only if the outer rectangle does too. But now this follows from the fact that our cocone was originally in $\Dbl{\V}$, and hence the map $\N\to \M^0$ was in fact an internal left adjoint, which implies the desired claim. 
\end{proof}
\begin{cor}
    The full subcategory $\At{\V}$ of $\Dbl{\V}$ spanned by atomically generated $\V$-modules is presentable. 
\end{cor}
\begin{proof}
    The fully faithful inclusion $\At{\V}\subset \Dbl{\V}$ has a right adjoint given by $\M\mapsto \PP_\V(\M^\at)$, so it suffices to prove that this right adjoint is accessible.

Let $\kappa$ be such that $\V\in\CAlg(\PrL_\kappa)$, and let $v\in\V^\kappa$. Let $C(v)$ denote the $\V$-enriched \category{} with two objects $a,b$ such that $\hom(a,a) = \hom(b,b)= \one_\V$ and $\hom(a,b) = v$. 

For each $v\in\V^\kappa, \PP_\V(C(v)) \in\Dbl{\V}$ is $\lambda_v$-compact for some regular $\lambda_v$, and since $\V^\kappa$ is small, we may take a uniform $\lambda$ for all $v\in\V^\kappa$, and we may also assume that $\V\in(\Dbl{\V})^\lambda$. We will prove that $\M\mapsto \PP_\V(\M^\at)$ is $\lambda$-accessible.

Given now a $\lambda$-filtered diagram $\M_\bullet: I\to \Dbl{\V}$, we look at the map $\colim_I \PP_\V(\M_i^\at)\to~\colim_I~\M_i$. We need to prove that $\colim_I \M_i^\at\to~\colim_I~\M_i$ is $\V$-fully faithful - that $\V\in(\Dbl{\V})^\lambda$ guarantees that it is, in fact, essentially surjective when corestricted to atomics.  

But now let $x,y\in \M_i^\at$, and $v\in\V^\kappa$, we have $$\Map(v,\hom_{\colim_I \M_i}(x,y)) = \Map_\V(C(v), \colim_I \M_i) \times_{\Map_\V(\{\one_\V\}, \colim_I\M_i)\times\Map_\V(\{\one_\V\}, \colim_I\M_i)} \{x,y\}  $$
and then, using \Cref{thm: UPVPsh} together with \Cref{cor:atomicimpliesinternal} we find that this is equivalently $$\Map_{\Dbl{\V}}(\PP_\V(C(v)), \colim_I \M_i) \times_{\Map_{\Dbl{\V}}(\V, \colim_I\M_i)\times\Map_{\Dbl{\V}}(\V, \colim_I\M_i)} \{x,y\}$$ and therefore, since $I$ is $\lambda$-filtered and $\PP_\V(C(v)), \V$ are both $\lambda$-compact, this is equivalent to $$\colim_I \Map_{\Dbl{\V}}(\PP_\V(C(v)), \M_i) \times_{\Map_{\Dbl{\V}}(\V, \M_i)\times\Map_{\Dbl{\V}}(\V, \M_i)} \{x,y\}$$
and thus to $$\colim_I \hom_{\M_i}(x,y)$$

Thus in total, we do have $\colim_I\M_i^\at\simeq (\colim_I \M_i)^\at$, as was to be shown. 
\end{proof}

\section{Limits and internal homs}\label{section:limits}
The goal of this section is to give explicit constructions of limits and internal homs in $\Prdbl$. It follows from \Cref{thm:DblIsPrL} that they exist, but we give here explicit constructions - we note that the proof that these constructions work does \emph{not} depend on the presentability result, and as we mentioned before, this is how these limits and internal homs were originally constructed by Efimov. 

Before constructing them in full generality by introducing a way of constructing dualizable categories, we start with a number of ``elementary'' examples, that is, examples where the naive limit is correct. 
\subsection{Elementary examples}
In this first subsection, we deal with elementary examples where limits in $\Prdbl$ behave as one would expect naively. 

The first observation is the following - recall that $(\PrL_{\st})^{iL}$ is the wide subcategory of $\PrL_{\st}$ whose morphisms are internal left adjoints:
\begin{lm}\label{lm:finitelimobj}
    Let $I$ be a finite \category, and let $\M_\bullet: I\to (\PrL_{\st})^{iL}$ be a diagram, and let $\lim_I\M_i$ denote its limit \emph{in $\PrL_{\st}$}, and let $f_\bullet:\N\to \M_\bullet$ be a cocone in $(\PrL_{\st})^{iL}$. The induced map $\N\to \lim_I\M_i$ is an internal left adjoint. 
\end{lm}
\begin{rmk}
    The more general version of this claim in $\V$-modules would involve so-called $\V$-absolute limits, rather than finite limits. 
\end{rmk}
\begin{rmk}
    We are not claiming that the projection maps $\lim_I\M_i\to \M_j$ are internal left adjoints, in particular it need not be the case that $\lim_I\M_i$ is a limit in $(\PrL_{\st})^{iL}$. 
\end{rmk}
We will soon give a more general version of this argument, but for now we stick to this one. 
\begin{proof}
Let $f_i^R$ denote the (colimit-preserving) right adjoint to $f_i:\N\to\M_i$, and $p_i:~\lim_I\M_i\to~\M_i$ denote the $i$th projection.

We use \cite[Theorem B]{AsafLior}: there is a diagram $X: I\to \Fun(\lim_I\M_i, \N)$ whose value at $i$ is $f_i^R\circ p_i$ and such that the right adjoint to $\N\to \lim_I\M_i$ is given by $\lim_IX_i$ - more informally, this right adjoint is $\lim_I f_i^R\circ p_i$. 

As $I$ is a finite \category{}, this limit preserves filtered colimits, and our assumption guarantees that each $f_i^R\circ p_i$ preserves filtered colimits, and so we are done. 
\end{proof}

Based on the above argument, we see: 
\begin{lm}
    Let $I$ be an \category, and let $\M_\bullet: I\to (\PrL_{\st})^{iL}$ be a diagram, and let $\lim_I\M_i$ denote its limit \emph{in $\PrL_{\st}$}, and let $f_\bullet:\N\to \M_\bullet$ be a cocone in $(\PrL_{\st})^{iL}$. Let $i\mapsto f_i^Rp_i$ denote the diagram $I\to \Fun(\lim_I\M_i,\N)$ constructed in \cite[Theorem B]{AsafLior} 

Suppose that the for each $m\in\lim_I\M_i$, the limit $\lim_I f_i^Rp_i(m)$ is a ``stably absolute'' limit diagram in $\N$, i.e. it is preserved by any exact functor $\N\to C$.
    
    În this case induced map $\N\to \lim_I\M_i$ is an internal left adjoint. 
\end{lm}
\begin{proof}
    It suffices to prove that $\lim_I f_i^Rp_i: \lim_I\M_i\to \N$ preserves filtered colimits, and thus it suffices to prove that it preserves infinite coproducts. 

    By a retraction argument, it further suffices to prove that it preserves infinite coproducts of a single object $m\in\lim_I\M_i$. But now for any indexing set $S$, $\bigoplus_S:\N\to \N$ is an exact functor, and so by assumption it preserves $\lim_If_i^R p_i(m)$. Since $f_i^Rp_i$ also preserves $\bigoplus_S$, it follows in total that the canonical map $\bigoplus_S \lim_If_i^Rp_i(m)\to \lim_If_i^Rp_i(\bigoplus_Sm)$ is an equivalence, as was to be shown. 
\end{proof}

Of course, finite \categories{} are absolute limit diagrams in the stable case, so the previous lemma is indeed a special case of this lemma, but there will be other examples. 
\begin{cor}\label{cor:naivelim}
    Let $I$ be an \category{}, and let $\M_\bullet:I\to \Prdbl$ be a diagam, and let $\lim_I\M_i$ denote its limit in $\PrL_{\st}$. For this to be a limit diagram in $\Prdbl$, it suffices that:
    \begin{enumerate}
        \item $\lim_I\M_i$ is dualizable;
        \item the projection maps $p_i:\lim_I\M_i\to \M_i$ are internal left adjoints; 
        \item For each $m\in\lim_I\M_i$, the limit $\lim_Ip_i^Rp_i(m)$ is absolute for exact functors.  
    \end{enumerate}
    Point 3. is in particular satisfied as soon as $I$ is a finite \category{} (or a retract thereof). 

    Furthermore, it suffices to check point 2. for a weakly initial set of vertices $I_0\subset \pi_0(I^\simeq)$. 
\end{cor}
\begin{ques}
    Is this an if and only if ? Clearly, points 1. and 2. are necessary, so the question is about point 3..
\end{ques}
\subsubsection{Pullbacks along localizations}
\begin{lm}\label{lm:telpullback}
Let $\C_i\to\C_{01}$ be compact-preserving localizations of compactly generated stable categories, $i\in\{0,1\}$. 

Suppose further that $\ker(\C_i\to \C_{01})$ is compactly generated, for both $i=0,1$. 

Then $\C_0\times_{\C_{01}}\C_1$ is compactly-generated by $\C_0^\omega\times_{\C_{01}^\omega}\C_1^\omega$. 
\end{lm}
\begin{proof}
Because the two localizations preserve compact objects, we can form $\D := \C_0^\omega\times_{\C_{01}^\omega}\C_1^\omega$, and it is easy to check that $\D\subset (\C_0\times_{\C_{01}}\C_1)^\omega$ because pullbacks commute with filtered colimits. 

Therefore, by a classical lemma for cocomplete stable categories (see e.g. \Cref{lm:kappacpctgen}), it suffices to prove that for any nonzero $x\in \C_0\times_{\C_{01}}\C_1$, there exists $y\in \D$ and a nonzero map $y\to x$, to conclude that $\Ind(\D)\simeq \C_0\times_{\C_{01}}\C_1$. 

So let $x= (x_0,x_1, \gamma : p_0(x_0)\simeq p_1(x_1))$, where we let $p_i : \C_i\to \C_{01}$ denote our localization. If $x$ is nonzero, at least one of $x_0,x_1$ is nonzero. By symmetry of the hypotheses, we might as well assume $x_0\neq 0$. Therefore there exists $y_0\in \C_0^\omega$ with a nonzero map $y_0\to x_0$. 

Then $p_0(y_0)\in \C_{01}^\omega$ is, at least up to a retract, in the image of $\C_1^\omega$. In fact we can choose this retract: $p_0(y_0)\oplus\Sigma p_0(y_0)\simeq p_0(y_0\oplus\Sigma y_0)$ is in the image. $y_0\oplus\Sigma y_0$ of course is also compact and has a nonzero map to $x_0$, so we may replace $y_0$ by a different one so that $p_0(y_0)\simeq p_1(y_1)$ for some $y_1\in \C_1^\omega$, and this is what we do. 

Now we have a map $y_0\to x_0$, which induces a map $p_1(y_1)\simeq p_0(y_0)\to p_0(x_0)\simeq p_1(x_1)$. Because $y_0\to x_0$ is nonzero, any map $(y_0,y_1)\to (x_0,x_1)$ in the pullback that lifts this is also nonzero. So it will suffice to lift $p_1(y_1)\to p_1(x_1)$ to some map $y_1\to x_1$. 

This we cannot do in general (notice that so far we haven't used the telescope conjecture !). In fact, there is a description of mapping spectra in localizations as $\map(p_1(y_1),p_1(x_1))\simeq \colim_{\tilde y_1\to y_1}\map(\tilde y_1,x_1)$ where the colimit is indexed over maps $\tilde y_1\to y_1$ with fiber in $\ker(\C_1\to \C_{01})$. 

In particular, there is a diagram in $\C_1$ of the form : $$\xymatrix{\tilde y_1 \ar[r] \ar[d] & x_1 \\
y_1}$$ 
such that upon applying $p_1$, the vertical map becomes an equivalence, and going up-right gives exactly our original map $p_1(y_1)\to p_1(x_1)$. 

Now, letting $F$ denote the fiber of $\tilde y_1\to y_1$, we can also write $\tilde y_1 = \mathrm{fib}(y_1\to \Sigma F)$. Because of the telescope conjecture for $\C_1\to \C_{01}$, we can write $F$ as a filtered colimit of compact objects that are still in the kernel. Because $y_1$ is compact, the map $y_1\to\Sigma F$ factors through one of those. Passing to fibers, we get a bigger diagram :

 $$\xymatrix{y_1'\ar[r] \ar[rd] & \tilde y_1 \ar[r] \ar[d] & x_1 \\
& y_1}$$ 
where both vertical arrows induce equivalences upon applying $p_1$ (and in particular $y_1'\to \tilde y_1$ does too), but now $\tilde y_1'$ is compact too. 

In other words, up to replacing $y_1$ with $y_1'$, we have found a lift $y_1'\to x_1$ of $p_1(y_1')\simeq~p_1(y_1)\to~p_1(x_1)$. 

We therefore have a map $(y_0,y_1')\to (x_0,x_1)$ lifting our nonzero map $y_0\to x_0$, and this is therefore nonzero, and of course $(y_0,y_1')\in \D$. This is what we wanted. 
\end{proof}

\begin{cor}
    Let $\C_0\to \C_{01}\leftarrow \C_1$ be two localizations in $\Prdbl$. The forgetful functor $\Prdbl\to \PrL_{\st}$ commutes with their pullback. 
\end{cor}
\begin{proof}
We note first that $I$ is finite, so to apply \Cref{cor:naivelim}, it suffices to check that the pullback in $\PrL_{\st}$ is dualizable and that the projections to each $\C_i$ are internal left adjoints. 

Let $K_i:=\ker(\C_i\to \C_{01})$. For $\kappa$-large enough, we still have a localization sequence $K_i^\kappa\to~\C_i^\kappa\to~\C_{01}^\kappa$ by \Cref{prop:kappaloc} ($\kappa=\omega_1$ is large enough here). Thus $\Ind(\C_i^\kappa)\to~\Ind(\C_{01}^\kappa)$ is a compact-preserving localization between compactly generated categories with compactly generated kernel. 

Picking $\kappa$ large enough we can make this true for both $i=0,1$. Now the cospan $\C_0\to~\C_{01}\leftarrow~\C_1$ is a retract in $\PrL_{\st}$ of the cospan $\Ind(\C_0^\kappa)\to \Ind(\C_{01}^\kappa)\leftarrow\Ind(\C_1^\kappa)$, and thus its pullback is a retract of the pullback of the latter. 

By \Cref{lm:telpullback}, the pullback of the latter is compactly generated, and hence our pullback is dualizable. Furthermore, one easily checks that the map $$\C_0\times_{\C_{01}}\C_1\to \Ind(\C_0^\kappa)\times_{\Ind(\C_{01}^\kappa)}\Ind(\C_1^\kappa)$$ is left adjoint to the map $$\Ind(\C_0^\kappa)\times_{\Ind(\C_{01}^\kappa)}\Ind(\C_1^\kappa)\to \C_0\times_{\C_{01}}\C_1$$ and in particular, the former is an internal left adjoint. 

Now, $\Ind(\C_0^\kappa\times_{\C_{01}^\kappa}\C_1^\kappa)\to \Ind(\C_i^\kappa)$ preserves compacts. Combining these last two statements with a bit of diagram chasing shows that $\C_0\times_{\C_{01}}\C_1\to \C_i$ is an internal left adjoint for $i=0,1$, as was to be shown. 
\end{proof}
As a special case of this, we note: 
\begin{cor}
    Let $X$ be a qcqs scheme, and let $U,V$ be qc opens covering $X$. In this case, the following square is a pullback diagram in $\Prdbl$: 
    \[\begin{tikzcd}
	{\QCoh(X)} & {\QCoh(V)} \\
	{\QCoh(U)} & {\QCoh(U\cap V)}
	\arrow[from=1-1, to=2-1]
	\arrow[from=2-1, to=2-2]
	\arrow[from=1-2, to=2-2]
	\arrow[from=1-1, to=1-2]
\end{tikzcd}\]
In other words, $\QCoh(-): \mathrm{Sch}_{qcqs}\to \Prdbl$ satisfies Zariski descent. 
\end{cor}
\begin{cor}
    Let $X$ be a locally compact Hausdorff topological space, and let $K\cup L=X$ be a closed cover of $X$. The following square is a pullback diagram in $\Prdbl$: 
    \[\begin{tikzcd}
	{\Sh(X)} & {\Sh(L)} \\
	{\Sh(K)} & {\Sh(K\cap L)}
	\arrow[from=1-1, to=2-1]
	\arrow[from=2-1, to=2-2]
	\arrow[from=1-2, to=2-2]
	\arrow[from=1-1, to=1-2]
\end{tikzcd}\]
\end{cor}
\begin{rmk}
    One can further deduce that $\Sh(-): \mathcal K(X)\op\to \Prdbl$ is a $\mathcal K$-sheaf in the sense of \cite[Definition 7.3.4.1.]{HTT}, and not only in $\PrL_{\st}$. One can deduce from this together with \Cref{prop:cpctasshyper} that $\Prdbl$, while $\omega_1$-compactly generated is not itself compactly assembled. 

    Indeed, if it were, then we could run the argument from \Cref{rmk:hypercomplete} in this setting too. 
\end{rmk}
We give another example of a pullback that behaves particularly well, with potential applications to continuous $K$-theory. We learned of this example from Sasha Efimov:
\begin{prop}
    Let $p:\M\to \N \in \PrL_{\st,\omega}$ be a localization. In that case the pullback $\M\times_\N\M$ is compactly generated. 
\end{prop}
\begin{rmk}
    Unlike in \Cref{lm:telpullback}, we do not assume that $\ker(p)$ is compactly generated. Here, what will save us is that we have the same $\M$ on both sides. 
\end{rmk}
\begin{proof}
Let $(m,m',p(m)\simeq p(m'))$ be in the pullback. 

There exists a cofiber sequence $k\to x\to m$ with $k\in\ker(p)$ and such that the map $p(m)\simeq p(m')$ is the composite $p(m)\simeq p(x)\simeq p(m')$ for some map $x\to m'$. We thus find a map $(x,m', p(x)\simeq p(m'))\to (m,m',p(m)\simeq p(m'))$ with fiber $(k,0, p(k)=0)$, and a map $(x,m',p(x)\simeq p(m'))\to (m',m',p(m')= p(m'))$ (the latter since $p(x)\simeq p(m')$ is $p$ of a map in $\M$) with fiber $(\fib(x\to m'),0)$. 

Since $p(x\to m')$ is an equivalence, it follows that $\fib(x\to m')$ is in $\ker(p)$, like $k$, in particular we are reduced to proving that elements of the form $(k,0), k\in\ker(p)$ are generated under colimits by $\M^\omega\times_{\N^\omega}\M^\omega$. 

Following the proof of \Cref{prop:telescope}, we may assume without loss of generality that $k$ is a sequential colimit $\colim_\mathbb N k_n$ where $k_n\in \M^\omega$, and the map $p(k_n\to k_{n+1})$ is null. We then claim that $(k,0)$ is the sequential colimit of $(k_n,k_n, p(k_n)= p(k_n))$, where the transition maps are given by $(k_n,k_n)\xrightarrow{(i_n,0)}(k_{n+1},k_{n+1})$ together with the nullhomotopy of $p(k_n\to k_{n+1})$. 
\end{proof}
\begin{ex}
    Let $\M\in\Prdbl$. By \Cref{lm:calkcpct}, $\Ind(\M^{\omega_1})\to \Calk_{\omega_1}(\M)$ is a compact preserving localization between compactly generated categories, so the above proves that $\Ind(\M^{\omega_1})\times_{\Calk_{\omega_1}(\M)}\Ind(\M^{\omega_1})$ is compactly generated. 

    Furthermore, it is clear that for any localizing invariant $E$, $$E_{cont}(\M)\simeq E(\M^{\omega_1}\times_{\Calk_{\omega_1}(\M)^\omega}\M^{\omega_1})$$.  In principle, this lets one access $\pi_0(K_{cont}(\M))$, as we understand $K_0$ of small stable \categories. 
\end{ex}
\subsubsection{Descent}
The generality of \Cref{cor:naivelim} allows for more general limits, for example we also obtain Galois descent. More generally and precisely:
\begin{cor}
   Let $\V\in\CAlg(\PrL_{\st})$, and let $A\in\CAlg(\V)$ be a descendable algebra in the sense of \cite[Definition 3.18]{Akhil}. Finally, let $\M$ be a $\V$-module which is dualizable over $\Sp$. In this case, the limit cone $\M\simeq \lim_\Delta \Mod_{A^{\otimes n}}(\M)$ in $\PrL_{\st}$ from \cite[Proposition 3.22]{Akhil} is a limit cone in $\Prdbl$ too. 

   In particular, if $A$ is $G$-Galois for some group $G$, we may rewrite this as $\M\simeq \Mod_A(\M)^{hG}$.  
\end{cor}
\begin{proof}
    Conditions 1. and 2. from \Cref{cor:naivelim} are obviously satisfied, so we are left with 3.. For this one, we simply refer to \cite[Proposition 3.20]{Akhil} (and note that the class of limit diagram which are absolute for exact functors is preserved under tensoring, cf. also the proof of \cite[Proposition 3.22]{Akhil}). 

    The ``In particular'' part follows from general considerations, cf.\cite[Section 6.5]{MNN1}.
\end{proof}
This result also applies to the limit decompositions from \cite[Section 6.5]{MNN1}, \cite[Corollary 9.16]{Akhil}. 
\begin{cor}
    $\QCoh(-):\mathrm{Sch}_{qcqs}\to \Prdbl$ satisfies Galois descent. 
\end{cor}

\subsection{Constructing compactly assembled categories}\label{section:S-constr}
In this subsection, we review a general construction of compactly assembled \categories, due to Dustin Clausen, which we will use in the next subsection to describe limits and internal homs in $\Prdbl$.

Given $\M\in \PrL_{\st}$ and a collection $S$ of maps in $\M$, we will introduce conditions on $S$ so that there exists a universal dualizable category mapping to $\M$ where compact maps are all sent to $S$; more precisely, for there to exist a dualizable category $(\M,S)$ equipped with a functor $(\M,S)\to \M$ sending compact maps to $S$ and such that for any $\N\in\Prdbl$, this functor induces a fully faithful functor $$\Fun^{iL}(\N,(\M,S))\to \Fun^L(\N,\M)$$ with essential image those functors $\N\to \M$ sending compact maps to maps in $S$. 

This construction will help to describe limits and internal homs in $\Prdbl$ and can be used in some cases to describe them explicitly. We will give a simple example of this phenomenon in \Cref{section:prod}. 

The conditions on $S$ are all fairly simple, except for one key axiom related to the work in \Cref{section:smooth}, we will explain where this axiom comes in - for most of the results, it is irrelevant. 
\begin{rmk}
    Fix $\M\in\PrL_{\st}$ and an arbitrary collection $S$ of maps (which all factor through some fixed small subcategory of $\M$), closed under composition with all maps in $\M$ on either side. The functor $\N\mapsto \Map^L_S(\N,\M)$ sending $\N$ to the full subgroupoid of $\Fun^L(\N,\M)$ on the functors sending compact maps to $S$ (has small values and) \emph{a priori} sends colimits in $\Prdbl$ to limits in spaces. By \Cref{thm:DblIsPrL} and the adjoint functor theorem, it follows that it is representable by some $(\M,S)$. Our goal in this section is to describe reasonable conditions on $S$ under which $(\M,S)$ is actually describable explicitly. 
\end{rmk}
We will begin the construction and start with no assumptions on $S$ except the most basic one for set theoretic reasons:
\begin{assu}\label{ass:settheory}
    Let $S$ be a nonempty collection of maps in $\M$, closed under composition with any maps in $\M$ on either side, and such that there exists a small subcategory $\M_0\subset \M$ such that for any $s\in S$, there exists $m\in \M_0$ such that $s$ factors through $m$. 

    We say $S$ is an accessible ideal. 
\end{assu}
\begin{rmk}
    One can always assume $\M_0= \M^\kappa$ for some cardinal $\kappa$. 
\end{rmk}
\begin{defn}
    Let $(\M,S)_{basic}\subset \Ind(\M)$ denote the full subcategory spanned by objects of the form $\colim_{r\in\mathbb Q_{\geq 0}} y(m_r)$, where each map $m_r\to m_s, r<s$ is in $S$. 
\end{defn}
\begin{obs}
    \Cref{ass:settheory} guarantees (using a cofinality argument) that $(\M,S)_{basic}\subset~\Ind(\M^\kappa)$ for some cardinal $\kappa$, and in fact that it is small.  
\end{obs}
Based on the above set-theoretic observation, the following definition is sound:
\begin{defn}
    Let $(\M,S)\subset \Ind(\M)$ denote the smallest stable subcategory closed under colimits and containing $(\M,S)_{basic}$.

This is a presentable stable category.
\end{defn}
We first aim to prove:
\begin{prop}\label{prop:Dustindbl}
    If $S$ is an accessible ideal closed under (de)suspensions, then $(\M,S)$ is dualizable. 
\end{prop}
We have a number of lemmas leading up to it. Throughout, $S$ is an accessible ideal. 
\begin{lm}\label{lm:basiccpct}
    Let $m\in(\M,S), n\in(\M,S)_{basic}$, where $n$ is written as $\colim_{\mathbb Q_{\geq 0}}y(n_r)$. Any map $m\to n$ factoring through some $y(n_r)$ in $\Ind(\M)$ is compact in $(\M,S)$.
\end{lm}
\begin{proof}
It is compact in $\Ind(\M)$ (it factors through a compact!) and thus, as $(\M,S)\subset~\Ind(\M)$ is a full subcategory closed under colimits, it is compact in $(\M,S)$.
    
\end{proof}
\begin{lm}\label{lm:basicexhaust}
    Let $m\in (\M,S)_{basic}$ - it is compactly exhaustible in $(\M,S)$, in fact in $(\M,S)_{basic}$. 
\end{lm}
\begin{proof}
    Say $m=\colim_{\mathbb Q_{\geq 0}}y(m_r)$. By a cofinality argument, $m=\colim_{n\in\NN} \colim_{r\in]n,n+1[} y(m_r)$, and now each $\colim_{r\in ]n,n+1[}y(m_r)$ is in $(\M,S)_{basic}$ as $]n,n+1[\cong \mathbb Q_{\geq 0}$ by Cantor's theorem, and each of the transition maps $\colim_{r\in]n,n+1[} y(m_r)\to \colim_{r\in]n+1,n+2[} y(m_r)$ factors through $y(m_{n+1})$ in $\Ind(\M)$ and is thus compact in $(\M,S)$. 
\end{proof}
\begin{rmk}\label{rmk:cpctMS}
    The proof of this lemma also shows that \Cref{lm:basiccpct} admits a converse. In particular, compact maps in $(\M,S)$ between objects of $(\M,S)_{basic}$ are also compact in $\Ind(\M)$. 
\end{rmk}

\begin{lm}
    If $S$ is closed under (de)suspensions, then for any nonzero $x\in (\M,S)$ there exists $m,n\in (\M,S)_{basic}$, and maps $m\to n\to x$ such that the composite is nonzero, and $m\to n$ is compact. 
\end{lm}
\begin{proof}
    By definition of $(\M,S)$, if $x$ is nonzero then $\map(c,x)\neq 0$ for some $c\in (\M,S)_{basic}$. Up to (de)suspending $c$ enough, we may assume $\pi_0\map(c,x)\neq 0$. Writing $c=\colim_{\mathbb Q_{geq 0}}y(m_r)$ with transition maps in $S$, and using $c\simeq \colim_\NN y(m_r)$, and finally the Milnor short exact sequence, we find that either $\{\pi_0\map(m_r,x)\}$ or $\{\pi_1\map(m_r,x)\}$ is not $0$ as a pro-system, and thus, up to suspending if needed, there is some $r$ and some nonzero map $y(m_r)\to x$ factoring through $y(m_{r+1})$. 

    But now we may write the map $y(m_r)\to y(m_s)$ as the composite $$y(m_r)\to \colim_{]r,r+\frac{1}{2}[}y(m_t)\to \colim_{]r+\frac{1}{2},r+\frac{3}{4}[}y(m_t)\to y(m_{r+1})$$ where now the middle two terms are in $(\M,S)_{basic}$ and the middle map is compact, so that we get the desired result.
\end{proof}
\begin{proof}[Proof of \Cref{prop:Dustindbl}]
    Putting together the preceding lemmas, we find that we can simply apply \Cref{thm:dblviacpctmaps}. 
\end{proof}
\begin{assu}
    From now on, $S$ is assumed to be closed under (de)suspensions. 

    A stable accessible ideal is an accessible ideal closed under (de)suspensions. 
\end{assu}
\begin{cor}\label{cor:UPSoneside}
    Let $S$ be a stable accessible ideal in $\M$, and $\N\in\Prdbl$. If $f:\N\to \M$ is colimit-preserving and sends compact maps in $\N$ to $S$, then the corresponding internal left adjoint $\N\to \Ind(\M)$ from \Cref{thm:nonstandadj} (cf. also \Cref{rmk:nonstandadjInd}) factors through $(\M,S)$. 
\end{cor}
\begin{proof}
    Following the proof of \Cref{thm:nonstandadj}, we see that this internal left adjoint is given by $f\circ \hat y_\N$. 

    As $\N$ is dualizable, it suffices to prove that it sends compactly exhaustible objects in $\N$ to $(\M,S)$, so let $n= \colim_{\mathbb Q_{\geq 0}}n_r$ be a compactly exhaustible object (cf. \Cref{cor:cpctexhaustQ}). We have, by \Cref{lm:basicnuc}, $\hat y(n) =\colim_{\mathbb Q_{\geq 0}}y(n_r)$. 
    
    Each $n_r\to n_s$ is compact so its image $y(f(n_r))\to y(f(n_s))$ is in $S$, and so $f\circ \hat y(n)$ is in fact in $(\M,S)_{basic}$, as needed. 
\end{proof}
To obtain a converse (which is what we are looking for at the end of the day), we need some kind of ``robustness'' of $S$. Essentially, we need $(\M,S)_{basic}$ to consist exactly of $(\M,S)^{\omega_1}$, and this means essentially that it has to be closed under countable colimits. 
\begin{prop}\label{prop:UPSotherside}
If $S$ is a stable accessible ideal, and if $(\M,S)_{basic}\subset (\M,S)$ is closed under countable colimits, then for any dualizable $\N$ and any internal left adjoint $\N\to~(\M,S)$, the composite $\N\to (\M,S)\to \Ind(\M)\to \M$ sends compact maps to $S$. 
\end{prop}
\begin{proof}
    If it is closed under countable colimits, then because its elements are all $\omega_1$-compact in $(\M,S)$, we find that $(\M,S)^{\omega_1}= (\M,S)_{basic}$.

    In particular, any compact map in $(\M,S)$ factors through a compact map in $(\M,S)_{basic}$, and thus, by \Cref{rmk:cpctMS}, through a map between $m,n\in (\M,S)_{basic}$ of the form $m\to~y(m')\to~y(m'')\to~n$. 

    So now let $f:\N\to (\M,S)$ be an internal left adjoint, and $\alpha: x\to y$ be a compact map in $\N$. By internal left adjointness, $f(\alpha)$ is compact in $(\M,S)$ and thus factors, in $\Ind(\M)$, through $y(S)$. Thus $\colim\circ f(\alpha)$ factors through $S$, and thus lies in $S$. 
\end{proof}
\begin{lm}\label{lm:countablecoprod}
    Assume $S$ is closed under finite direct sums. In this case, $(\M,S)_{basic}$ is closed under countable coproducts. 
\end{lm}
\begin{proof}
    Let $(m^n_r): \NN\times\mathbb Q_{\geq 0}\to \M$ be a family of $\mathbb Q_{\geq 0}$-indexed diagrams such that for each $n\in\NN, r<s\in\mathbb Q_{\geq 0}$, the transition map $m^n_r\to m^n_s$ is in $S$. 

    In this case, $\bigoplus_\NN \colim_{\mathbb Q_{\geq 0}} y(m^n_r)\simeq \colim_{r\in\mathbb Q_{\geq 0}}\bigoplus_{n\leq r}y(m^n_r)$ by a cofinality argument, and for each $r<s$, $\bigoplus_{n\leq r}y(m^n_r)$ is a finite direct sum of maps that are either in $S$ or of the form $0\to y(m^n_s)$ (which are also in $S$ because $S$ is an ideal). This coproduct is thus also in $(\M,S)_{basic}$. 
\end{proof}
We are left with understanding when $(\M,S)_{basic}$ is closed under cofibers. This is where the key axiom comes in, mimicking the behaviour of \Cref{prop:cpctcofib}:
\begin{assu}\label{assu:cofib}
Consider a diagram of vertical cofiber sequences, as in \Cref{prop:cpctcofib}\footnote{Just as we did there, for notational simplicity, we have suppressed the nullhomotopies, but they are again part of the commutative diagram}: 
 \[\begin{tikzcd}
	{x_0} & {y_0} & {z_0} \\
	{x_1} & {y_1} & {z_1} \\
	{x_2} & {y_2} & {z_2}
	\arrow["{f_0}", color={rgb,255:red,214;green,92;blue,92}, from=1-1, to=1-2]
	\arrow["{g_0}", color={rgb,255:red,214;green,92;blue,92}, from=1-2, to=1-3]
	\arrow["{f_1}", color={rgb,255:red,214;green,92;blue,92}, from=2-1, to=2-2]
	\arrow["{g_1}", color={rgb,255:red,214;green,92;blue,92}, from=2-2, to=2-3]
	\arrow[from=2-1, to=3-1]
	\arrow[from=1-1, to=2-1]
	\arrow[from=1-2, to=2-2]
	\arrow[from=1-3, to=2-3]
	\arrow[from=2-2, to=3-2]
	\arrow[from=2-3, to=3-3]
	\arrow["{f_2}"', from=3-1, to=3-2]
	\arrow["{g_2}"', from=3-2, to=3-3]
\end{tikzcd}\]
where $f_0,f_1,g_0,g_1$ are in $S$. Then $g_2\circ f_2$ is also in $S$. 
\end{assu}
\begin{rmk}
    We could instead assume the simpler axiom that maps in $S$ are closed under cofiber sequences -- but this is too strong and not satisfied in the relevant examples, cf. \Cref{warn:cpctcofib}.  
\end{rmk}
\begin{thm}\label{thm:UPDustin}
    Suppose $S$ is a stable accessible ideal, closed under finite direct sums and satisfying \Cref{assu:cofib}.
    
    In that case, $(\M,S)_{basic}$ is closed under countable colimits and thus for any $\N\in~\Prdbl$, the functor $(\M,S)\to \M$ induces an equivalence: 
    $$\Fun^{iL}(\N,(\M,S))\to \Fun^L_S(\N,\M)$$
    where the right hand side is the full subcategory of $\Fun^L(\N,\M)$ spanned by those functors that send compact maps in $\N$ to maps in $S$. 
\end{thm}
\begin{proof}
    By \Cref{lm:countablecoprod}, to prove that $(\M,S)_{basic}$ is closed under countable colimits, it suffices to prove that it is closed under cofibers. 

    Let $x,z \in (\M,S)_{basic}$, and $f:x\to z$. We write $x=\colim_{\mathbb Q_{\geq 0}} y(x_r)$ and similarly for $z$ with $y(z_s)$. 

Using the fact that the transition maps in those systems are compact, and using the presentation of $\mathbb Q_{\geq 0}$ as $\colim_k G_k$ as in the proof of \Cref{lm:rationals}, we find that the map $f$ can be realized as the colimit of a map of diagrams $x_r\to z_r$. 

We now take the cofiber of this map of diagrams. As any nonidentity map in $\mathbb Q_{\geq 0}$ can be written as a composite of two nonidentity maps, it follows from the construction together with \Cref{assu:cofib} that any nonidentity map in this diagram is in $S$, and hence the cofiber really is in $(\M,S)_{basic}$.

    We may now conclude by combining \Cref{thm:nonstandadj}\footnote{And \Cref{rmk:nonstandadjInd}.} and \Cref{prop:Dustindbl},\Cref{cor:UPSoneside} and \Cref{prop:UPSotherside}. 
\end{proof}
\begin{ex}
    When $S$ is the collection of compact maps, all the axioms are satisfied by \Cref{ex:factorcompact}, \Cref{prop:cpctcofib} and elementary verifications regarding suspensions and finite direct sums. Furthermore, in this case, by \Cref{lm:basicnuc}, the map $(\M,S)\to \M$ is fully faithful: its essential image is the largest dualizable subcategory of $\M$ embedded through an internal left adjoint. 

    More generally, whenever $S$ is included in the collection of compact maps, the functor $(\M,S)\to \M$ is fully faithful.
\end{ex}

\begin{ex}
    When $\M=\V\in\CAlg(\PrL_{\st})$ and $S$ is the collection of trace-class maps, we prove in \cite[Section 4.3.1]{companion} that $(\V,S)$ is the ``rigidification'' $\Rig(\V)$ of $\V$.
\end{ex}

\subsection{Describing limits and internal homs}\label{section:limitsalwaysexist}

Combining \Cref{thm:UPDustin} with \Cref{cor:strongadjcompactass}, we automatically get the following corollaries:
\begin{cor}
    Let $\M_\bullet: I\to \Prdbl$ be a small diagram with limit \emph{in} $\PrL_{\st}$ given by $\lim_I\M_i$, and let $S$ be the collection of maps in $\lim_I\M_i$ that are sent to compact maps in each $\M_i$. The cone $(\lim_i \M_i,S)\to \lim_I\M_i\to\M_\bullet$ is a limit cone in $\Prdbl$. 
\end{cor}

\begin{cor}
    Let $\M,\N\in\Prdbl$, and let $S$ be the collection of maps $\eta:f\to g$ in $\Fun^L(\M,\N)$ such that for any compact map $m_0\to m_1$ in $\M$, the composite $f(m_0)\to~g(m_0)\to~g(m_1)$ is compact.  In this case the functor $\M\otimes (\Fun^L(\M,\N),S)\to\M\otimes\Fun^L(\M,\N)\to \N$ withnesses $(\Fun^L(\M,\N),S)$ as the internal hom of $\M$ and $\N$ in $\Prdbl$. 
\end{cor}
As an application, we deduce from this second corollary and from \Cref{thm:smooth} (see also \Cref{rmk:relsmooth}):
\begin{cor}
    Let $A$ be a smooth $k$-algebra for some commutative ring spectrum $k$, and $\M\in\Dbl{\Mod_k}$. The canonical map $\Hom^\dbl_k(\Mod_A,\M)\to \Fun^L_k(\Mod_A,\M)$ is fully faithful. 
\end{cor}
\begin{proof}
    Indeed, while not every underlying compact map is compact, there exists an integer $n$ such that every composite of $n$ underlying compact maps is compact by \Cref{thm:smooth}. This difference vanishes when looking at the $(\M,S)$-construction: if $\mathbb Q_{\geq 0}\to \M$ is a diagram in which every transition map is in $S$, then every transition map is the composite of $n$ maps in $S$ (because in $\mathbb Q_{\geq 0}$, every transition map is the composite of $n$ transition maps!). 
\end{proof}
In a similar way, one obtains: 
We leave the following corollary as an exercise: 
\begin{cor}\label{cor:finlimff}
     Let $I$ be a finite \category, and let $\M_\bullet: I\to \Prdbl$ be a diagram, and let $\lim_I\M_i$ denote its limit \emph{in $\PrL_{\st}$}, and $\lim_I^\dbl\M_i$ its limit in $\Prdbl$. The canonical map $\lim_I^\dbl\M_i\to \lim_I \M_i$ is fully faithful. 
\end{cor}
This implies in particular that monomorphisms in $\Prdbl$ are fully faithful, see e.g. \cite[Appendix A]{sashaI} for an explanation.  \subsubsection{Products}\label{section:prod}
The final thing we do here is describe a special case where the above construction of limits gives a concrete calculation (there are other more interesting ones, but they are beyond the scope of this paper).
\begin{prop}\label{prop:prodcpctgen}
    Let $\M_i\in\PrL_{\st,\omega}$ be an $I$-indexed family of compactly generated categories. Their product in $\Prdbl$ agrees with their product in $\PrL_{\st,\omega}$, i.e. it is given by $\Ind(\prod_I \M_i^\omega)$. 
\end{prop}
\begin{proof}
    In $\M_i$, any compact map factors through a compact. Hence, letting $S$ denote the collection of maps in $\prod_i \M_i$ which are pointwise compact, we see that any $\NN$-indexed diagram $(m_i^\bullet): \NN\to\prod_i\M_i$ with transition maps in $S$ is $\Ind$-equivalent to one where all the $m_i$'s are compact, i.e. its colimit in $\Ind(\prod_i\M_i)$ lies in $\Ind(\prod_i\M_i^\omega)$. 
\end{proof}

\appendix
\section{Around stable \categories}\label{app:stable}
In this appendix, we collect some standard facts about stable \categories{} and their localizations which can be found throughout the literature. We discuss two themes: localizations and ($\kappa$-)compact generation. 
\subsection{Localizations of stable \categories}
We begin with some definitions.
\begin{defn}
A functor $f:C\to D$ between \categories{} is called a localization if there is a class $W$ of arrows such that for any $E$, the restriction functor $$\Fun(D,E)\to \Fun(C,E)$$ is fully faithful and has essential image $\Fun_W(C,E)$, namely the full subcategory spanned by functors that send $W$ to equivalences. 
\end{defn}
\begin{defn}
    A functor $f:C\to D$ between \categories{} is called a Bousfield localization if it admits a fully faithful right adjoint, and it is called a Bousfield colocalization if it admits a fully faithful left adjoint. 
\end{defn}
The following is standard:
\begin{lm}\label{lm:BousDK}
    Let $f:C\to D$ be a Bousfield localization. In this case, $f$ is a localization. 
\end{lm}
If we want to stress that we are dealing with a localization which is not a Bousfield (co)localization, we will say it is a ``Dwyer-Kan'' localization. 

The following is perhaps less standard:
\begin{lm}\label{lm:DKBous}
    Let $f:C\to D$ be a Dwyer-Kan localization. If it admits a right adjoint, then this right adjoint is fully faithful, and in particular $f$ is a Bousfield localization. 
\end{lm}
\begin{proof}
    Let $g$ be the left adjoint. It suffices to prove that $g_! : \Fun(D\op,\Ss)\to \Fun(C\op,\Ss)$ is fully faithful. But now $f\dashv g$ implies that $g\op\dashv f\op$, which in turn implies $f^*\dashv g^*$, and so $f^* \simeq g_!$. But $f^*$ is clearly fully faithful, by definition of Dwyer-Kan localization, so we are done. 
\end{proof}
\begin{cor}\label{cor:leftrightff}
    Let $i\dashv \Gamma\dashv R$ be a chain of adjunctions, where $i:C\to D$. If $i$ is fully faithful, then so is $R$.
\end{cor}
\begin{proof}
    If $i$ is fully faithful, then $\Gamma$ is a Bousfield colocalization. By the dual of \Cref{lm:BousDK}, $\Gamma$ is a Dwyer-Kan localization, and by \Cref{lm:DKBous}, $R$ is fully faithful. 
\end{proof}
\begin{defn}
    An exact functor $f:A\to B$ between stable \categories{} is called a homological epimorphism if the induced functor $\Ind(f)$ is a (Bousfield) localization. 

    It is called a Karoubi projection if furthermore the kernel of $\Ind(f)$ is compactly generated, and a Verdier projection if it is both a Karoubi projection and essentially surjective. 

    We call it a left (resp. right) split projection if it admits a left (resp. right) adjoint\footnote{Note that a left/right split Karoubi projection is automatically Verdier}.
\end{defn}
The following is a special case of \Cref{cor:leftrightff}:
\begin{cor}
    Let $f: A\to B$ be a functor between stable \categories. If $f$ admits both a left and a right adjoint $f^L, f^R$, the $f^L$ is fully faithful if and only if $f^R$ is. 
\end{cor}
\begin{lm}\label{lm:fundamentalsequence}
Let $f: A\to B$ be a right split Verdier projection, with kernel $i: K\to A$. The inclusion $i$ admits a right adjoint $i^R$, and there is a canonical fiber sequence $$ii^R\to \id_A\to f^Rf$$

In particular, the pair $(i^R,f)$ is jointly conservative.
\end{lm}
\begin{proof}
    Let $j:A\to A$ denote the functor $\fib(\id_A\to f^R f)$. We note that because $ff^R\to\id_B$ is an equivalence, and by the triangle identity, $f\circ j =0$, so $j$ factors canonically as a $A\xrightarrow{\tilde j} K\xrightarrow{i} A$. The canonical map $i\tilde j\simeq j\to \id_A$ provides a co-unit, and we claim it is the co-unit of an adjunction : it suffices to check that for any $k\in K, a\in A$, the composite $$\map_K(k,\tilde j(a))\to \map_A(i(k),i(\tilde j(a)))\to \map_A(i(k),a)$$ is an equivalence. The first map is an equivalence as $i$ is fully faithful, so we are left with the second one. By definition, this second map is the fiber of $$\map_A(i(k),a)\to \map_A(i(k),f^Rf(a))\simeq \map_B(fi(k), f(a))=0$$ and so it is an equivalence. 
\end{proof}
Note that the space of nullhomotopies in this fiber sequence is (even objectwise!) contractible, because the mapping spectrum $\map(i(k),f^R(b))$ is trivial for any $k\in K, b\in B$. 
\begin{cor}\label{cor:recollementturn}
    Let $f: A\to B$ be a right split Verdier projection, with kernel $i: K\to A$. By the previous lemma, $i$ admits a right adjoint $i^R$. It is a left split Verdier projection with kernel $B\xrightarrow{f^R}A$. 
\end{cor}
\begin{proof}
    $i^R$ certainly admits a fully faithful left adjoint, so this is only a claim about the kernel. The fiber sequence from the previous lemma shows that $i^R(a)=0$ if and only if the unit $a\to f^R f(a)$ is an equivalence. But of course the map $f^R : B\to A$ induces an equivalence between $B$ and the full subcategory of such $a$'s. 
\end{proof}
\begin{lm}
Let $f: C\to D$ be an internal left adjoint between dualizable presentable stable \categories. It is a Bousfield localization if and only if $(f^R)^\vee:~\Fun^L(C,\Sp)\to~\Fun^L(D,\Sp)$ is one. 

The kernel of $(f^R)^\vee$ is $\Fun^L(\ker(f),\Sp)$. 
\end{lm}
\begin{proof}
This follows the fact that $(f^R)^\vee$ is left adjoint to $f^\vee$, with induced co/unit maps. Indeed, this implies the ``only if'' direction. 

Furthermore, because $C,D$ are dualizable, the above observation also implies that the equivalence $C\simeq \Fun^L(\Fun^L(C,\Sp),\Sp)$ turns $f\mapsto (f^R)^\vee$ into an involution, from which we get ``if''. 

For the claim about the kernel, we note that $(f^R)^\vee H = H\circ f^R$ and so this is $0$ if and only if $H$ factors through $\ker(f)$ via the right adjoint of the inclusion $\ker(f)\to C$, cf. \Cref{cor:recollementturn}. 
\end{proof}
\begin{cor}
    An exact functor $f:A\to B$ is a homological epimorphism (resp. Karoubi projection, resp. Verdier projection) if and only if $f\op$ is. 
\end{cor}
\begin{proof}
    This follows from the previous lemma, knowing that $\Fun^L(\Ind(A),\Sp)\simeq~\Fun^{ex}(A,\Sp)\simeq~\Ind(A\op)$ and that under this identification, $f\mapsto (f^R)^\vee$ corresponds to $\Ind(f)\mapsto \Ind(f\op)$. 
\end{proof}
\begin{cor}\label{cor:homepi}
    An exact functor $f:A\to B$ between small stable \categories{} is a homological epimorphism if and only if for all $b_0,b_1\in B$, the following canonical map is an equivalence:
    $$\colim_{A_{/b_1}}\map(b_0,f(-))\to \map(b_0,b_1)$$
\end{cor}
\begin{proof}
By the previous lemma, we can reduce to $f\op$, and thus to characterizing when $f_!:\Fun^{ex}(A,\Sp)\to\Fun^{ex}(B,\Sp)$ is Bousfield localization, i.e. when $f_!f^*\to \id_{\Fun(B,\Sp)}$ is an equivalence. Since $f_!, f^*$ both preserve colimits and $\Fun^{ex}(B,\Sp)$ is generated under filtered colimits by $\map(b_0,-), b_0\in B$, we are reduced to asking when $f_!\map_B(b_0,f(-))(b_1)\to~\map(b_0,b_1)$ is an equivalence for all $b_0,b_1\in B$. 

The claim then follows from the pointwise formula for left Kan extensions\footnote{There is a subtlety with left Kan extensions and exact functors. As the relevant colimits here are filtered, the statement for mapping spaces implies the corresponding statement for mapping spectra, so there is no issue.}
\end{proof}
\begin{cor}
    Homological epimorphisms are closed under arbitrary products.
\end{cor}
\begin{proof}
    Taking into account the AB6 axiom\footnote{Namely, the distributivity of products over filtered colimits. It holds in $\Set$, therefore in $\Ss$ by taking homotopy groups, and therefore in $\Sp$.}, the condition from \Cref{cor:homepi} is clearly closed under products - note that for $f:A\to B$ and $b\in B, A_{/b}$ admits finite colimits and is therefore filtered. 
\end{proof}
Another crucial corollary of \Cref{lm:fundamentalsequence} is:
\begin{cor}\label{cor:colimleftright}
    Let $f:A\to B$ be a right split Verdier projection with kernel $i:K\to A$. Assume $A$ admits $I$-shaped colimits. In this case, $B$ and $K$ also do; furthermore the right adjoint $f^R$ preserves $I$-shaped colimits if and only if the right adjoint $i^R$ does. 
\end{cor}
\begin{proof}
    $f$ is a left adjoint and hence preserves all colimits that exist in $A$. It follows that if $b_\bullet : I\to B$ is a diagram, then $b_\bullet \simeq ff^R(b_\bullet)\to f(\colim_I f^R(b_i))$ is a colimit diagram. 

    Similarly, if $k_\bullet : I\to K$ is a diagram, $f(\colim_I i(k_j))\simeq \colim_I fi(k_j) = 0$ so that $\colim_I i(k_j)\in K$. It follows that it is also a colimit in $K$, as $K$ is a full subcategory of $A$. This proves the existence of colimits in $B,K$, as well as their preservation by $i,f$. 

    Now we can use the fiber sequence from \Cref{lm:fundamentalsequence} to prove the claim : suppose $f^R$ preserves $I$-shaped colimits. In this case, so does $f^R f$ and therefore, by the fiber sequence $ii^R\to\id_A\to f^Rf$, so does $ii^R$. Since $i$ preserves them and is conservative, then $i^R$ preserves them too. The argument for the converse is similar. 
\end{proof}
\begin{cor}\label{cor:incstrongiffprojstrong}
    Let $f: C\to D$ be a Bousfield localization in $\PrL_{\st}$ with kernel $i:K\to C$. $f$ is an internal left adjoint if and only if $i$ is.
\end{cor}
\begin{proof}
    Apply the previous corollary to any small \category{} $I$. 
\end{proof}
We outline another consequence of the fiber sequence from \Cref{lm:fundamentalsequence}:
\begin{prop}\label{prop:adjVerdier}
    Consider a commutative diagram : 
    \[\begin{tikzcd}
	{K_0} & {A_0} & {B_0} \\
	{K_1} & {A_1} & {B_1}
	\arrow["f"', from=1-1, to=2-1]
	\arrow["g", from=1-2, to=2-2]
	\arrow["{i_0}", from=1-1, to=1-2]
	\arrow["{p_0}", from=1-2, to=1-3]
	\arrow["{i_1}"', from=2-1, to=2-2]
	\arrow["{p_1}"', from=2-2, to=2-3]
	\arrow["h", from=1-3, to=2-3]
\end{tikzcd}\] 
where $p_0,p_1$ are split Verdier projections with kernel inclusions $i_0, i_1$. In that case, the left square is horizontally right adjointable if and only if the right square is; and if $f,g,h$ have right adjoints, then the same holds for vertical right adjointability.
\end{prop}
\begin{proof}
    For the first part, we need to prove that the map $f(i_0)^R\to (i_1)^R g$ is an equivalence if and only if the map $g(p_0)^R\to (p_1)^Rh$. 

We note that by \Cref{lm:fundamentalsequence}, we have cofiber sequences: $$(i_1)(i_1)^R g\to g\to (p_1)^R (p_1) g$$ and $$g(i_0)(i_0)^R\to g\to g(p_0)^R(p_0)$$
We claim that the construction of the canonical maps makes it clear that they are compatible, namely : using the equivalences $g\circ i_0\simeq i_1\circ f$ and $p_1\circ g\simeq h\circ p_0$, they become $$i_1(i_1)^R g\to g\to (p_1)^R h p_0$$ and $$i_1 f  (i_0)^R\to g\to g(p_0)^Rp_0$$ respectively and unraveling the definition of the canonical maps, we find that we have a map of co/fiber sequences between the two. As the middle map is the identity $g\to g$, the claim follows: either map is an equivalence if and only if the other one is (note that $i_1$ is fully faithful, and that $p_0$ is a localization, so postcomposing with the former and precomposing with the latter does not change whether or not a transformation is an equivalence). 

For the second part, the argument is similar: we need to prove that $i_0\circ f^R\to g^R\circ i_1$ is an equivalence if and only if $p_0\circ g^R\to h^R\circ p_1$ is; and we prove that those maps fit into similar cofiber sequences from the cofiber sequences from \Cref{lm:fundamentalsequence}. We leave those details to the reader. 
\end{proof}
\begin{cor}\label{cor:kerproj}
    Let $f:A\to B$ be a homological epimorphism. $\ker(\Ind(f))$ is compactly generated if and only if it is compactly generated by $\ker(f)$, if and only if $A/\ker(f)\to B$ is fully faithful.
\end{cor}
\begin{proof}
    We note that $\Ind(f)$ is an internal left adjoint, thus by \Cref{cor:incstrongiffprojstrong} it follows that the inclusion $\ker\to\Ind(A)$ is also an internal left adjoint, so it preserves compacts. Thus it is compactly generated if and only if it is compactly generated by $\ker\cap \Ind(A)^\omega$. 

    If $A$ is idempotent-complete, we are done because $\Ind(A)^\omega= A$. Else, we note that if $k\in\ker\cap \Ind(A)^\omega$, then $k\oplus\Sigma k \in A$\footnote{Any $x\in\Ind(A)^\omega$ is a retract of some $x'\in A$. If $x\oplus y = x'$, then the cofiber of the projection onto $y$ followed by the inclusion of $y$ is $x\oplus\Sigma x$, which is therefore in $A$.}, and it is also in $\ker(\Ind(f))$, thus it is in $\ker(f)$, and $k$ is a retract thereof, so $\ker(f)\subset \ker(\Ind(f))\cap \Ind(A)^\omega$ is dense. 

    Now, if $\ker(\Ind(f))$ is compactly generated, we note that for $a_0,a_1\in A$, by \Cref{lm:fundamentalsequence} we have a cofiber sequence: 
    $$\colim_{k\in \ker(f)/ii^R(a_1)}\map_A(a_0, k) \simeq \map_{Ind(A)}(a_0, ii^R(a_1))\to \map_{Ind(A)}(a_0,a_1)\to \map_{\Ind(A)}(a_0,f^Rf(a_1))$$
    so that the canonical map $$\colim_{k\in\ker(f)/a_1}\map_A(a_0,\mathrm{cofib}(k\to a_1))\to \map_B(f(a_0),f(a_1))$$ is an equivalence. By \cite[Theorem I.3.3]{NS}, the source is $\map_{A/\ker(f)}(a_0,a_1)$. 

    The same argument gives the converse, which we leave as an exercise. 
\end{proof}
\begin{defn}
    A Karoubi (resp. Verdier) sequence is a sequence $K\to A\to B$ of exact functors between stable \categories{} where $K\to A$ is the kernel of $A\to B$, and $A\to B$ is a Karoubi (resp. Verdier) projection. Verdier sequences are also called localization sequences. 
\end{defn}
\begin{cor}
    Karoubi (resp. Verdier) projections are closed under arbitrary products.
\end{cor}
\begin{proof}
We use the previous result and again the AB6 axiom to conclude that $\prod_I A_i/\prod_I K_i\to~\prod_I(A_i/K_i)$ is fully faithful. 
\end{proof}
\begin{prop}\label{prop:loc=fibcofib}
    In $\Cat_{\st}$, Verdier sequences are precisely the squares :\[\begin{tikzcd}
	K & A \\
	0 & { B}
	\arrow[from=1-1, to=2-1]
	\arrow[from=1-1, to=1-2]
	\arrow[from=1-2, to=2-2]
	\arrow[from=2-1, to=2-2]
\end{tikzcd}\] 
which are both cartesian and cocartesian. 

The same holds for localization sequences in $\PrL_{\st}$, and for Karoubi sequences in the category of idempotent-complete stable categories.
\end{prop}
\begin{proof}
    We deal with $\PrL_{\st}$ - the case of idempotent-complete categories can be deduced from this one using $\Ind(-)$, and then the case of ordinary stable categories can be deduced from that one via an object-wise analysis. 

    In $\PrL_{\st}$, we note that limits are underlying, so this square being cartesian is equivalent to $i: K\to A$ being the kernel of $A\to B$ - in particular $i$ is fully faithful. Furthermore, this square being cocartesian is equivalent to the square of right adjoints being cartesian, so that $B\to A$ is the kernel of $i^R: A\to K$, and is therefore also fully faithful. It follows that $A\to B$ is a localization, with kernel $K$ - this is the definition of a localization sequence. 

    We leave the converse to the reader, it is straightforward. 
\end{proof}
\subsection{Compact generation}
We spend a bit of time on $\kappa$-compact generation. First, we recall the following classical consideration:
\begin{lm}\label{lm:kappacpctgen}
    Let $C$ be a cocomplete stable \category, and $\kappa$ a regular cardinal. Suppose $S$ is a set of $\kappa$-compact objects of $C$ such that for each nonzero $c\in C$, there exists a nonzero map $s\to c$ from some $s\in S$. 

    In this case, $C$ is presentable, and in fact $\kappa$-compactly generated. Its $\kappa$-compacts consist of the (idempotent completion of) the $\kappa$-small colimit closure of $S$ (and so $C$ is generated under colimits by $S$).
\end{lm}
\begin{proof}
    Let $S'$ denote the $\kappa$-small colimit closure of $S$. We first claim that $S'$ is small. Indeed we can construct it by transfinite induction, adding at each step the $\kappa$-small colimits of the previous step, and for some $\lambda$ large enough, any $\kappa$-small diagram in the $\lambda$th stage of this construction will factor through some $\mu$th stage, $\mu<\lambda$. Thus the $\lambda$th stage will be closed under $\kappa$-small colimits. 

    Second, we consider the map $\Ind_\kappa(S')\to C$ induced by the universal property of the $\Ind_\kappa$-completion. It is fully faithful, and it admits a right adjoint given by $c\mapsto \map_C(-,c)_{\mid S'}$. 

    Our assumption guarantees that this right adjoint is conservative, and hence the adjunction must be an equivalence, which proves everything we needed. 
\end{proof}
From the proof, it is clear that we have:
\begin{add}\label{add:kappacpctgen}
    In the situation from \Cref{lm:kappacpctgen}, it suffices to assume that every nonzero $c\in C$ receives a nonzero map from some $s$ which is in the colimit-completion of $S$.
\end{add}
The second property we want to record is the following:
\begin{prop}\label{prop:nonzerocolim}
    Let $C$ be a $\kappa$-compactly generated stable \category, and $x_\bullet : I\to C$ be a $\kappa$-filtered diagram with colimit $x$. 

    If for all $i$, $x_i\to x$ is nullhomotopic, then $x=0$.
\end{prop}
\begin{proof}
    Indeed, for any $\kappa$-compact $c$, and any $k\in\mathbb Z$, we have $\pi_k\map(c,x)\simeq \colim_I\pi_0\map(\Sigma^kc,y_i)$. If all the $x_i\to x$'s are null, this is a filtered colimit in abelian groups where all inclusion maps are $0$, and its colimit is thus $0$. It follows that for all $c\in C^\kappa$, $\map(c,x) = 0$, and thus $x=0$. 
\end{proof}
\begin{rmk}
    It is crucial that $C$ be $\kappa$-compactly generated for this result to hold. We construct a counterexample otherwise below.
\end{rmk}
\begin{ex}
    We first construct an example in a non-presentable case. We explain below how to modify it to make it presentable. The example is in $(\Mod_k)\op$ for a field\footnote{We thank Ishan Levy for pointing out this simplification.} $k$: let $V_n = k^{\oplus \mathbb N_{\geq n}}$, and let $V_{n+1}\to V_n$ be the canonical inclusion. In this case, in $\Mod_k$, $\lim_n V_n\simeq \Omega (\prod_\mathbb N k/\bigoplus_\mathbb N k)$ is in degree $-1$ and hence (because we are over a field!) all the relevant projection maps are $0$, but clearly the limit itself is not $0$.

    Given that example, we note that we can use the Yoneda embedding restricted to a full subcategory $C$ containing the limit diagram to get a counterexample in $\Fun'(C,\Sp)$, the full subcategory of $\Sp$ spanned by the reduced functors $C\to\Sp$ sending our limit diagram to a limit diagram. 
\end{ex}
The final property we want to record is a variation on \Cref{cor:kerproj}:
\begin{prop}\label{prop:kappaloc}
    Let $ \M\xrightarrow{p} \N$ be Bousfield localization in $\PrL_{\st}$ with kernel $i:\mathcal K\to \M$. Suppose $p$ preserves $\kappa$-compacts and $\M,\mathcal K$ are $\kappa$-compactly generated. 

    In this case, $\mathcal K^\kappa\to \M^\kappa\to \N^\kappa$ is a Karoubi sequence, and it is Verdier if $\kappa$ is uncountable.
\end{prop}
\begin{proof}
We need to prove first that $\M^\kappa/\mathcal K^\kappa\to \N^\kappa$ is fully faithful. 
    Let $m_0,m_1\in \M^\kappa$. We need to prove that the canonical map $$\colim_{f:m_1\to x,\fib(f)\in\mathcal K^\kappa}\map(m_0,x)\to \map(p(m_0),p(m_1))$$ is an equivalence, i.e. we need to prove that $$p^Rp(m_1)\simeq\colim_{z\in\mathcal K^\kappa, z\to m_1}\mathrm{cofib}(z\to m_1)$$

This is now clear since $p^Rp(m_1)\simeq \mathrm{cofib}(ii^R(m_1)\to m_1)$, and $\mathcal K$ is $\kappa$-compactly generated so that $ii^R(m_1)\simeq \colim_{z\to m_1, z\in\mathcal K^\kappa}z$.

It follows that $\M^\kappa/\mathcal K^\kappa\to \N^\kappa$ is fully faithful, so we now need to prove that it is essentially surjective. 

Given a map $p(x)\to p(y),x,y\in\M^\kappa$, the above argument shows that one can lift it to a map $x\to \tilde y, \tilde y\in \M^\kappa$. It follows from this that the image of $\M^\kappa$ is closed under pushouts (and hence finite colimits in general), under sequential colimits and $\kappa$-small copoducts. Combining these three, we find that the image is closed under $\kappa$-small colimits. 

In particular, $\Ind_\kappa(p(\M^\kappa))\to \Ind_\kappa(\N^\kappa)$ is fully faithful, and the source is closed under colimits, so that the map $\M\to \N$ factors through it - but the latter map is also surjective! Thus $p(\M^\kappa) = \N^\kappa$, up to retracts if $\kappa=\omega$, and literally if $\kappa$ is uncountable (because then $p(\M^\kappa)$ is closed under retracts, as these are given by countable colimits). 
\end{proof}
\section{Splitting invariants}\label{app:invariants}
\newcommand{\K}{\mathcal{K}}
In this appendix, we say a few words about splitting invariants\footnote{These are typically called ``additive invariants'' in the literature, e.g. in \cite{BGT}. We changed the name here to avoid confusing with the usual ``additivity'' of a functor.}, simply enough to be able to phrase and prove the Eilenber swindle, and to relate them to localizing invariants. 
\begin{nota}
We let $\Cat_{\st}$ denote the $\infty$-category of stable $\infty$-categories and exact functors between those. 
\end{nota}
\begin{defn}
Let $E: \Cat_{\st}\to\E$ be a product-preserving functor to a category with finite products. $E$ is said to be a \emph{splitting invariant} if for any cofiber sequence of exact functors $F,G,H : \C\to \D$, $F\to G\to H$, there exists some homotopy $E(G)\simeq E(F)+~E(H)$.
\end{defn}
\begin{rmk}
    In this definition, $E(F)+E(H)$ is defined as follows: we observe that $\Cat_{\st}$ is semi-additive, and so any product-preserving functor $E:\Cat_{\st}\to \E$ lifts canonically to $\CMon(\E)$, where $+$ is clearly defined. 
\end{rmk}

\begin{defn}
Let $\C$ be a stable $\infty$-category. We let $S_2\C \subset \Fun(\square,\C)$ denote the full subcategory spanned by those squares which are cofiber sequences, i.e. cocartesian squares where the bottom left term is a zero object. 

We define three functors $f,t,c: S_2\C\to \C$ that take a cofiber sequence $X\to Y\to Z$ to $X,Y,Z$ respectively \footnote{$f$ stands for ``fiber'', $t$ for ``total'' or ``target'', $c$ for ``cofiber''.}. Note that there is a canonical cofiber sequence of functors $f\to t\to c$ from $S_2\C$ to $\C$. 
\end{defn}
\begin{defn}
Let $\K\xrightarrow{i}\C\xrightarrow{p} \D$ be two exact functors. We say that they form a \emph{split exact sequence} if $p$ is a split Verdier projection with kernel $i:\K\to\C$. 
\end{defn}
\begin{ex}
A fundamental example of a split exact sequence is given by $\C\to S_2\C\to \C$, where the first functor sends $x$ to the cofiber sequence $x\to x\to 0$, and the second sends a cofiber sequence $x\to y\to z$ to $z$. The right adjoints are $(x\to y\to z)\mapsto x$ and $z\mapsto (0\to z\to z)$ respectively. 
\end{ex}
The following lemma is essentially due to Waldhausen
 \cite[Proposition 1.3.2]{Wald}, and proves that our definition of splitting invariant agrees with that of ``additive invariant'' in \cite{BGT}. 
\begin{lm}
Let $E: \Cat_{\st}\to \E$ be a product-preserving functor to a category with finite products. The following are equivalent: 
\begin{enumerate}
    \item $E$ is a splitting invariant; 
    \item For any stable $\infty$-category $\C$, and for the canonical cofiber sequence of functors $$f\to~t\to~c : S_2\C\to~\C,$$ $E(t)\simeq E(f)+E(c)$; 
    \item For any stable $\C$, $E$ applied to $S_2\C\xrightarrow{(f,c)} \C\times \C$ yields an equivalence; 
    \item For any split exact sequence $\K\xrightarrow{i}\C\xrightarrow{p}\D$, letting $r$ denote the right adjoint to $i$, $E$ applied to $\C\xrightarrow{(r,p)}\K\times\D$ yields an equivalence. 
\end{enumerate}
\end{lm}.
\begin{proof}

Note that 2. is a special case of 1., and 3. is a special case of 4.. We prove that 1. implies 4., and that 3. implies 2., and finally that 2. implies 1..\newline  

(1. implies 4.):  We observe that in such a split sequence $\K\xrightarrow{i} \C\xrightarrow{p} \D$, using $s$ to denote a right adjoint of $p$ and $r$ a right adjoint of $i$, we have, by \Cref{lm:fundamentalsequence}, a co/fiber sequence $$ir\to \id_\C\to sp$$ By 1., it follows that $E(\id_\C) \simeq E(i)E(r) + E(s)E(p)$.  Now, letting $(i,s) : \K\times\D\to \C$ denote $(c,d)\mapsto i(c)\oplus s(d)$, we find, using that $E$ preserves direct sums, that this means $E(\id_\E) \simeq E((i,s))\circ E((r,p))$. In the other direction, $E((r,p))\circ E((i,s)) \simeq E((r,p)\circ (i,s))$, and $(r,p)\circ (i,s) : (c,d)\mapsto (r(i(c)\oplus s(d)), p(i(c)\oplus s(d)))$ is equivalent to the identity, because $r\circ i$ is equivalent to the identity and $r\circ s \simeq 0$, and similarly, $p\circ s\simeq \id_\D$ and $p\circ i\simeq 0$. 

It follows that $E((i,s))$ and $E((r,p))$ are inverse to one another, thus proving 4.\newline 

(3. implies 2.): We know that $E((f,c))$ is an equivalence, but also that $i_0\oplus~i_1:~(c,c')\mapsto~(c\to~c\oplus~c'\to~c')$ is a section of $(f,c)$, so that it must be sent by $E$ to the inverse of $E((f,c))$. In particular, $$E(t) \simeq E(t)\circ E(\id_{S_2\C}) \simeq E(t) \circ E((i_0\oplus i_1))\circ E((f,c)) \simeq E(t)\circ (E(i_0)\oplus E(i_1))\circ E((f,c)) $$
$$\simeq (E(\id_\C)\oplus E(\id_\C))\circ E((f,c)) \simeq E(f) + E(c)$$  as desired. \newline 

(2. implies 1.): Consider a cofiber sequence $F\to G\to H$ of functors $\C\to \D$. It induces a single functor $S: \C\to S_2\D$ such that postcomposing it with the canonical cofiber sequence of functors $S_2\D\to \D$ yields our original cofiber sequence. In particular, $E(G) \simeq E(t)\circ E(S) $ and by 2. we get $E(t)\circ E(s)\simeq (E(f)+E(c))\circ E(s) \simeq  E(f\circ s)+ E(c\circ s)\simeq E(F)+E(H)$. 
\end{proof}
\begin{prop}
    Let $E:\Cat_{\st}\to \E$ be a splitting invariant. For any $C\in\Cat_{\st}$, the commutative monoid $E(C)$ is group-complete, and $E(\Sigma id_C)$ is an inverse.
\end{prop}
\begin{proof}
There is a fiber sequence $\id_C\to 0\to \Sigma \id_C$ which turns into $E(\id_C)+E(\Sigma \id_C)\simeq E(0) \simeq 0$. Using $E(\id_C)\simeq \id_{E(C)}$, we get the result. 
    \end{proof}
\begin{cor}
    For any category with finite products $\E$, the forgetful functor $\CGrp(\E)\to\E$ induces an equivalence on the categories of splitting invariants. 
\end{cor}
In particular, we may as well restrict our attention to additive categories.
\begin{prop}[The Eilenberg swindle]\label{prop:Eilenberg}
    Let $E:\Cat_{\st}\to \E$ be a splitting invariant, and let $C\in\Cat_{\st}$. If there exists a functor $f:C\to C$ with an equivalence $f\oplus \id_C\simeq f$, then $E(C)\simeq 0$. 

    Such a functor exists if $C$ admits countable coproducts or products.
\end{prop}
\begin{proof}
The first part comes from the equivalence $E(f)\simeq E(f)+E(\id_C)$ that we obtain. Together with group-completeness of $E(C)$, it follows that $E(\id_C)\simeq 0$ and so $E(C)\simeq 0$.

    The second part comes from the equivalence $\id_\C\oplus\bigoplus_{\mathbb N}\id_C\simeq \bigoplus_{\mathbb N}\id_C$\footnote{It sometimes comes in handy to observe that for $\bigoplus_{\mathbb N}\id_C$ to exist, we only need to know that $C$ admits countable coproducts of a single object. }.
\end{proof}
We finally discuss the relation to localizing invariants, as defined in \Cref{defn:locinv}. 
\begin{lm}
    Let $\E$ be an additive category with finite limits. Any localizing invariant $\Cat_{\st}\to \E$ is a splitting invariant.
\end{lm}
\begin{proof}
Let $E$ be a localizing invariant.

    Let $\K\to \C\to \D$ be a split Verdier sequence. It is in particular a localization sequence and so it is sent to a fiber sequence $E(\K)\to E(\C)\to E(\D)$ in $\E$. 

    Furthermore, the right adjoint to $\K\to\C$ provides a comparison between this split fiber sequence and the split fiber sequence $E(\K)\to E(\K\oplus\D)\to E(\D)$. Since $\E$ is additive, it follows that the map $E(\C)\to E(\K\oplus\D)$ is an equivalence. 
\end{proof}
\begin{rmk}
With values in an arbitrary $\E$ with finite products, there is no version of the above $3$-lemma, and one cannot conclude. In particular, the identity functor of $\Cat_{\st}$ sends localization sequences to fiber sequences, but it is not a splitting invariant. 
\end{rmk}

\section{Unstable Milnor sequences}\label{app:milnor}
The goal of this appendix is to remind the reader of some results from \cite[Chapter IX, §2, 3]{BK} and make them explicit for our purposes. Specifically, we re-introduce Bousfield and Kan's nonabelian $\lim^1$ functor, and establish a criterion for when a map of towers of spaces induces a $\pi_0$-isomorphism on the limits, which is relevant for the main body of the paper. 
\begin{lm}\label{lm:pi0pb}
    Let $X\to Y \leftarrow Z$ be a diagram of spaces, with pullback $P$. For every $p\in P$, the fiber of $\pi_0(P)\to \pi_0(X)\times \pi_0(Z)$ at the image of $p$ is isomorphic to a double coset $\pi_1(Z,p)\backslash\pi_1(Y,p)/\pi_1(X,p)$.   
\end{lm}
\begin{proof}
    Without loss of generality, we may assume $X,Y,Z$ are connected, and in this case we need to compute $\pi_0(P)$. Again, without loss of generality we may thus assume that $X,Y,Z$ are $BH,BG,BK$ for discrete groups $G,H,K$ because the induced map on pullbacks is an isomorphism on $\pi_0$. For discrete groups, the claim is obvious. 
\end{proof}
\begin{defn}
    Let $G_\bullet :\mathbb N\op\to \mathbf{Grp}$ be an inverse system of ordinary groups. We let $\lim^1_{non ab}G := \pi_0(\lim_\mathbb N BG_n)$. 
\end{defn}
Via the previous lemma, we can describe it as a double coset, reminiscent of the usual $\lim^1$. 
\begin{lm}
    Let $X_\bullet: \mathbb N\op\to \Ss$ be a diagram of spaces. In this case, for every $f\in~\pi_0(\lim_\mathbb N X_n)$, the fiber of $\pi_0(\lim_\mathbb N X_n)\to \lim_\mathbb N \pi_0(X_n)$ over the image of $f$ is given by $\lim^1_{nonab}\pi_1(X_n,f)$; in particular it only depends on the pro-isomorphism type of $\{\pi_1(X_n,f)\}$. 
\end{lm}
\begin{proof}
Writing the limit as a pullback $\prod_n X_n\times_{\prod_n X_n \times \prod_n X_n}\prod_n X_n$, and using the \Cref{lm:pi0pb}, we see that this fiber only depends on the $\pi_1$'s of the spaces involved. More precisely, mapping $X_n\to \tau_{\leq 1}X_n$, we see that the relevant fibers agree. 
\end{proof}
\begin{lm}\label{lm:unstableMilnor}
    Let $h_n: X_n\to  Y_n$ be a map of $\mathbb N\op$-indexed diagrams of spaces. Suppose it is a pro-$\pi_0$-isomorphism, and for every $f\in \lim_\mathbb N X_n$, a pro-$\pi_1$-isomorphism at $f$, namely $\pi_1(X_n,f_n)\to \pi_1(Y_n,h_n(f_n))$ is a pro-isomorphism. 

    In this case, $h: \lim X\to \lim Y$ is a $\pi_0$-isomorphism. 
    \end{lm}
    \begin{proof}
          Consider the following commutative square of sets: 
    \[\begin{tikzcd}
	{\pi_0(\lim_\mathbb NX_n)} & {\pi_0(\lim_\mathbb NY_n)} \\
	{\lim_\mathbb N\pi_0(X_n)} & {\lim_\mathbb N\pi_0(Y_n)}
	\arrow[from=2-1, to=2-2]
	\arrow[from=1-1, to=2-1]
	\arrow[from=1-2, to=2-2]
	\arrow[from=1-1, to=1-2]
\end{tikzcd}\]
Our assumptions and the previous lemmas guarantee that it is a square of sets of the following form:
\[\begin{tikzcd}
	S & T \\
	{S_0} & {T_0}
	\arrow["{f,\cong}", from=2-1, to=2-2]
	\arrow["p",two heads, from=1-1, to=2-1]
	\arrow["q",two heads, from=1-2, to=2-2]
	\arrow[from=1-1, to=1-2]
\end{tikzcd}\]
where the vertical maps are surjective, the bottom horizontal map is an isomorphism, and for any $s\in S$, the induced map on fibers $p^{-1}(p(s))\to q^{-1}(f(p(s))$ is an isomorphism. For any such square of sets, the map $S\to T$ is bijective - this is an easy exercise in set theory. 
    \end{proof}
    
\bibliographystyle{alpha}
\bibliography{Biblio.bib}
\end{document}